\documentclass[12pt]{amsart}

\usepackage[utf8]{inputenc}
\usepackage[english]{babel}
\usepackage[a4paper,twoside,top=1.2in, bottom=1.1in, left=0.8in, right=0.8in]{geometry}

\usepackage{graphicx}
\usepackage{caption} 

\usepackage{amsmath}
\usepackage{amsthm}
\usepackage{amssymb}
\usepackage{esint} 
\usepackage{mathrsfs} 
\usepackage{faktor} 
\usepackage{dsfont} 
\usepackage{xcolor}
\usepackage{array}
\usepackage{hhline}
\usepackage{enumerate}
\usepackage{enumitem} 
\usepackage{comment} 
\usepackage{imakeidx} 
\usepackage{xparse} 
\usepackage{mathtools}
\usepackage{fancyhdr} 
\usepackage{ifthen} 
\usepackage{forloop} 
\usepackage{xstring}
\usepackage{emptypage} 
\usepackage{minibox}

\usepackage[initials]{amsrefs} 

\usepackage{tikz}
\usetikzlibrary{arrows,shapes,patterns,calc,fadings,decorations.pathreplacing,cd}

\usepackage{hyperref} 
\hypersetup{
	colorlinks = true,
	linkcolor = {blue},
	urlcolor = {red},
	citecolor = {blue},
}

\newcounter{results}[section] 

\newcounter{steps}[section] 

\theoremstyle{plain}
\newtheorem{theorem}[results]{Theorem}

\newtheorem{lemma}[results]{Lemma}
\newtheorem{proposition}[results]{Proposition}
\newtheorem{corollary}[results]{Corollary}

\newtheorem{fact}[results]{Fact}

\newtheorem*{theorem*}{Theorem}
\newtheorem*{lemma*}{Lemma}
\newtheorem*{proposition*}{Proposition}
\newtheorem*{corollary*}{Corollary}
\newtheorem*{exercise*}{Exercise}
\newtheorem*{fact*}{Fact}

\theoremstyle{remark}
\newtheorem{remark}[results]{Remark}

\newtheorem*{remark*}{Remark}
\newtheorem*{question*}{Question}

\theoremstyle{definition}
\newtheorem{definition}[results]{Definition}

\newtheorem*{definition*}{Definition}
\newtheorem*{example*}{Example}

\numberwithin{equation}{section}

\newcommand{\Z}{\ensuremath{\mathbb Z}}
\newcommand{\R}{\ensuremath{\mathbb R}}

\newcommand \eps{\ensuremath{\varepsilon}}

\DeclareMathOperator{\tr}{tr}

\DeclareMathOperator{\supp}{supp}
\DeclareMathOperator{\sym}{Sym}
\DeclareMathOperator{\im}{Image}

\let\div\undefined
\newcommand{\div}{\ensuremath{\mathrm{div}}} 

\DeclareMathOperator{\Ric}{Ric} 
\DeclareMathOperator{\inj}{inj} 

\newcommand{\met}{\ensuremath{\mathscr R}} 
\newcommand{\metaf}{\ensuremath{\mathscr R}^{\text{AF}}}
\newcommand{\mM}{\ensuremath{\mathcal M}} 
\newcommand{\mK}{\ensuremath{\mathcal K}} 
\newcommand{\mL}{\ensuremath{\mathcal L}} 
\newcommand{\mR}{\ensuremath{\mathcal R}} 
\newcommand{\mS}{\ensuremath{\mathcal S}} 
\newcommand{\mC}{\ensuremath{\mathcal C}} 
\newcommand{\mO}{\ensuremath{\mathcal O}} 
\newcommand{\mU}{\ensuremath{\mathcal U}} 
\newcommand{\mt}{\ensuremath{\mathfrak t}}
\newcommand{\Fix}{\text{Fix}} 
\newcommand{\Diff}{\text{Diff}} 

\newcommand{\f}{\phi}
\newcommand{\e}{\varepsilon}
\newcommand{\be}{\begin{equation}}
\newcommand{\ee}{\end{equation}}

     \title[Constrained deformations of PSC metrics, II]{Constrained Deformations of Positive Scalar curvature metrics, II}
     \author{Alessandro Carlotto and Chao Li}
     \address{ \noindent Alessandro Carlotto: 
     	\newline ETH D-Math, R\"amistrasse 101, 8092 Z\"urich, Switzerland 
     	 \newline \textit{E-mail address: alessandro.carlotto@math.ethz.ch} 
     	 \newline \newline \indent Chao Li: 
     	 \newline Princeton University - Department of Mathematics, Fine Hall, 304 Washington Road, 08544 Princeton, United States of America
     	 \newline \textit{E-mail address: chaoli@math.princeton.edu} 
     	 \newline Courant Institute of Mathematical Sciences - New York University, 251 Mercer St, 10012 New York, United States of America
\newline \textit{E-mail address: chaoli@nyu.edu}
}
     
\begin{document}
     	
     	\begin{abstract}
     	We prove that various spaces of constrained positive scalar curvature metrics on compact 3-manifolds with boundary, when not empty, are contractible. The constraints we mostly focus on are given in terms of local conditions on the mean curvature of the boundary, and our treatment includes both the mean-convex and the minimal case. We then discuss the implications of these results on the topology of different subspaces of asymptotically flat initial data sets for the Einstein field equations in general relativity.
     	\end{abstract}

     	\maketitle      
     
     \tableofcontents
							
	\section{Introduction} \label{sec:intro}
	
	Let $X^n$ be a compact, connected and orientable $n$-dimensional manifold with boundary $\partial X$. 
The endowment of a Riemannian metric $g$ on $X$ allows, among other things, to define the following two geometric functions:
	\[
	R_g: X\to\mathbb{\R},  \ \  \ \ H_g:\partial X\to\R,
	\]
	describing the \emph{scalar curvature} of $X$ and the \emph{mean curvature} of its boundary, respectively.
	The interplay between these two functions, in its diverse forms, and the related study of spaces of positive scalar curvature metrics on manifolds with boundary has proven to be rather subtle and partly elusive, as it is witnessed by the extensive list of open problems and far-reaching conjectures proposed by Gromov in \cite{Gro18a, Gro18b} (see also \cite{Gro18book} for a more general contextualisation).

	We wish to continue here the study we initiated in \cite{CarLi19}, aimed at understanding the topology of spaces of metrics on $X$ defined by pairs of pointwise conditions given in terms of the two curvature functions defined above. In particular, in this article we will focus on the case when $X$ has dimension equal to three, where one can employ Hamilton's Ricci flow, and (as we are about to describe) various recent generalisations thereof, to derive remarkably strong conclusions.
	
	More specifically, we will couple the requirement of \emph{positive scalar curvature} (henceforth abbreviated PSC) with the requirement that the boundary be mean-convex (possibly: \emph{weakly} mean-convex) or minimal, and use the spaces $\met_{R>0,H>0}$, $\met_{R>0,H\geq 0}$ and $\met_{R>0,H=0}$ as references for our discussion.
	Besides the self-evident geometric significance, these are arguably among the most fundamental conditions that arise in the study of asymptotically flat initial data sets for the Einstein field equations under natural physical axioms (such as the so-called \emph{dominant energy condition}) and the customary \emph{outer trapping condition} on the boundary; we will indeed come back to this theme later in the introduction.
	
	\
	
	We recall from Section 2 of \cite{CarLi19} that, for any given compact manifold with boundary $X$, any one of these three spaces of metrics is empty if and only if the other two are; in fact from Theorem 1.1 therein we actually have an explicit topological description of those compact manifolds $X$ for which these spaces are \emph{not} empty: those must be (interior) connected sums of handlebodies (of any genus $\gamma\geq 0$), spherical space forms and a certain number of copies of $S^2\times S^1$. 
	
	 By combining elliptic deformations techniques and the theory of singular Ricci flows (as developed by Kleiner-Lott \cite{KleLot17, KleLot18} and then by Bamler-Kleiner, see in particular \cite{BamKle17,BamKle19}) we will prove the following statement:
	
	\begin{theorem}\label{thm:MainOpen}
	Let $X$ be a compact 3-manifold with boundary. Then any of the three spaces $\met_{R>0, H>0}, \met_{R>0,H\geq 0}, \text{or} \ \met_{R>0,H=0}$ consisting of positive scalar curvature metrics for which the boundary is, respectively, mean-convex, weakly mean-convex, or minimal is either empty or contractible.
	\end{theorem}
	
	This conclusion mirrors the corresponding one for compact surfaces, which however builds on much more classical tools, of complex-analytic character, such as the uniformisation theorem (for a detailed argument, which builds on basic properties of the Beltrami equation, see \cite{CarWu21}). Such a picture should however be compared and contrasted with the higher-dimensional scenario, where these spaces of metrics are expected to be non-trivial from a homotopy-theoretic perspective: in particular, by analogy with the closed case (for which the reader may wish to consult the surveys \cite{Car21}, \cite{Wal17, Wal18}, \cite{Sch14} or \cite{RosSto01}) we do expect these spaces to have, in many cases of interest, infinitely many connected components.
	
	We explicitly note that Theorem \ref{thm:MainOpen} is already new and highly significant in the special case of the three-dimensional disk $D^3$, as at the moment there seems to be essentially no information at all about the homotopy groups of these spaces of metrics on $D^n$ for $n\geq 4$. In fact, there seems to be some consensus on the expectation that even just this study at the $\pi_0$-level (i.e. at the level of path-connected components) may be quite out of reach, possibly also for those values of $n$ for which the space of positive scalar curvature metrics on spheres (or, more generally, on compact simply-connected manifolds \emph{without boundary}) has been reasonably understood (e.g. $n=7$, see \cite{GL83}). So, the case of manifolds with boundary does arguably seem to exhibit some genuinely new challenges compared to the closed case.

	Somewhat more in detail, there are \emph{three} fundamental building blocks in the proof of Theorem \ref{thm:MainOpen}. We will now describe them and provide three reference statements whose combination directly implies such a result.
First of all, as we showed in \cite[Section 3]{CarLi19}, one needs to start by straightening the boundary in a controlled fashion so that (among other things) the scalar curvature is kept positive throughout the deformation process. 
	
		\begin{theorem}\label{thm:ReductionToDouble}
	Let $X$ be a compact 3-manifold with boundary. Then any of the three spaces $\met_{R>0, H>0}, \met_{R>0,H\geq 0}, \text{or} \ \met_{R>0,H=0}$, when not empty, is weakly homotopy equivalent to the space $\met_{R>0, D}$ of doubling metrics of positive scalar curvature. 
	\end{theorem}
	
	  We refer the reader to Section \ref{sec:Elliptic} for the definition of doubling metrics, which are (roughly speaking) those that are `totally geodesic to all (odd) orders' i.e. those metrics that can be reflected across the boundary of $X$ thereby determining a smooth Riemannian metric on the closed manifold one obtains as \emph{double} of $X$. We note that Theorem \ref{thm:ReductionToDouble} follows from the very recent deformation results in \cite{BarHan20}, and we refer the reader again to Section \ref{sec:Elliptic} for a more specific, comparative discussion of the two methodologies.


	 Once the boundary is straightened, our initial problem can essentially be rephrased as a question about spaces of positive scalar curvature metrics subject to an equivariance constraint. We then work within the category of \emph{reflexive 3-manifold}, which we had introduced in Section 4 of \cite{CarLi19}, and adapt to that context the methodology designed by Bamler-Kleiner in \cite{BamKle19}.
	 Thus, the second fundamental block is the aforementioned employment of parabolic techniques, specifically of the singular Ricci flow, to prove the triviality of all homotopy groups of spaces of positive scalar curvature metrics with doubling boundary conditions:

	\begin{theorem}\label{thm:ContractDoubling}
	Let $X$ be a compact 3-manifold with boundary. Then the space $\met_{R>0, D}$ consisting of doubling metrics of positive scalar curvature is either empty or weakly contractible (that is to say: weakly homotopy equivalent to a point). 
	\end{theorem}
	
	As it is well-known, all milestones in our understanding of the Ricci flow (intended in its broadest sense) have led to significant advances on the study of spaces of metrics satisfying certain curvature conditions: Hamilton's original result \cite{Hamilton1982threemanifolds} immediately implies that the \emph{moduli space} of metrics of positive Ricci curvature is contractible, Perelman's theory of Ricci flow with surgery \cite{Per02, Per03a, Per03b} (cf. \cite{KleLot08} and \cite{MorTia07}) has then strikingly been employed by Marques \cite{Mar12} to prove that the moduli space of metrics of positive scalar curvature is path-connected and finally, almost a decade later, Bamler and Kleiner proved that on a compact orientable 3-manifold (without boundary) the space of PSC metrics is either empty or contractible. This breakthrough builds on the definition of a \emph{canonical} version of the Ricci flow (indeed, the aforementioned singular Ricci flow) together with several beautiful ideas about how to use it to design metric deformations, most prominently that of \emph{partial homotopy}, which we shall describe later. We also note that both the argument in \cite{Mar12} and the one in  \cite{BamKle19} crucially build on a \emph{backward-in-time induction scheme}, although the way this idea is actually implemented in the two cases is patently different: in the former the inductive assumption is used `at the surgical times' while in the latter one proceeds from the extinction time backwards by a small, fixed time interval $\Delta T$. 
Although the work by Bamler-Kleiner plays an essential role as an input for the present article, it is a matter of fact that not all arguments developed there are, so to say, canonical, and thus there do need to be suitable modifications when working in the equivariant setting. We will aim at indicating those modifications quite transparently, while avoiding to repeat all that can be easily transplanted from there. 

The proof of Theorem \ref{thm:ContractDoubling} is developed in Section \ref{sec:OutlineParab}, Section \ref{sec:Rstruct} and Section \ref{sec:PartialHom} with some additional, technical material relocated in three appendices for the sake of expository convenience. In order to not overload the present introduction, we have in fact decided to devote Section \ref{sec:OutlineParab} to a detailed outline of the proof of such a theorem, with all key definitions and an accurate description of how all tools described later come together to complete the argument.

Lastly, we note \emph{en passant} that Theorem \ref{thm:ContractDoubling} can be interpreted purely at the level of compact manifolds without boundary, and provides a partial description about the way a special class of equivariant PSC metrics sits inside the space all PSC metrics on the same background manifold. We expect this theme to be further explored (e.g. for different group actions) in the near future.

Let us now describe the third building block for Theorem \ref{thm:MainOpen}, which has to do with the general question about what structural assumptions (or properties) would ensure that a continuous map that induces isomorphisms at the level of all homotopy groups is actually a homotopy equivalence. While it is well-known that this is not always the case, classical results due to Whitehead (see \cite{Whi49a,Whi49b}) guarantee the conclusion to hold whenever the topological spaces involved are CW complexes. In general, to the aim of studying infinite-dimensional manifolds (e.g. within the realm of global non-linear analysis), other conditions have been singled out and we do verify that indeed our spaces are good enough for such an improvement to hold.

	\begin{theorem}\label{thm:ANRstructure}
	Let $X$ be a compact 3-manifold with boundary. Then any of the three spaces $\met_{R>0, H>0}, \met_{R>0,H\geq 0}, \text{or} \ \met_{R>0,H=0}$, when not empty, is an absolute neighborhood retract.
	\end{theorem}
	
	We postpone the discussion of this result to Section \ref{sec:ANR}. However, we wish to stress how (differently e.g. from the case treated in \cite{BamKle19}) we do need to pay some attention to the way Theorem \ref{thm:MainOpen} can actually be derived from Theorem \ref{thm:ReductionToDouble} and Theorem \ref{thm:ContractDoubling}: indeed, the relations that define the spaces of metrics we work with are \emph{not open} in general, and actually some ad hoc arguments need to be derived both for $\met_{R>0,H\geq 0}$ and for $\met_{R>0,H=0}$. We also mention that the statement above is just a special instance of a more general result, see Remark \ref{rem:ExtANR}. In any event, given the \emph{trio} of ancillary results above, our first main theorem follows at once:

     \begin{proof}(of Theorem \ref{thm:MainOpen})
     By transitivity, Theorem \ref{thm:ReductionToDouble} and Theorem \ref{thm:ContractDoubling} imply that any of the three spaces of metrics in question, 
     $\met_{R>0, H>0}, \met_{R>0,H\geq 0}, \text{or} \ \met_{R>0,H=0}$,
     is weakly contractible, in other words this space has trivial homotopy groups of all orders.
     
     Hence, due to Theorem \ref{thm:ANRstructure} which ensures the space is question to be an absolute neighborhood retract, a suitable Whitehead-type theorem (see Corollary \ref{cor:ContractibilityTopCrit}) allows to conclude. 
     \end{proof}

We can then move on and say some more words about the implications of our work on the study of initial data sets for the Einstein field equations (see e.g. \cite{Car21b} for a thorough introduction to the topic, at least from the perspective of geometric analysis).
Although our first main theorem concerns \emph{compact} manifolds with boundary, a striking blow-up argument (see Section 9 of \cite{Mar12} and \cite{Sch84}, where this idea was  employed to complete the resolution of the Yamabe problem) allows to derive similar consequences for asymptotically flat initial data sets. Here is a (somewhat informal) statement summarising what we can obtain in that respect:

\begin{theorem}\label{thm:GR constractible}Let $X_\infty$ denote a fixed background manifold that is assumed to be diffeomorphic to $\R^3\setminus \left(\sqcup_{j=1}^{\ell} P_{\gamma_j}\right)$, where $P_{\gamma_i}$ is a standard handlebody in $\R^3$ of genus $\gamma_i\ge 0$, and we further assume them to be pairwise separated (in the standard sense that different handleboies are contained in disjoint Euclidean balls). 
\begin{enumerate}
\item {The subspace of asymptotically flat Riemannian metrics subject to any of the following four curvature conditions
\begin{enumerate}
\item $R\geq 0, H\geq 0$, henceforth denoted by $\met^{AF}_{R\geq 0, H\geq 0}$;
\item $R\geq 0, H= 0$, henceforth denoted by $\met^{AF}_{R\geq 0, H=0}$;
\item $R=0, H\geq 0$, henceforth denoted by $\met^{AF}_{R=0, H\geq 0}$;
\item $R=0, H= 0$, henceforth denoted by $\met^{AF}_{R=0, H=0}$;
\end{enumerate}
is (not empty and) contractible.}
\item {The subspace of maximal, asymptotically flat solutions to the vacuum Einstein constraint equations, subject to the dominant energy condition and marginally outer trapped boundary, is (not empty and) contractible.}
\end{enumerate}
\end{theorem}

We refer the reader directly to Section \ref{sec:GR} for a specification of the regularity of the class of the data we work with, and for the corresponding asymptotic decay assumptions. We explicitly note that the theorem above applies, as a basic special instance, to the case when $X_{\infty}$ is diffeomorphic to $\mathbb{R}^3$ minus a finite number of disjoint balls, which is indeed the most general 3-dimensional topology (of asymptotically flat time-symmetric data) that is compatible with the dominant energy condition and the \emph{horizon boundary} requirements (see \cite{HuiIlm01}), i.e. with the setting of the Riemannian Penrose inequality. It is also to be remarked that, without much effort, we could also prove a version of Theorem \ref{thm:GR constractible} for asymptotically flat manifolds with multiple ends (and compact boundary components) although we did not explicitly discuss it here as that would bring essentially no new ideas into play.

The proof of Theorem \ref{thm:GR constractible} given Theorem \ref{thm:MainOpen} requires somewhat more (partly subtle) technical work than may look at first sight. In fact, Section \ref{sec:GR} turns out to be a rather significant section of the present article, and we believe some of the ancillary results we obtain there (for instance the construction given in Lemma \ref{lemma.diffeom.exponential.map}) may be of independent interest and applicability beyond our specific purposes.

Let us conclude this introduction with one more comment about our second main theorem. When the background manifold has the topology of $\mathbb{R}^3$ minus a ball, if we simply consider the implications of the result above (say for item (1)(d)) at the $\pi_0$-level we obtain the following striking implication: any given asymptotically flat, time-symmetric solution to the vacuum constraints with minimal boundary can be connected, through a path of such solutions, to the simplest one we know, i.e. the (one-ended) Riemannian Schwarzschild manifold. This seems, at least in a certain respect, a rather non-trivial conclusion as we now know how large and complicated the space of such solutions actually is, as it contains (among others) ample classes of exotic solutions like the localised ones constructed by the first-named author and Schoen in \cite{CarSch16}. So, loosely speaking, the analytic complexity of the space of solution to the Einstein constraints should be contrasted with the striking simplicity of the topology of such a space, which we hereby prove to be contractible.

\

			\textbf{Notation and conventions}.	For the convenience of the reader, we describe here some conventions and notational principles that will be adopted throughout this article. 
			
			The manifolds we deal with are always assumed to be smooth (i.e. $C^{\infty}$), as are their boundaries (if any); when we write \emph{compact manifold} we mean it to be boundaryless, and will instead specify \emph{compact manifold with boundary} if appropriate. All metrics are Riemannian and smooth, and the space of metrics on a given manifold (henceforth denoted $\met$), and all subsets thereof, are endowed with the corresponding topology; given a Riemannian metric $g$ we let $d_g$ denote the corresponding distance, and by writing $B_g(p,r)$ we mean the metric ball of center $p$ and radius $r$.
			
			Concerning the mean curvature, we adopt the following convention: the unit sphere in $\R^3$ has mean curvature equal to $2$; a domain is mean-convex if it has positive mean curvature, thus if an outward deformation with unit speed \emph{increases} area; we shall say that a manifold is \emph{weakly} mean-convex if it has non-negative mean curvature.
			
			For any $n\in\mathbb{N}=\left\{0,1,2\ldots\right\}$ we let $D^n$ denote the closed unit ball in $\R^n$, which we shall also refer to as $n$-dimensional disc, and by $\Delta^n$ the standard $n$-dimensional simplex. When we specify the radius, and write $D^n(r)$ or $B^n(r)$, we mean the $n$-dimensional disc or the $n$-dimensional open ball of radius $r>0$ and centered at the origin in $\R^n$. Throughout this article, $f_0$ denotes the standard reflection in $\R^3$ with respect to the last coordinate plane (and by slight abuse of notation, the same denotes restrictions of the same map to any invariant subdomain). We further let $I=[0,1]$ denote the standard unit interval in $\mathbb{R}$ and, for any $k\in\mathbb{N}=\left\{0,1,2\ldots\right\}$ we let $I^k$ denote the standard product of $k$ copies of $I$, i.e. the unit cube $[0,1]^k$ in $\mathbb{R}^k$.
			
			We will write $u_{+}$ and $u_{-}$ to represent the positive and, respectively, negative part of any given real-valued function $u$, so that $u=(u_{+}-u_{-})/2$ (thus, in particular, if $u\ge 0$ then $u_{-}=0$). We employ the letter $\nu$ to denote a cutoff function, defined once and for all in Section \ref{sec:Rstruct}.

		\

		\noindent \textit{Acknowledgements:} The authors would like to thank Richard Bamler for helpful clarifications, Christos Mantoulidis for conversations on themes related to the present project, Sander Kupers and Boyu Zhang for indicating significant references and related discussions.
		
			This project has received funding from the European Research Council (ERC) under the European Union’s Horizon 2020 research and innovation programme (grant agreement No. 947923). C. L. wishes to acknowledge the support of the National Science Foundation through NSF grant DMS-2005287. The results contained in this paper have been presented by the first-named author in June 2021 during the 8th European Congress of Mathematics.

		\section{Straightening the boundary: weak homotopy equivalences}\label{sec:Elliptic}
		

		Let $n\geq 2$ (in fact: $n\geq 3$ in most circumstances, as we shall specify below) and let $X^n$ be a smooth (i.e. $C^{\infty}$) compact manifold with boundary. 
		Taken two copies of $X$ (which we henceforth label $X_1, X_2$, with $\iota:X_1\to X_2$ the identity map) we consider on the disjoint union $X_1\sqcup X_2$ the equivalence relation $\sim $ given by declaring $x_1\sim x_2$ if $x_1\in \partial X_1, x_2\in \partial X_2$ and $x_2=\iota(x_1)$. We thus define, as a set, $M=(X_1\sqcup X_2)/\sim$ (with the projection $\pi:X_1\sqcup X_2\to M$, and let $X'=\pi(\partial X_1)=\pi(\partial X_2)$) endowed with the differentiable structure associated to the cover $\left\{\pi(X_1\setminus \partial X_1), \pi(X_2\setminus \partial X_2), Y\right\}$ where $Y$ is a collar neighborhood of $X'$ in $M$. Indeed, given any metric $g_0$ on $X$ we know that, by the tubular neighborhood theorem, the exponential map provides an identification $ Y_1\equiv \partial X_1\times (-\epsilon,0]$, and similarly (on the second copy) $Y_2\equiv \partial X_2\times [0,\epsilon)$ so that we can just set $Y=\pi(Y_1\cup Y_2)$. Note that, if we vary the metric $g_0$ the resulting manifold $M$ does not change, up to diffeomorphism.
		
		\begin{remark}\label{rmk:doubling}
Here are below we will employ the label $_{D}$ in order to refer to the \emph{doubling} metrics on the background manifold $X$, namely those metrics that can be smoothly doubled to a smooth Riemannian metrics on the double $M$.
To rephrase this requirement recall that, near any connected component of $\partial X$, there is a standard local representation of a metric $g$ determined by a distance function from a boundary component
\begin{equation}\label{eq:FermiCoord}
g=dt\otimes dt + g_t,
\end{equation}
hence $g$ is doubling if and only if 
\be\label{eq:DoublingCond}
\frac{\partial^{(2\ell+1)}}{\partial t^{(2\ell+1)}}g_{t}=0, \ \text{for all} \ \ell\in\mathbb{N}.
\ee
More generally, given curvature conditions $\ast$ defined in terms of the scalar curvature (typically: $R>0$) we shall write $\met_{\ast, D}$ to denote the set of doubling metrics satisfying $\ast$.
\end{remark}

		One may first wonder what manifolds with boundary do support positive scalar curvature \emph{doubling} metrics.
With that goal in mind, let us first recall the following basic deformation result.

	\begin{lemma}\label{lem:BasicEquiv}(Corollary 2.6 in \cite{CarLi19})
			Let $n\geq 3$ and let $X^n$ be a connected, compact manifold with boundary. Then the following three assertions are equivalent:
	\begin{enumerate}
	\item [i)]{$\met_{R>0, H>0}\neq\emptyset$;}
	\item [ii)]{$\met_{R>0, H\geq 0}\neq\emptyset$;}
	\item [iii)]{$\met_{R\geq 0, H>0}\neq\emptyset$.}
	\end{enumerate}	
Furthermore, each of these conditions is equivalent to
	\begin{enumerate}
	\item [iv)]{$\met_{R\geq 0, H\geq 0}\neq\emptyset$,}
\end{enumerate}	
unless the space $\met_{R\geq 0, H\geq 0}$ only contains Ricci flat metrics making the boundary totally geodesic.
	\end{lemma}
	
	\begin{remark}\label{rem:Def3dNd}
	The proof of the statement above was written, in \cite{CarLi19} for $n=3$, but applies identically (except for purely notational changes) to any dimension $n\geq 3$; we however note that when $n=3$ the last clause would additionally force $X$ (if orientable, else its orientable double cover) to be diffeomorphic to $S^1\times S^1\times I$ and any metric in $\met_{R\geq 0, H\geq 0}$ to be flat. In general, i.e. for $n\geq 4$, the same conclusion does not hold (as shown by well-known examples like the product of a ribbon $S^1\times I$ times a $K3$ surface); however, some additional information may be derived by appealing to the characterisation of compact Ricci-flat manifolds (of any dimension) given in \cite{FisWol75}.
	\end{remark}
	
	Hence, it follows from Theorem 5.7 in \cite{GroLaw80-Spin} (combined with an appropriate smoothing technique, for which \cite{Mia02} would suffice) that one can ensure that, so to say, doubling metrics exist in abundance:
	
	\begin{corollary}\label{cor:StraightBoundary}
	Let $n\geq 3$ and let $X^n$ be a connected, compact manifold with boundary. If $X$ supports metrics in $\met_{R\geq 0, H\geq 0}$ at least one of which is either non Ricci flat or has non-totally geodesic boundary then $X$ shall also support doubling metrics of positive scalar curvature.
	\end{corollary}
	
	As we showed in \cite[Section 3]{CarLi19} one can actually start from the Gromov-Lawson doubling counstruction to actually design isotopies of smooth metrics that straighten the boundary. In fact, following the very same arguments, rather simple modifications allow to design deformations for finite-dimensional families of Riemannian metrics; however, the resulting isotopies do not preserve the doubling boundary conditions throughout the deformation. That being said, a totally different construction (ultimately appealing to a suitable $h$-principle) has very recently been proposed by B\"ar-Hanke in \cite{BarHan20}, which does indeed enjoy such additional feature. In particular, there holds the following statement:
		
		\begin{theorem}\label{thm:WeakHom}
		Let $X^3$ be a connected,  compact 3-manifold with boundary. Then, in the following commutative diagram, any inclusion map induces a weak homotopy equivalence
		\[
  \begin{tikzcd}
    & & \met_{R>0, H>0} \arrow{d}{}   \\
      \met_{R>0, D}\arrow{r}{} & \met_{R>0, H=0} \arrow{r}{} & \met_{R>0, H\geq 0}.
  \end{tikzcd}
\]
\end{theorem}

\begin{remark}\label{rmk:WeakHom}
Employing the long exact sequence for the homotopy of pairs, the claim that the inclusion map $A\rightarrow X$ induces a weak homotopy equivalence follows by showing that all \emph{relative} homotopy groups vanish, which in turn can be conveniently obtained by checking that for any integer $i\in\mathbb{N}$ one has that \emph{any (continuous) map $(D^i,\partial D^i)\to (X,A)$ is homotopic, through such maps, to a map $D^i\to A$}; when $i=0$ one requires that each connected components of $X$ contains points of $A$. 
\end{remark}

 The weak homotopy equivalences $\met_{R>0, H=0}\rightarrow \met_{R>0, H\geq 0}$ and $\met_{R>0, D}\rightarrow \met_{R>0, H\geq 0}$ follow from Theorem 32 and Corollary 34 in \cite{BarHan20} (applied for $\sigma=0, h_0=0$);
similarly the weak homotopy equivalence $\met_{R>0, H>0 }\rightarrow \met_{R>0, H\geq 0}$ follows from Theorem 33 therein (again for $\sigma=0, h_0=0$).

		\section{ANRs and strong homotopy equivalences}\label{sec:ANR}
		
		We now wish to discuss how \emph{weak} homotopy equivalences, see Remark \ref{rmk:WeakHom}, can be upgraded to actual homotopy equivalences, so to ultimately conclude the contractibility of some of the spaces of metrics we are dealing with.
		It is to be remarked that these tools will also be crucially employed in Section \ref{sec:GR}, when discussing the applications to asymptotically flat initial data sets in general relativity. 
		
		In order to avoid ambiguities, we shall first recall certain basic definitions and a few classical facts.
		
		\begin{definition}\label{def:Retract}
		Given a topological space $Z$, a subset $A\subset Z$ is called a retract of $Z$ if there exists a continuous map (to be called a \emph{retraction}) $r:Z\to A$ such that $r(a)=a$ for any $a\in A$.
		\end{definition}
		
		\begin{remark}\label{rem:BasicRetract}
		\begin{itemize}
		    \item a retraction determines a factorisation of the identity map of the set $A$, namely we have the commutative diagram
		    
		    		\[
  \begin{tikzcd}
   A \arrow{r}{i} \arrow{dr}{id_A}  & Z \arrow{d}{r}   \\
       & A.
  \end{tikzcd}
\]
		    
		    thus we have $r_{\ast}\circ i_{\ast}=(id_A)_{\ast}$ at the level of homotopy groups (in particular the inclusion induces an injection).
		    \item in the setting above, if $Z$ is Hausdorff (e.g. if it is a metric space) then $A$ must be a closed subset of $Z$;
		    \item the following (universal) extension property holds true: given any topological space $W$ and $f\in C^0(A,W)$ there exists $\tilde{f}\in C^0(Z,W)$ extending $f$; indeed, it suffices to set $\tilde{f}=f\circ r$. 
		    		\[
  \begin{tikzcd}
   A \arrow{r}{f}  & W   \\
   Z \arrow{u}{r} \arrow{ur}{\tilde{f}}  & 
  \end{tikzcd}
\]
		\end{itemize}
		\end{remark}

		\begin{definition}\label{def:DefRetract}
		In the setting of the previous definition, we will say that $A\subset Z$ is a deformation retract of $Z$ if there exists a homotopy $H:Z\times I\to Z$ such that $H(z,0)=i\circ r(z)$ and $H(z,1)=z$ for all $z\in Z$.
		\end{definition}
		
		\begin{remark}\label{rem:DeformationRetract}
		If $A$ is a deformation retract of $Z$ then the inclusion and retraction maps determine a (strong) homotopy equivalence between $A$ and $Z$. In particular, the two spaces will have the same homotopy groups.
		\end{remark}
		
		\begin{definition}\label{def:ANR}
		  Let $\mathscr{T}$ denote the class of metric spaces. Then:
		  \begin{enumerate}
		      \item $X\in \mathscr{T}$ is called an absolute retract (henceforth abbreviated AR) if 
		      \[
		      \begin{cases}
		      Y\in \mathscr{T} \\
		      X \subset Y \ \text{closed} 
		      \end{cases} \ \Longrightarrow \ X \ \text{is a retract of} \ Y.
		      \]
		      \item $X\in \mathscr{T}$ is called an absolute neighborhood retract (henceforth abbreviated ANR) if 	      \[
		      \begin{cases}
		      Y\in \mathscr{T} \\
		      X \subset Y \ \text{closed}
		      \end{cases} \ \Longrightarrow \ X \ \text{is a neighborhood retract of} \ Y, i.e.
		      \]
		      there exists $U\subset Y$ open, with $X$ a retract of $U$.
		       \end{enumerate}
		 \end{definition}
		 
		 \begin{remark}
		 As shown e.g. in Theorem 3.1(ii) of \cite{MarSeg82} the condition that a topological space $X$ be an ANR could be equivalently be given as follows: for every metric space $Y$, every closed subspace $A$ of $Y$ and every continuous map $f:A\to X$, there are an open neighbourhood $U$ of $A$ in $Y$ and a continuous map $\overline{f}: U \to X$ that extends $f$.
		 	 We explicitly note that the latter is in fact the definition adopted by Palais in \cite{Pal66}.
		 	 \end{remark}
		 	 
		 	 The following statement is one of the most fundamental results related to such a notion.
		 	 
		 	 \begin{theorem}\label{thm:Dugundji}(Dugundji extension theorem, cf. Theorem 3.3 in \cite{MarSeg82})
		 	 A convex subset of a normed linear space is an ANR. 
		 	 \end{theorem}
		 	 
		 	 To provide a succinct statement for a criterion, due to Palais, generalising the assertion above we first need to agree on the following.
		 	 
		 	 \begin{definition}\label{def:PalaisManifold}
		 	  (In this section) we will call \emph{manifold} a topological manifold with boundary, i.e. more precisely a Hausdorff space where each point has a neighborhood homeomorphic to a convex open set either in a locally convex topological vector space $V$ or in a half-space of $V$, the latter being understood as a set of the form $\left\{v\in V \ : \ \ell(v)\geq 0\right\}$ for a continuous linear functional $\ell$ on $V$.
		 	 \end{definition}
		 	 
		 	 \begin{theorem}\label{thm:PalaisCriterion} (Theorem 5 in \cite{Pal66})
		 	 A metrisable manifold is an ANR. More generally, a metrisable space is an ANR if each point has a neighborhood homeomorphic to a convex set in a locally convex topological vector space.
		 	 \end{theorem}
		 
		 Note that the previous theorem implies that, in particular, Banach manifolds and  (more generally) Fr\'echet manifolds are indeed examples of ANRs; as a result, the class of all smooth tensors, of any fixed type, on a compact manifold is an ANR. That being said, we
		 next recall the following two general topological criteria:
		 
		 \begin{proposition}\label{prop:Criteria} (Proposition A.6.4 in \cite{FriPic90})
		 \begin{itemize}
	\item	 An open subset of an ANR is an ANR.        
		         	 	\item	 A retract of an ANR is an ANR.  
		 \end{itemize}
		         	 \end{proposition}    	 
		

		Now, the key reason why we are interested in studying absolute neighborhood retracts is the following assertion:
		
		\begin{theorem}(see Theorem 15 in \cite{Pal66}, cf. \cite{Mil59})
		A map of ANRs that induces an isomorphism at the level of all homotopy groups (for any choices of base points) is a homotopy equivalence. 
		\end{theorem}
		
		\begin{corollary}\label{cor:ContractibilityTopCrit}
		If $X$ is an ANR then the following assertions are equivalent:
		\begin{enumerate}
		    \item $\pi_n(X)=0$ for all $n\in\mathbb{N}=\left\{0,1,2,\ldots \right\}$;
		    \item $X$ is contractible.
		\end{enumerate}
		\end{corollary}
		
		Let us first see an application of this machinery to the simpler case of compact manifolds without boundary. For a fixed compact background manifold $M$, we can regard $\met$ (the space of Riemannian metrics on $M$) as an open cone inside the locally convex topological vector space of smooth symmetric $(0,2)$ tensors on $M$. The latter is in fact a Fr\'echet space. Hence $\met$ is an ANR. We claim that $\met_{R>0}$ is also an ANR: indeed one can consider the continuous maps
						    		\[
  \begin{tikzcd}
   \met \arrow{r}{Scal}  & C^{\infty}(M)  
   \arrow{r}{Min} & \R 
  \end{tikzcd}
\]
where $\text{Min}(f)=\min_{x\in M}f(x)$.
Then $\met_{R>0}$ is the pre-image, through such composite map, of the open interval $(0,\infty)$ so this is an open set inside $\met$ and thus by Proposition \ref{prop:Criteria} we conclude as claimed.

In a similar fashion we can also derive, as a direct consequence of the chain of topological results above, Theorem \ref{thm:ANRstructure} for the space $\met_{R>0, H>0}$, i.e. when only \emph{open} conditions come into play. 

\begin{proof}(of Theorem \ref{thm:ANRstructure}, for $\met_{R>0, H>0}$)
As above, consider the Fr\'echet space $\met$ of Riemannian metrics on $X$; by Theorem \ref{thm:PalaisCriterion} and the first part of Proposition \ref{prop:Criteria} this is an ANR. In fact, the space $\met_{R>0, H>0}$ is also an ANR: arguing as above, one considers 
						    		\[
  \begin{tikzcd}
   \met \arrow{r}{Scal}  & C^{\infty}(X)  
   \arrow{r}{Min} & \R 
  \end{tikzcd}, 
     \begin{tikzcd}
   \met \arrow{r}{Mean}  & C^{\infty}(\partial X)  
   \arrow{r}{Min} & \R 
  \end{tikzcd}
\]
and $\met_{R>0, H>0}$ is the intersection of the pre-images of $(0,\infty)$ through the two maps above. Thus, 
this is an open set inside $\met$, whence (again by Proposition \ref{prop:Criteria})  the conclusion follows.     
\end{proof}

Gaining the corresponding results for the spaces $\met_{R>0, H\geq 0}$ and $\met_{R>0, H=0}$ requires some more work, and suitable \emph{ad hoc} arguments. This is what we deal with in the remainder of this section.

\

Consistently with the notation employed throughout this article, we let $(X,g)$ denote a compact Riemannian manifold with (nonempty) boundary. Denote then by $V$ the outward-pointing unit normal vector field of $\partial X$, and by $i(\partial X,g)$ the normal injectivity radius, which is defined as the supremum of the set of numbers $r$ such that
\[\exp_{\cdot}(-rV): \partial X\to X,\quad p\in \partial X\mapsto \exp_p(-rV)\]
is a diffeomorphism onto its image. It should be a standard fact that $i(\partial X,g)$ depends continuously on $g$ with respect to the $C^2$ topology, but (to our surprise) we were unable to find any proper reference. As a result, for the sake of completeness (and given the importance of this matter with respect to the purposes of the present work) we have decided to provide a short proof of such a result exploiting a beautiful idea in \cite{Sakai1983continuity}, see Appendix \ref{sec:NormalInjRad}. 

\

In this section, we let $t: X\rightarrow \R_{\ge 0}$ be the distance function to $\partial X$. We shall write a Riemannian metric $g$ in normal coordinates, i.e. we write $g=dt^2 + g_t(x)$, with $x\in \partial X$ and $t$ non-negative, small enough.
For suitably small values of $\rho$, the level surface $t^{-1}(\rho)$ is smooth and, denoted by $S^g_{\rho}$ its shape operator, we note that at any given point
 the Hessian $\nabla_g^2 t$ is the second fundamental form of $t^{-1}(\rho)$, taken with respect to the unit normal vector $\nabla_g t$, so it equals $S^g_{\rho}$ modulo a contraction with the background metric $g$.
 Also, with our general sign conventions, its trace $\Delta_g t$ equals \emph{modulo a sign change} the mean curvature of the level surface. We will further denote by $H_g(x,t)$ the mean curvature at a point $x$ on the slice at distance $t$, with respect to $-\nabla_g t$. When no ambiguity is likely to arise we will write $H_g(\cdot)$ in lieu of $H_g(\cdot, 0)$. Firstly, we need the following preparatory result.

\begin{lemma}\label{lemma.eta}
    There exists a function $\eta: \met_{R>0}\to\R$, continuous with respect to the $C^2$-topology, such that the normal injectivity radius of $\partial X$ is at least $\eta(g)$,  $\eta\leq 1$ and, in addition, the following properties hold true:
    \begin{enumerate}
        \item for any metric $g\in\met_{R>0}$ there holds 
        \begin{equation}\label{eq:eta1}
        \eta(g)\leq \min\left\{\sqrt{\min_{X }R_g}/10, \ \min_{X} R_g/100\right\};
        \end{equation}
        \item for any metric $g\in\met_{R>0}$, and for all $t\le \eta(g)$ there holds
        \begin{equation}\label{eq:eta2}
         \|H_g(\cdot,t)|_{L^{\infty}(\partial X)}\le 2 (\|H_g\|_{L^{\infty}(\partial X)}+1)
         \end{equation}
         as well as
         \begin{equation}\label{eq:eta3}
         \|H_{g}(\cdot, 0)\|_{C^2(\partial X, g_t)}\le 2 (\|H_g\|_{C^2(\partial X, g_0)}+1)
         \end{equation}
    \end{enumerate}
\end{lemma}

\begin{proof}
It descends from Proposition \ref{proposition.continuity.normal.inj} that one can find a function $\eta_1$, depending continuously on $g$ with respect to the $C^2$-topology, such that the normal injectivity radius of $\partial X$ is at least $\eta_1(g)$. In particular, this determines a cotinuous function $\eta_1:\met\to\R$, which we shall then restrict (without renaming) to $\met_{R>0}$, the subset of metrics of positive scalar curvature.

The scalar curvature function is a continuous map from the space of metrics, endowed with the $C^2$-topology, so it is clear that \eqref{eq:eta1} can be accomodated as well. (For instance, one could preliminarily define $\eta$ to be the lower envelope of the three continuous functions $\eta_1(g), \sqrt{\min_X R_g}/10, \min_X R_g/100$).

Similarly, possibly by shrinking $\eta(g)$ even further one can make sure \eqref{eq:eta3} holds: indeed, the family $(g_t)$ is a continuous path of smooth metrics on $\partial X$ so, for every given $g$ one can always satisfy \eqref{eq:eta3} for small enough $t$ (continuously depending on $g$) just by virtue of a continuity argument. In fact, note that for this specific purpose the $C^1$ topology on the space of metrics would suffice, as the expression of the Hessian of a function in local coordinates only involves the components of the metric as well as its first-order derivatives.

Thus, the rest of the proof concerns the requirement \eqref{eq:eta2} instead. Then, given $g\in\met$, consider a unit speed geodesic $\gamma$ orthogonal to $\partial X$. For $t\le \eta_1(g)$, we have the following Riccati equation along $\gamma$ (see e.g. Corollary 7.24 in \cite{Lee19}):
    \[
    \nabla^g_{\partial_t} (S^g_s) + (S^g_t)^2 = -\text{Rm}_g(\partial_t,\cdot) \partial_t,\]
    where $\text{Rm}_g$ stands for the Riemann curvature tensor, seen as a tensor of type $(1,3)$.
    
    By standard ODE comparison theory, thanks to the compactness of $X$ there exists a positive number $\eta(g)\leq\eta_1(g)$, depending on a two-sided bound on the sectional curvature of $X$ and on the second fundamental form of $\partial X$, such that when $s\le \eta(g)$, one has
    \[
    |S^g_t|^2(\gamma(s))\le 2 (|S^g_0|^2(\gamma(0))+1)\leq 2 (\|S^g_0\|_{L^{\infty}(\partial X)}^2+1)
    \]for any such $\gamma$.
    
    Tracing the Riccati equation, we then have
    \[\partial_t(H_g(\cdot, t)) + |S^g_t|^2 = -\Ric_g(\partial_t,\partial_t).\]
     Thus, exploiting the previous bound on $|S^g_t|^2$, by possibly shrinking $\eta$ we similarly conclude that $|H_g(x,t)| \le 2 (|H_g |(x,0)+1)$ for all $x\in\partial X$, whenever $t\le \eta(g)$. 
    
   As a result, the argument above allows to design a function $\eta$, depending continuously on $g$ with respect to the $C^2$-topology, and bounded from above by $\eta_1$ at each point, satisfying all desired properties.

\end{proof}

To proceed, we fix a smooth function $\f: [0,1]\to[-1,0]$ supported in $[0,2/3)$, such that $\f(x)=-x$ when $x\in [0,\frac 13]$, and also $|\f'|, |\f''|<5$.

\begin{proof}(of Theorem \ref{thm:ANRstructure}, for $\met_{R>0, H\geq 0}$)

We consider the following space of Riemannian metrics $\met_1$ on $X$ defined by:
\begin{multline*}
    \met_1=\bigg\{g \in\met : R_g>0, \\ 1000 (H_g)_{-} <  \min\left\{\sqrt{\min_{X} R_g}, \min_{X}R_g\cdot \eta(g), \min_X R_g\cdot (\|H_g\|_{L^\infty(\partial X)}+1)^{-1}\right\} \bigg\}.
\end{multline*}
 Observe that $\met_{R>0, H\ge 0}\subset \met_1$, and that $\met_1$ is an open subset of the space of Riemannian metrics on $X$. Hence $\met_1$ is an ANR.
We will now prove that there exists a retraction $F: \met_1\to \met_{R>0, H\ge 0}$, which will imply the conclusion by virtue of (the second assertion in) Proposition \ref{prop:Criteria}.

 For any $g\in \met_1$, denote $\eps=\max (H_g)_{-}$, and consider $\eta=\eta(g)$ provided by Lemma \ref{lemma.eta}. Define the functions $u, w: X\to\mathbb{R}$ by letting 
\[ 
 w(x,t)=\frac14 \eps \eta \f \left(\frac{t}{\eta}\right), \ \ u=e^w. 
\] 
It is of course understood that the function $w$ extends to a smooth function on $X$, that is identically equal to 0 away from a tubular neighborhood of width $\eta$ of the boundary. 
 We then define
    \[F(g)=u^{4} g.\]
    We observe first that if $g\in \met_{R>0,H\ge 0}$, then $\eps=0$, hence $w=0, u=1$ on $X$ and thus $F(g)=g$. That is to say, $F$ is the identity map on $\met_{R>0, H\ge 0}$. Thus, it suffices to check that $F(\met_1)\subset \met_{R>0, H\ge 0}$. Denote $\tilde g=F(g)$. We compute:
    \[H_{\tilde g} = u^{-3}\left(H_g u +4 V(u)\right)=u^{-2}\left(H_g+4V(w)\right),\]
    and, along $\partial X$,
    \[
    V(w) = -\frac14 \eps \f'(0) = \frac14 \eps.\]
    Thus $H_g+4V(w) \ge -\eps + \eps \ge 0$.
    
    Similarly, 
    \[R_{\tilde g}=u^{-5} (R_g u - 8 \Delta_g u)= u^{-4}(R_g - 8|\nabla_g w|^2 - \Delta_g w).\]
    We have for $t\leq\eta(g)$ that
    \[\nabla_g w = \frac 14 \eps \f'\left(\frac{t}{\eta}\right)\nabla_g t\quad \Rightarrow \quad |\nabla_g w|^2 \le \frac{1}{16}\eps^2 |\f'|^2<\frac{1}{20} \min_X R_g,\]
      thanks to the fact that $g\in\met_1$.
    Moreover,
    \[\Delta_g w = \frac{\eps}{4\eta} \f''\left(\frac{t}{\eta}\right)|\nabla_g t|^2 + \frac{\eps}{4} \f'\left(\frac{t}{\eta}\right)\Delta_g t.\]
    Thus, relying  on \eqref{eq:eta2} in Lemma \ref{lemma.eta} and again on the specific conditions defining $\met_1$, we estimate:
    \[|\Delta_g w|\le \eps \left(\frac{5}{4\eta} + \frac52 \|H_g\|_{L^\infty(\partial X)}+\frac 52\right)<\frac{1}{20} \min_X R_g.\]
    We therefore conclude that $R_{\tilde g}>0$ on $X$, so $F$ maps $\met_1$ inside $\met_{R>0, H\geq 0}$, as claimed.
\end{proof}

We now proceed and consider the case of \emph{minimal} boundary conditions. In what follows, we let $\nabla_g, \Delta_g$ denote the operators associated to the metric $g$, and let $\overline{\nabla}_{g_t},\overline{\Delta}_{g_t}$ denote the corresponding operators associated to $g_t$ on a $t=\text{constant}$ slice. We recall that, for any function $\psi\in C^2(X,\R)$ there holds
 \begin{equation}\label{eq:LaplaceInterface}
 \Delta_g \psi = \overline{\Delta}_{g_t} \psi + H_{g}(\cdot, t) \frac{\partial \psi}{\partial t} + \frac{\partial^2 \psi}{\partial t^2}.
 \end{equation}

The following proof is then a variation of the argument above.

\begin{proof}(of Theorem \ref{thm:ANRstructure}, for $\met_{R>0, H=0}$)

We consider the following space of Riemannian metrics $\met_2$ on $X$ defined by: \begin{multline*}
\met_2=\bigg\{g\in\met: R_g>0, \\  1000\|H_g\|_{C^2(\partial X, g_0)}<\min\left\{\sqrt{\min_X R_g}, \min_X R_g \cdot \eta(g), \min_X R_g\cdot (\|H_g\|_{C^2(\partial X, g_0)}+1)^{-1}\right\}\bigg\},
\end{multline*}
where the function $\eta$ is again the one constructed in Lemma \ref{lemma.eta}. 
    Clearly $\met_2$ is an open subset of the set $\met$ of Riemannian metrics on $X$ hence it is an ANR.
We will now prove that there exists a retraction $F: \met_2\to \met_{R>0, H=0}$, which will imply the conclusion by virtue of (the second assertion in) Proposition \ref{prop:Criteria}.
    
    For a metric $g\in \met_2$, and a point near $\partial X$, we define a function $w$ by letting 
    \[w(x,t)=-\frac14 H_g(x,0)\eta \f\left(\frac{t}{\eta}\right),\] which extends smoothly to all of $X$.
    Then define $u=e^w$, and $F(g)=u^4 g$.
    
 $F$ is the identity map on $\met_{R>0, H=0}$; denoted $\tilde g= F(g)$ it is also straightforward to check that $H_{\tilde g}=0$. Thus, it only remains to verify that $R_{\tilde g}>0$, which ensures that such a map is well-defined. As noted before, this requirement is equivalent to $8|\nabla_g w|^2 + \Delta_g w < R_g$. 
    
    To check this inequality, we first estimate the gradient term:
    \[|\nabla_g w|\le \frac14 |\overline{\nabla}_{g_t} H_{g}(\cdot,0)|\eta + \frac54 |H_{g}(\cdot,0)|\leq \frac12 (\|H_g\|_{C^2(\partial X, g_0)}+1)\eta + \frac54 |H_{g}(\cdot,0)|,\]
    where we have exploited equation \eqref{eq:eta3}.
    Thus, employing \eqref{eq:eta1} in Lemma \ref{lemma.eta}, the fact that $g\in\met_2$ implies, in particular, that $|\nabla_{g} w|^2 <\frac{1}{20} \min_X R_g$. 
       Using \eqref{eq:LaplaceInterface} and then again \eqref{eq:eta3} as well as \eqref{eq:eta2}, we similarly have:
    \begin{multline*}|\Delta_g w|(\cdot,t) \le \frac14 |\overline{\Delta}_{g_t} H_{g}(\cdot, 0)|\eta + \frac54 |H_{g}(\cdot, 0)| |H_g(\cdot,t)| + \frac54 \frac{|H_{g}(\cdot,0)|}{\eta}\\
    \leq \frac12 (\|H_g\|_{C^2(\partial X, g_0)}+1)\eta + \frac52 |H_{g}(\cdot, 0)|(\|H_g\|_{C^2(\partial X, g_0)}+1) + \frac54 \frac{|H_{g}(\cdot,0)|}{\eta}.\end{multline*}
Now, simply by appealing to \eqref{eq:eta1} and to the appropriate assumptions in the definition of $\met_2$ one concludes that $|\Delta_g w|< \frac{1}{20} \min_X R_g$, which allows to complete the proof.
\end{proof}
		
		\begin{remark}\label{rem:ExtANR}
		We explicitly note that simple variations on the arguments we presented above allow to derive the conclusion of Theorem \ref{thm:ANRstructure} for the spaces $\met_{R>R_0, H>H_0}$, $\met_{R>R_0, H\geq H_0}$ and $\met_{R>R_0, H=H_0}$ for any value of $R_0$ and $H_0$, and for any compact background manifold of dimension $n\geq 3$. 
		\end{remark}
		
        \section{Outline of the proof of the weak contractibility}\label{sec:OutlineParab}
        
        In this section we shall prove that the space of \emph{doubling} positive scalar curvature metrics has vanishing homotopy groups in all degrees (i.e. of Theorem \ref{thm:ANRstructure}), giving for granted various technical details whose presentation is postponed to Section \ref{sec:Rstruct}, Section \ref{sec:PartialHom}, Appendix \ref{sec:PSCconf}, Appendix \ref{app:EquivExtHom} and Appendix \ref{app:EquivDiskRemov} instead. We will in fact also describe the precise content of all these additional sections and describe how the pieces come together.

        We will freely employ the terminology concerning singular Ricci flows, and families thereof, that is given in \cite{BamKle19} and limit ourselves to explicitly recall here those notions that are most fundamental, or anyway key to the arguments we are about to present. The proof of Theorem \ref{thm:ContractDoubling} builds on a suitable \emph{equivariant} counterpart of the constructions given there; for the sake of expository clarity and concreteness we have decided to develop the theory to the extent needed for our specific scopes rather than dealing with general group actions (whose treatment may, however, be relevant to other classes of applications).
        
        \begin{definition}\label{def:RicciFlowSpacetime}
        A Ricci flow spacetime is a tuple $(\mM, \mt, \partial_\mt,g)$ with the following properties:
        \begin{enumerate}
            \item $\mM$ is a smooth 4-manifold with boundary $\partial\mM$;
            \item $\mt:\mM\to [0,\infty)$ is a smooth function without critical points (called \emph{time function});
            \item $\partial \mM=\mt^{-1}(0)$, i.e. the boundary of $\mM$ coincides with the initial time slice; 
            \item $\partial_{\mt}$ is a smooth vector field (called \emph{time vector field}) on $\mM$ that satisfies $\partial_{\mt}\mt\equiv 1$;
            \item $g$ is a smooth Riemannian metric on the horizontal (equiv. spatial) sub-bundle $\ker (d\mt)$;
            \item $g$ satisfies the Ricci flow equation $\mL_{\mt}g=-2\Ric(g)$.
        \end{enumerate}
        \end{definition}

        We will further need some related notation: for $t\geq 0$ we set $\mM_{t}=\mt^{-1}(t)$, and let $g_t$ denote the restriction of $g$ to $\mM_t$. It is tacitly understood, in item (6) above and throughout the article, that all curvature tensors (such as the Ricci curvature) at a given point of $\mM_t$ refer to $g_t$. 
        When no ambiguity is likely to arise, we shall systematically write $\mM$ in lieu of $(\mM, \mt, \partial_\mt,g)$.
        
        \begin{definition}\label{def:SingularRicciFlow}
        A singular Ricci flow is a Ricci flow spacetime $(\mM, \mt, \partial_\mt,g)$ satisfying the following properties:
        \begin{enumerate}
            \item $\mM_0$ is a compact manifold; 
            \item $(\mM, \mt, \partial_\mt,g)$ is complete in the sense of Definition 3.12 in \cite{BamKle19};
            \item for every $\e>0$ and $0\leq T<\infty$ there is a constant $r_{\e,T}$ such that $(\mM, \mt, \partial_\mt,g)$ satisfies the $\e$-canonical neighborhood assumption below scale $r_{\e,T}$ on $[0,T]$.
        \end{enumerate}
        It is said that such a spacetime is \emph{extinct at time $t$} if $\mM_t=\emptyset$ (note that if a singular Ricci flow is extinct at time $t$ then it is also extinct at all later times $t'\geq t$, cf. Theorem 1.11 in \cite{KleLot17}).
        \end{definition}
        
        Concerning the notion of completeness of Ricci flow spacetimes, a good example to be kept in mind is that of a rotationally symmetric dumbbell metric on $S^3$: if the neck is pinched enough at the initial time, and if let $t_0>0$ denote the (first) singular time for the flow, then the associated Ricci flow spacetime (which is obtained, following \cite{KleLot17}, as a suitable limit of surgical Ricci flows in the sense of Perelman, as the surgery parameters become arbitrarily small) will have a point missing on the slice $\mM_{t_0}$, which however can only be approached through curves along which the ambient curvature of $\mM$ (in the sense above, i.e. for the spatial sub-bundle) gets arbitrarily large. 
        
        \
        
         Roughly speaking, our strategy here is a refinement of the one presented for proving Theorem 1.2 in \cite{CarLi19}. We start with a somewhat more general definition of the basic objects we will be dealing with.

    \begin{definition}\label{def:ReflexTriple}
    	We shall define a reflexive $n$-manifold to be a triple $(M,g,f)$ such that:
    	\begin{enumerate}
    		\item {$M$ is a manifold with boundary, of dimension $n\geq 3$, which we postulate to be orientable;}
    		\item {$g$ is a smooth Riemannian metric on $M$;}
    		\item {$f\in C^{\infty}(M,M)$ is an isometric involution of $(M,g)$ such that the following condition (henceforth denoted $(\star_{\text{sep}})$ holds:
    		the set of its fixed points, $\Fix(f)$, is a smooth hypersurface (possibly with boundary, in case $\partial M\neq\emptyset$) and for any $x\in M\setminus \Fix(f)$ we have that $x$ and $f(x)$ belong to different connected components of $M\setminus \Fix(f)$.}
    	\end{enumerate}	
    \end{definition}
    
    Let us stress that, consistently with the notational conventions we have stipulated, when we write \emph{manifold with boundary} we also allow for the special case when the boundary in question is actually empty. Also, it is understood that if $M$ is connected then $\Fix(f)\neq\emptyset$.
    
    In particular, we explicitly note that $M$ may be non-compact, and (be it compact or not) it is allowed to have a non-empty boundary. When we write that $\Fix(f)$ \emph{is a smooth hypersurface} we mean that $\Fix(f)$ is a codimension one smooth embedded submanifold; when $\partial M\neq\emptyset$ we agree that the boundary of such an hypersurface is a subset of $\partial M$ (hence it is in fact a closed submanifold of $\partial M$).
    
  However, in the special case when $M$ is a compact manifold without boundary, then it follows that:		
    			\begin{enumerate}
    			\item{if $M$ is connected, then $\Fix(f)$ is a closed hypersurface;}
    			\item{if $M$ is disconnected, then there exist $2d$ connected components 
    				\[
    				\left\{M^{\ast}_1, M^{\ast\ast}_1,\ldots, M^{\ast}_d, M^{\ast\ast}_d \right\}
    				\] such that $\left\{M^{\ast}_i, M^{\ast\ast}_i\right\}$ are isometric under $f$ for any $i=1,\ldots, d$, and any other component $M^{\bullet}$ in $M$ is such that the triple $(M^{\bullet}, g^{\bullet}, f^{\bullet})$ is itself a connected reflexive triple (with $g^{\bullet}, f^{\bullet}$ denoting the restriction of the metric $g$ and the involution $f$ to $M^{\bullet}$, respectively).}	
    		\end{enumerate}
Note that, by virtue of the requirement (iii) above, $\Fix(f)$ is always a separating submanifold, hence it is two-sided and orientable (since $M$ is postulated to be so).

         Given a compact topological space $K$ (which we will most often take to be a $k$-dimensional disk, $D^k$, for some $k\geq 0$), a compact manifold with boundary $X$ and a finite-dimensional family of Riemannian metrics $(g^s)_{s\in K}$ with each $g^s$ doubling, we will describe the family in question as a family of reflexive triples $(M^s,g^s,f^s)_{s\in K}$ where $M^s=M$ is the double of $X$, $g^s$ is (with slight, yet convenient abuse of notation) the corresponding doubled Riemmannian metric and $f^s=f\in C^\infty(M,M)$ is an involution of $M$ encoding the symmetry of the manifold in question, in the sense that $f^{\ast}g^s=g^s$ for all $s\in D$.
         We shall then consider the associated Ricci flow spacetimes  $\mM^s$ and employ them to ultimately design a homotopy deforming the given space (typically: a simplex) of metrics $K$ towards positive scalar curvature metrics. The key point of our work is that all operations we consider need to be, in a suitable sense, equivariant, so to allow from the previous construction(s) to derive by restriction a corresponding homotopy at the level of the doubling metrics $g^s$ on $X$. To properly formalise this idea, we start by introducing a convenient specification of the two definitions above.

        \begin{definition}\label{def:ReflexiveRicciFlow}
       We shall say that $(\mM, \mt, \partial_\mt,g, f)$ is a reflexive Ricci flow spacetime (respectively: a reflexive singular Ricci flow) if  $(\mM, \mt, \partial_\mt,g)$ is a Ricci flow spacetime (respectively: a singular Ricci flow) and, in addition, $f\in C^{\infty}(\mM,\mM)$ satisfies the following properties:
       \begin{enumerate}
           \item for any $t\geq 0$ the map $f$ restricts to a smooth map at the level of the time-t-slice, i.e. $f_t\in C^{\infty}(\mM_t, \mM_t)$, and
           the triple $(\mM_t, g_t, f_t)$ is a reflexive 3-manifold without boundary;
           \item for any $t', t''\geq 0$ if $x\in \mM_{t'}$ survives until time $t''$ then (in the sense of \cite{BamKle19}) so will the point $f_{t'}(x)$ and one has $f_{t''}(x(t''))=(f_{t'}(x))(t'')$; equivalently, denoted by $\Phi=\Phi(t',t'')$ the flow of the time vector field $\mt$ from $t', t''$, we have that whenever either composition is well-defined then the following diagram of smooth maps is commutative:
           		\[
  \begin{tikzcd}
   \mM_{t'} \arrow{r}{f_{t'}} \arrow{d}{\Phi}  & \mM_{t'} \arrow{d}{\Phi}   \\
    \mM_{t''} \arrow{r}{f_{t''}}  & \mM_{t''} .
  \end{tikzcd}
\]
       \end{enumerate}
        \end{definition}

       The following statement ensures the existence and uniqueness of singular Ricci flows emanating from compact Riemannian 3-manifolds. To phrase it we employ the terminology introduced in Section 4 of \cite{BamKle17} concerning \emph{continuous families} of structures (such as, e.g., smooth manifolds, smooth maps, tensor fields, Riemannian manifolds, Ricci flow spacetimes, singular Ricci flows etc.\ldots). For the specific purposes of the present work, it will be important to recall the content of Corollary 4.24 in \cite{BamKle19}: if $(M^s)_{s\in K}$ is a transversely continuous family of compact manifolds with boundary, and the parameter space $K$ is connected, then in fact the family in question comes from a fiber bundle with fiber $M$ (independent of $s$) and structure group $\Diff(M)$.


        \begin{theorem}\label{thm:ShortTime}(Theorem 4.1 and Theorem 4.2 in \cite{BamKle17}) Let $K$ denote an arbitrary topological space.
        For any continuous family of closed Riemannian 3-manifolds $(M^s,g^s)_{s\in K}$ there is a continuous family of singular Ricci flows $(\mM^s)_{s\in K}$ whose continuous family of time-0-slices $(\mM^s_0,g^s_0)_{s\in K}$ is isometric to $(M^s,g^s)_{s\in K}$. Furthermore, uniqueness holds in the following sense. Consider two families of singular Ricci flows $(\mM^{i,s})_{s\in K}, i=1,2$ and isometries $\phi_i: \cup_{s\in K} \mM^{i,s}_0\to\cup_{s\in K} M^s$: then there is an isometry $\Psi: \cup_{s\in K}\mM^{1,s}\to \cup_{s\in K}\mM^{2,s}$ of continuous families of singular Ricci flows with the property that $\phi_1=\phi_2\circ \Psi$ on $\cup_{s\in K}\mM^{1,s}_0$.
        \end{theorem}
        
        In the case of classical, smooth Ricci flows, the local uniqueness result implies that initial data with symmetry will evolve maintaining such a symmetry. The same conclusion actually holds true, when properly phrased, in our setting, which we
     encode in the following statement:

        \begin{theorem}\label{thm:UniquenessEquivariance}
        Given any reflexive triple $(M,g,f)$ we have that the unique singular Ricci flow it determines is a reflexive singular Ricci flow in the sense of Definition \ref{def:ReflexiveRicciFlow}.
        \end{theorem}
        
        \begin{proof}
        It follows from the work by Dinkelbach-Leeb \cite{DinLee09} that, for whatever choice of the surgical parameters, the Ricci flow with surgery emanating from reflexive 3-manifolds is a reflexive Ricci flow with surgery in the sense of Theorem 4.26 in \cite{CarLi19}. The corresponding singular Ricci flow is then uniquely determined as a subsequential limit of surgical Ricci flow by the uniqueness theorem stated above. Thus, one can follow the same construction developed by Kleiner-Lott in \cite{KleLot17} for the non-equivariant case, and indeed it is easily checked that such a limit process ensures the desired equivariance properties.
        \end{proof}
        
        Concerning the long-time behaviour of the flow we will employ the following basic statement, whose input is a purely topological condition on the background manifold.
        
        \begin{theorem}\label{thm:Extinction}(Theorem 2.16 in \cite{BamKle17-Diff})
        Let $M$ be a finite connected sum of spherical space forms and copies of $S^2\times S^1$ (which happens if and only if $M$ supports positive scalar curvature metrics), and let $K$ be a compact subset of Riemannian metrics on $M$. Then there exists a finite time $T_{\text{ext}}$ such that any singular Ricci flow $(\mM, \mt, \partial_\mt,g)$ with the property that $(\mM_0, g_0)$ is isometric to $(M,h)$ for some $h\in K$ is extinct at time $T_{\text{ext}}$.
        \end{theorem}
        
        Now, as we will see, one can actually modify the space-time track of the singular Ricci flow $(\mM_t)_{t\geq 0}$ (which is in practice defined over a compact time interval, because of the previous conclusion on the extinction time) to obtain an actual homotopy of metrics. In particular, we will ultimately prove the following result:
        
           \begin{theorem}\label{thm:MainHomotopy}
        Let $X$ be an orientable compact manifold with boundary supporting doubling metrics with positive scalar curvature. Let $k\in\mathbb{N}_{\ast}=\left\{1,2,\ldots\right\}$ and let $(h^s)_{s\in D^k}$ be a continuous family of doubling Riemannian metrics on $X$.
        Then there is a continuous family of doubling Riemannian metrics $(h^s_t)_{s\in D^k, t\in [0,1]}$ on $X$, such that:
        \begin{enumerate}
            \item for all $s\in D^k$: $h^s_0=h^s$, and $h^s_1$ has positive scalar curvature; 
            \item if $(X, h^s)$ has positive scalar curvature, then so will $(X, h^s_t)$ for all $t\in [0,1]$.
        \end{enumerate}
        \end{theorem}
        
        \begin{remark}\label{rem:TopClassification}
      As already mentioned in the introduction, recall from Theorem 1.1 in \cite{CarLi19} that $X$, an orientable compact manifold with boundary, supports doubling metrics with positive scalar curvature if and only if it takes the form of a finite connected sum of handlebodies (of any genus $\gamma\geq 0$), spherical space forms and copies of $S^2\times S^1$. Here the connected sums are understood in the most standard sense, i.e. at interior points (away from the boundary).
        \end{remark}
        
        \begin{remark}\label{rem:RephraseDoubling}
        One can equivalently rephrase the statement above by considering, as an input, a continuous family of reflexive manifolds $(M, h^s, f)$ and obtaining, as an output, a continuous family of reflexive manifolds  $(M, h^s_t, f)$ ending at positive scalar curvature metrics, and preserving such a positivity condition when satisfied for some $s\in D^k$ at $t=0$.
        \end{remark}
        
        It is clear that Theorem \ref{thm:MainHomotopy} implies Theorem \ref{thm:ContractDoubling}, as described below.
        
        \begin{proof}
        First of all, let us observe that the space $\met_D$ (when not empty) is contractible (in fact, it is a cone): indeed, it can be regarded as the cone of positive definite elements in the vector space of smooth $(0,2)$-tensors on the double of $X$ that are invariant under the involution $f$ in the sense that $f^{\ast}\sigma=\sigma$. Hence, in particular $\met_D$ is path-connected.

Keeping in mind Remark \ref{rmk:WeakHom}, let $h: (D^k,\partial D^k)\to (\met_D, \met_{R>0, D})$ be a given continuous map of pairs. Applying Theorem \ref{thm:MainHomotopy} we obtain a continuous homotopy of pairs $\overline{h}=\overline{h}(k)$ with the two properties (1) and (2) above. The existence of such a homotopy, when read at the level $k=0$ ensures that $CC(\met_D)\leq CC(\met_{R>0,D})$ and, when read at the level $k=1$ ensures that $CC(\met_D)\geq CC(\met_{R>0,D})$, where we have denoted by $CC(Z)$ the number in $\overline{\mathbb{N}}$ of path-connected components of a topological space $Z$. As a result, $\met_{R>0,D}$ is path-connected.
        
        That being said, the existence of a continuous homotopy of pairs $\overline{h}=\overline{h}(k)$ with the two properties (1) and (2) above implies $\pi_{k-1}(\met_D, \met_{R>0, D})=0$. So from the contractibility of $\met_D$ it then follows that  $\met_{R>0, D}$ has vanishing homotopy groups in all degrees, as we had to prove.
        \end{proof}
        
        We will derive Theorem \ref{thm:MainHomotopy} from another similar statement, where (roughly speaking) the positive scalar curvature condition is replaced, throughout, by that of PSC-conformality.

        \begin{definition}\label{def:ReflexivePSCconf}
          A compact reflexive 3-manifold with boundary $(M,g,f)$ is called reflexively PSC-conformal if there is a smooth positive function $w\in C^{\infty}(M)$ such that:
          \begin{enumerate}
              \item [(0)] $w$ is $f$-equivariant, i.e. $w\circ f=w$; 
              \item [(1)] $w^4 g$ has positive scalar curvature;
              \item [(2)]  $w$ restricted to each boundary component of $M$ is constant;
              \item [(3)] every boundary component of $(M, w^4g)$ is totally geodesic and isometric to the standard round 2-sphere (of unit radius).
          \end{enumerate}
        \end{definition}
        Note that when $\partial M=\emptyset$ this definition is simply singling out manifolds of positive Yamabe invariant, or (which is the same, by the so-called trichotomy theorem) those manifolds such that the first eigenvalue of the conformal Laplace operator has positive sign. In the case of compact manifolds with boundary we impose additional, and fairly restrictive, boundary conditions on the conformal factor. The notion above is studied in Appendix \ref{sec:PSCconf}, where we collect (properly rephrased and proved in the equivariant setting) various results that are ancillary to some of the constructions we need to perform.
        
        Here is the anticipated statement:
        
        \begin{theorem}\label{thm:MainHomotopyConfVersion}
        Let $M$ be a compact 3-manifold diffeomorphic to a finite connected sum of spherical space forms and copies of $S^2\times S^1$, and let $f\in C^{\infty}(M,M)$ be an involutive diffeomorphism satisfying the assumption $(\star_{\text{sep}})$. Let
        $k\in\mathbb{N}_{\ast}=\left\{1,2,\ldots\right\}$ and let $(g^s)_{s\in D^k}$ be a continuous family of Riemannian metrics, such that for all $s\in D^k$ one has that $(M, g^s,f)$ is a reflexive 3-manifold.
           Let $K\subset D^k$ be a closed subset with the property that $(M, g^s)$ is a Riemannian manifold of positive scalar curvature for all $s\in K$. 
           
           Then there is a continuous family of Riemannian metrics $(g^s_t)_{s\in D^k, t\in [0,1]}$ on $M$, such that:
        \begin{enumerate}
            \item [(0)] for all $s\in D^k$ and $t\in [0,1]$ one has that $(M, g^s_t, f)$ is a reflexive 3-manifold;
            \item [(1)] for all $s\in D^k$ one has that $g^s_0=g^s$, and $g^s_1$ is reflexively PSC-conformal; 
            \item [(2)] if $s\in K$ then $(M, g^s_t, f)$ is reflexively PSC-conformal for all $t\in [0,1]$.
        \end{enumerate}
       \end{theorem}
       
       
       \begin{proof}(of Theorem \ref{thm:MainHomotopy} given Theorem \ref{thm:MainHomotopyConfVersion})
       
       Let us define
       \[
       D^k_{NN}:=\left\{ s\in D^k \ : \ R_{h^s}\geq 0 \right\}
       \]
       namely the subset of metrics having non-negative scalar curvature at all points. Of course, this is a closed subset of $D^k$ hence it is compact. By our topological assumption, the manifold $M$ obtained as double of $X$ does not support Ricci-flat metrics (indeed, it is a connected sum of spherical space forms and copies of $S^2\times S^1$) thus it follows from the evolution equation for the scalar curvature under Ricci flow and continuous dependence on initial data (as given, for instance, as Theorem A in \cite{BauGueIse20}) that one can find $\tau>0$ such that the Ricci flow with initial datum $h^s$ is well-defined (as a smooth classical flow) up to time $2\tau$ for all $s\in D^k$, and, in addition $R_{h^s(t)}>0$ for all $0<t\leq \tau$ whenever $s\in D^k_{NN}$. 
       
       At that stage, after this preliminary deformation, we wish to apply Theorem \ref{thm:MainHomotopyConfVersion} to the set of metrics $(h^s(\tau))_{s\in D^k}$, with $K=D^k_{NN}$.
       Thereby, we obtain a continuous family of Riemannian metrics $(\tilde{h}^s_{t})_{s\in D^k, \ t\in [\tau,1]}$ such that (among other things) $\tilde{h}^s_{t}$ is reflexively PSC-conformal whenever 
       \[
       (s,t)\in K_{\Pi}:=(K\times [\tau,1]) \cup (D^k\times \left\{1\right\}).
       \]
       Set then
       \[
       h^s_t=\begin{cases}
       h^s(t) & \text{if} \ t\in [0,\tau] \\
       \tilde{h}^s_{t} & \text{if} \ t\in [\tau,1].
    \end{cases}
       \]
        Recalling the first part of property (1) of Theorem \ref{thm:MainHomotopyConfVersion}, we see at once that this is indeed a continuous family of Riemannian metrics on $M$. 
      Also note that, in applying the theorem we have $f^s=f$ for all $s\in D^k$ so $(M, h^s_t, f)$ is indeed a reflexive 3-manifold for all $(s,t)\in D^k\times [0,1]$.

       Through a standard covering argument, we can find
           finitely many open sets $V_1,\ldots, V_L$ in $D^k\times [0,1]$ whose union covers  $K_{\Pi}$ while being disjoint from $D^k\times\left\{0\right\}$, and associated smooth (positive) conformal factors $w_1, \ldots, w_L$ respectively, such that $R_{w^4_i h^s_{t}}>0$ whenever $(s,t)\in V_i$. Let then $V_0=(D^k\times [0,1])\setminus K_{\Pi}$ and set $w_0=1$ identically on $M$. All conformal factors in question are even, i.e. they are equivariant with respect to $f$.
           
           Considering a partition of unity (through smooth cutoff functions $\varphi_0, \varphi_1, \ldots, \varphi_L)$ subordinate to the open cover $V_0, V_1 \ldots, V_L$ of $D^k$, we can take convex combinations of the conformal factors on the overlappings so to obtain (thanks e.g. to the standard results in Appendix B of \cite{CarLi19}) a continuous family of positive functions $w: D^k\times [0, 1]\to C^{\infty}(M)$ such that $R_{w^4h^s_t}>0$ for all $(s,t)\in K_{\Pi}$. Furthermore, our construction ensures that $h^s_0=h^s$ for all $s\in D^k$. 
         
         Rephrasing things in terms of restricted doubling metrics on $X$ completes the proof. 
       \end{proof}
        
        Hence, our initial task (i.e. proving Theorem \ref{thm:ContractDoubling}) boils down to discussing how to actually design such homotopies. For reasons that will become clear in the sequel, it turns out that such `classical' homotopies are in turn derived from certain \emph{partial homotopies}, which in our special equivariant context can be defined as follows. 
        
        In short, when one works with a finite dimensional simplex of initial data for the (singular) Ricci flow the corresponding time-t-slices $\mM^s_t$ (for fixed $t\geq 0$) will have different topology as one varies the parameter $s$ and thus the homotopies one can construct by backward-in-time induction (essentially following the idea in \cite{Mar12}, transposed to the very different setting of non-surgical flows) will in general be defined on domains that vary with $s$: this issue is solved by considering a fine simplicial subdivision of $D^k$ and allowing for topological changes as one transitions from one (say, top-dimensional) simplex to another, through a lower-dimensional simplex. These transitions roughly correspond to adding or removing regions of very high curvature, which are indeed those that \emph{may or may not} be in the domain of the homotopy corresponding to a certain simplex. Further, to effectively implement these ideas and operate with partial homotopies it turns out to be very convenient to slightly deform the metric to become exactly round in the regions of high curvature (a statement to be interpreted differently depending on the specific $\kappa$-solution which models the region we are considering). 
        
        In order to formalise this strategy, the first notion we need to recall (or, rather, recast in our equivariant setting) is that of \emph{spherical structure} which, in turn, in instrumental to defining the fundamental concept of $\mR$-structure.
        
        \begin{definition}\label{def:S-structure}
               
               Given a manifold $M$, possibly with boundary, and an
    involutive diffeomorphism $f\in C^{\infty}(M,M)$ satisfying the assumption $(\star_{\text{sep}})$, a reflexive spherical structure on an $f$-invariant set $U\subset M$ (by which we mean that $f(U)=U$) is a smooth fiber bundle structure on an open and dense $f$-invariant set $U'\subset U$, whose fibers are diffeomorphic to $S^2$ and equipped with a smooth fiberwise metric of constant curvature 1 such that the following statements hold:
               \begin{enumerate}
                   \item for any $x\in U$ there exist a neighborhood $V=V(x)\subset U$ and an $O(3)$ action $\zeta$ such that $\zeta_{|{V\cap U'}}$ preserves all $S^2$-fibers and acts effectively and isometrically on them and, in addition: 
                   \begin{enumerate}
                       \item $V$ is $f$-invariant, i.e. $f(V)=V$;
                       \item $f(\zeta(\rho, y))=\zeta(\rho, f(y))$ for all $\rho\in O(3), y\in V$;
                   \end{enumerate}
                   \item all orbits on $V\setminus U'$ are not diffeomorphic to spheres.
               \end{enumerate}
        Any local action satisfying the above properties will be called compatible with the spherical structure in question.

        Furthermore, in this setting we will say that the metric $g$ is compatible with the spherical structure $\mS$ on $U$ if for every $x\in U$ there exist an $f$-invariant set and an $O(3)$-action defined there that is compatible with $\mS$ and isometric with respect to $g$. With slight abuse of language, we will also say, in this same case, that $\mS$ is compatible with $g$.       
                  \end{definition}

        In the setting of the previous definition we will call $U$ the domain of such a spherical structure, henceforth denoted by $\text{domain}(\mS)$, any fiber in $U'$ a regular fiber of $\mS$, and any fiber in $U\setminus U'$ a singular fiber of $\mS$ instead.

        It is not difficult to write down the most general form of a Riemannian metric $g$ that is compatible with the standard spherical structure on the open cylinder $S^2\times (u,v)$. In the setting of Definition \ref{def:S-structure}, so with an isometric involution, we have the following counterparts, where for the second one we set $f_0(x)=(x_1, x_2, -x_3)$.
        
        \begin{lemma}\label{lem:NormalForm}
        \begin{enumerate}
        \item{Let $r_0>0$ and consider the standard spherical structure $\mS_0$ on $M=S^2\times (-r_0,r_0)$, and the map $f(x,r)=(x,-r)$.  If a Riemannian metric $g$ is such that $(M,g,f)$ is a reflexive manifold, then $g$ is compatible with $\mS_0$ if and only if it takes the form
        \[
        g= a^2(r)g_{S^2}+b^2(r)dr^2+\sum_{i=1}^3 c_i(r)(dr\otimes (dx^i)^{\#}+(dx^i)^{\#}\otimes dr)
        \]
        for smooth functions $a, b, c_1, c_2, c_3 \in C^{\infty}(-r_0, r_0)$ with $a, b$ positive and even, and $c_1, c_2, c_3$ odd.}
        \item{Let $r_1<r_2$ and consider the standard spherical structure $\mS_0$ on $M=S^2\times (r_1,r_2)$, and the map $f(x,r)=(f_0(x),r)$.  If a Riemannian metric $g$ is such that $(M,g,f)$ is a reflexive manifold, then $g$ is compatible with $\mS_0$ if and only if it takes the form
        \[
        g= a^2(r)g_{S^2}+b^2(r)dr^2+\sum_{i=1}^2 c_i(r)(dr\otimes (dx^i)^{\#}+(dx^i)^{\#}\otimes dr)
        \]
        for smooth functions $a, b, c_1, c_2 \in C^{\infty}(-r_0, r_0)$ with $a, b$ positive.}
        \end{enumerate}
           \end{lemma}
           
           \begin{proof}
              It suffices to look at the most general form in the non-equivariant case (cf. Lemma 5.5 in \cite{BamKle19}) and single out the subclasses that are invariant under the action of the involution $f$. 
           \end{proof}
      
      \begin{remark}\label{rem:PropertiesSstruct}  
        We briefly collect a few basic facts and notions about (general, i.e. non-equivariant) spherical structures:
        \begin{enumerate}
            \item any singular fiber is either a point or diffeomorphic to $\mathbb{R}\mathbb{P}^2$;
            \item if $M$ has dimension 3, then $\zeta$ is locally conjugate to one of the following models equipped with the standard $O(3)$ action: $S^2\times (-1,1)$,  $S^2\times (-1,1)/\mathbb{Z}_2$, $B^3$, $S^2\times [0,1)$;
            \item if $M$ has dimension 3, then spherical structures $\mS$ such that $\text{domain}(\mS)=M$ have been fully classified in Lemma 6.1 of \cite{BamKle19}.
         \end{enumerate}
        We further explicitly note that it follows from item (2) above that \emph{singular} fibers are isolated, i.e. they cannot accumulate.
               \end{remark}

        Now, it is clear that the requirement of having a \emph{reflexive} spherical structure places additional restrictions, and thus limits the actual number of local models for the action $\zeta$, and the corresponding list of cases to be contemplated in item (3) of Remark \ref{rem:PropertiesSstruct}.
        
        To formalise this comment, we first recall from \cite{CarLi19} the classification of totally geodesic surfaces in a product cylinder $S^2\times I$, which can only be either slices of the form $S^2\times \left\{t\right\}$ or product submanifolds of the form $S^1\times I$ (see Section 4 of \cite{SouVan12}). In fact, this conclusion holds true in greater generality. If $M$ is a manifold diffeomorphic to $S^2\times I$, $f$ is an
    involutive diffeomorphism of $M$ satisfying the assumption $(\star_{\text{sep}})$, for which $\mS_0$, the standard spherical structure, is a reflexive spherical structure, and $g$ is a metric compatible with $\mS_0$, then we can find a diffeomorphism $\phi: S^2\times [-r_0,r_0] \to M$ such that the pull-back of $g$ is a metric of the form $h^4(g_{S^2}+dr^2)$ for some smooth, positive function $h$: hence, since the results in \cite{SouVan12} actually provide a classification of \emph{totally umbilic} hypersurfaces, and the class of such hypersurfaces is invariant under conformal diffeomorphisms, we can again conclude that $\Fix(f)$ is either horizontal or vertical (in the sense above), whence the involution $f$ is uniquely determined by its fixed point set as per Corollary 4.8 in \cite{CarLi19}. 
        Thus, exploiting this fact and the list of (non-equivariant) local models for the action allows to derive the three following conclusions.

  \begin{lemma}\label{lemma.SingFiberReflexive}
   Suppose $(M,g,f)$ is a reflexive 3-manifold and let $\mS$ be a reflexive spherical structure compatible with $g$. Suppose $\mS$ contains a singular leaf $\Gamma$ diffeomorphic to $\mathbb{R}\mathbb{P}^2$. Then $\Gamma\cap \Fix(f)=\emptyset$. 
  \end{lemma}
  \begin{proof}
   Since $\Gamma$ is diffeomorphic to $\mathbb{R}\mathbb{P}^2$, a local neighborhood of $\Gamma$ is $\left (S^2\times (-1,1)\right)/\Z_2$, equipped with the standard $O(3)$ action. Note here that the $\Z_2$ action is defined by $(x,t)\to (-x,-t)$, and $\Gamma$ is obtained by $\left(S^2\times \{0\}\right)/\mathbb{Z}_2$. Thanks to the above discussion, $\Fix(f)$ can either be $\left(S^2\times \{t\}\right)/\Z_2$, or $\left(S^1\times (-1,1)\right)/\Z_2$, where $S^1$ stands for a equator in $S^2$. In the first case, we see that $t\ne 0$, as $\mathbb{R}\mathbb{P}^2$ cannot be a separating surface in an orientable 3-manifold (since it is not orientable). The second case cannot happen either, as $\left(S^1\times (-1,1)\right)/\Z_2$ is not separating.
  \end{proof}
  
  So, in short, when one needs to deal with singular fibers diffeomorphic to $\mathbb{R}\mathbb{P}^2$ those are away from the fixed locus of the symmetry, which allows to typically perform most/all constructions as in the non-equivariant case. Concerning the singular fibers that are points the above conclusion, of Lemma \ref{lemma.SingFiberReflexive}, cannot possibly hold true, yet we have the following replacement.
  
  \begin{lemma}\label{lemma.SingFiberReflexivePointCase}
   Suppose $(M,g,f)$ is a reflexive $PSC$ 3-manifold and $\mS$ a reflexive spherical structure compatible with $g$. Suppose $\mS$ contains a singular leaf that is a point, $\Gamma=\left\{x\right\}$. Then either $\Gamma\cap \Fix(f)=\emptyset$ or, instead,  there exists an open neighborhood $U$ of $x$, diffeomorphic to $B^3$, such that $U\setminus \left\{x\right\}$ is foliated by (regular) spherical fibers, each meeting $\Fix(f)$ orthogonally along a circle.
  \end{lemma}
  
  We can then derive the classification of equivariant local actions, and of \emph{orientable} reflexive 3-manifolds that admit spherical structure whose domain is the whole manifold itself. 
        
        \begin{lemma}\label{lem:ReflSstructClassification}
            Let $(M,g,f)$ be as in Definition \ref{def:S-structure}, and assume $M$ to be connected, orientable and have dimension equal to 3. Then:
            \begin{enumerate}
                \item if $\Fix(f)\cap V\neq\emptyset$ (for $V$ connected and $f$-invariant) then $\zeta$ is locally conjugate to one of the following models equipped with the standard $O(3)$ action: $S^2\times (-1,1)$, $B^3$, $S^2\times [0,1)$;
                \item (if $\Fix(f)\neq\emptyset$ and) the spherical structure $\mS$ satisfies $\text{domain}(\mS)=M$ then one of the following cases holds:
                \begin{enumerate}
                    \item $\mS$ only has regular fibers, $M$ is diffeomorphic to $S^2\times (0,1)$ (respectively: $S^2\times [0,1]$) and $\Fix(f)= S^2\times \left\{1/2\right\}$ or $\Fix(f)= S^1\times (0,1)$ (respectively: $\Fix(f)= S^2\times \left\{1/2\right\}$ or $\Fix(f)= S^1\times [0,1]$;)
                    \item $\mS$ only has regular fibers, $M$ is diffeomorphic to $S^2\times [0,1)$ and $\Fix(f)= S^1\times [0,1)$;
                    \item $\mS$ only has regular fibers, $M$ is diffeomorphic to $S^2\times S^1$ and $\Fix(f)= S^1\times S^1$;
                    \item $\mS$ has exactly one singular fiber, which is a point, $M$ is diffeomorphic to $B^3$ (respectively: $D^3$) and $\Fix(f)=B^2$ (respectively: $D^2$);
                    \item $\mS$ has exactly two singular fibers, which are points, $M$ is diffeomorphic to $S^3$ and $\Fix(f)=S^2$.
                \end{enumerate}
            \end{enumerate}
        \end{lemma}
    We note that, in each of the cases above, also the reflection map $f$ is completely determined, although we limited ourselves to indicate $\Fix(f)$ for the sake of expository convenience, leaving the (very intuitive) additional specifications to the reader. Also note that the orientability assumption is totally inessential (i.e. one could equally well present a classification covering the non-orientable cases), but we have decided to stick here to the setup of our main theorems.

            \begin{definition}\label{def:FlowPres}
            Let $U$ be an open subset of a (smooth) manifold with boundary $M$, and let $\mS$ be a spherical structure on $U$. We will say that a (smooth) vector field $X$ preserves the spherical structure $\mS$, or equivalently that $\mS$ is preserved by $X$, if the flow of such a vector field sends fibers to fibers, sends regular (respectively: singular) fibers to regular (respectively: singular) fibers and preserves the fiberwise metric on the regular fibers (as understood in Definition 5.1 of \cite{BamKle19} for the non-equivariant case, and as per Definition \ref{def:S-structure} for the reflexive case).
            \end{definition}
            
            \begin{remark}\label{rem:EquivCharPreservingFlow}
            It follows from Definition 5.4 of \cite{BamKle19} and Definition \ref{def:S-structure} that a vector field $X$ preserving a spherical structure $\mS$ acts on the fibers endowed with a compatible metric $g$ on $U\subset M$ as a homothety: if $\mO_1$ and $\mO_2$ are regular fibers then $\Phi^*(g_{|\mO_2})=\lambda^2 g_{|\mO_1}$, for a suitable $\lambda=\lambda_{12}>0$, and where $\Phi$ denotes the flow of the vector field $X$.
            \end{remark}

     At this stage we observe that one can easily refer the notions above to the case a singular Ricci flow $\mM$, essentially by looking at the level sets of the time function, $\mM_t$, and at the sub-bundle $\ker(d\mt)$. We can then proceed and enrich the previous structure, by embedding it into a more refined one.
        
        \begin{definition}\label{def:PrelimR} We shall define an $\mR$-structure on a singular Ricci flow $\mM$ as a tuple $\mR=(g',\partial'_{\mt}, U_{S2}, U_{S3}, \mS)$ consisting of a smooth metric $g'$ on $\ker d\mt \subset T\mM$, a vector field $\partial_{\mt'}$ on $\mM$ with $\partial_{\mt'}\mt=1$, open subsets $U_{S2}, U_{S3}$ in $\mM$ with $U_{S3}\setminus U_{S2}$ open, and a spherical structure $\mS$ on $U_{S2}$ such that for all $t\geq 0$:
        \begin{enumerate}
            \item $g'_t$ is compatible with $\mS$;
            \item $\partial'_{\mt}$ preserves $\mS$ (in the sense of Definition \ref{def:FlowPres});
            \item $U_{S2}\cap \mM_t$ is a union of (regular and singular) fibers of $\mS$;
            \item 
            \begin{enumerate}
            \item $U_{S3}\cap \mM_t$ is a union of compact components of $\mS_t$ where $g'_t$ has constant curvature;
            \item $U_{S3}$ is invariant under the forward flow of the vector field $\partial'_{\mt}$, and such a flow, when restricted to every component of $U_{S3}\cap \mM_t$, is a homothety with respect to $g'$, whenever defined.
            \end{enumerate}
        \end{enumerate}
        In the setting above, we will say that $U_{S2}\cup U_{S3}$ is the support of the $\mR$-structure in question.
        \end{definition}
        
        \begin{definition}\label{def:R-structure}
                        Given a reflexive singular Ricci flow $\mM$ (as per Definition \ref{def:ReflexiveRicciFlow}) we shall say that an $\mR$ structure $(g',\partial'_{\mt}, U_{S2}, U_{S3}, \mS)$ is a \emph{reflexive $\mR$-structure} if the following properties are true:
               \begin{enumerate}
                   \item $(f_{t})^*g'_t=g'_t$ for all $t\geq 0$;
                   \item $(f_t)_*\partial'_{\mt}=\partial'_{\mt}$ for all $t\geq 0$;
                   \item $f(U_{S2})=U_{S2}$;
                   \item $f(U_{S3})=U_{S3}$;
                   \item $\mS$ is a reflexive spherical structure in the sense of Definition \ref{def:S-structure}.
               \end{enumerate}
                   \end{definition}
                   
                     Recall, finally, that a suitable notion of (transverse) continuity has been defined both for spherical structures and, then, for $\mR$-structures in Section 5 of \cite{BamKle19}, which obviously applies (as a special case) when additional equivariance constraints are also in play. That being said, we will prove in Section \ref{sec:Rstruct} the existence of an $\mR$-structure associated to any given reflexive singular Ricci flow, and in fact that such an association can be made for families of flows and determines a transversely continuous correspondence for the objects in question. The reader is referred to Theorem \ref{thm:ExistR-struct} for a precise statement. 
                     
                     Giving that for granted, we proceed in our discussion and discuss how such reflexive $\mR$-structures will be employed in the proof of Theorem \ref{thm:MainHomotopyConfVersion}. Basically, like we anticipated above, the preliminary construction of a reflexive $\mR$-structure for a reflexive singular Ricci flow allows to design partial homotopies, which in particular will yield an actual homotopy emanating from the initial data for the flow and ending at positive scalar curvature metrics, while preserving such a subspace of metrics along the deformation. To formalise these concepts, it is convenient to introduce two more definitions.

        \begin{definition}\label{def:MetricDeformation}
              Let $K$ be a topological space. A reflexive metric deformation over $K$ is a triple $(Z, g_{s,t}, f)$ consisting of reflexive 3-manifolds with  boundary, and such that the boundary components are spheres,
              where $(g_{s,t})_{s\in K, t\in [0,1]}$ is a continuous family of Riemannian metrics and such that for all $s\in K$ the Riemannian manifold $(Z, g_{s,1}, f)$ is 
              reflexively PSC-conformal.
        \end{definition}
        

        \begin{definition}\label{def:PartialHom}
        \emph{First part: setup.}
        Let $K$ be (the geometric realisation of) a finite-dimensional simplex $\mK$. Consider a fiber bundle $\pi: E\to K$ whose fibers are smooth, compact Riemannian 3-manifolds, which we shall regard as a (transversely) continuous family of reflexive Riemannian 3-manifolds $(M^s,g^s, f^s)_{s\in K}$.
        Considered the continuous family of reflexive singular Ricci flows $(\mM^s, \mt, g^s,\partial^s_{\mt}, f^s)_{s\in K}$ whose time-0-slices are isometric to $(M^s, g^s, f^s)$, let $(\mR^s_{s\in K}$ be an associated, transversely continuous family of reflexive $\mR$-structures, which we denote
        \[
        \mR^s=(g'^{,s},\partial'^{,s}_{\mt}, U^{s}_{S2},  U^{s}_{S3},\mS^s)_{s\in K}.
        \]

        \emph{Second part: core.}
        For every simplex $\sigma\in\mK$ consider a reflexive metric deformation \[
        (Z^{\sigma}, (g^{\sigma}_{s,t})_{s\in\sigma, t\in [0,1]}, f^{\sigma})
        \]and, for given $T\geq 0$, a transversely continuous family of reflexive embeddings $(\psi^{\sigma}_s: Z^{\sigma}\to \mM^{s}_T)_{s\in\sigma}$ in the sense that $\psi^{\sigma}_s\circ f^{\sigma}= f^s \circ \psi^{\sigma}_s$. We will call 
        \[
        \left\{(Z^{\sigma}, (g^{\sigma}_{s,t})_{s\in\sigma, t\in [0,1]}, f^{\sigma}, (\psi^{\sigma}_s)_{s\in\sigma})\right\}_{\sigma\in \mK}
        \]
        a reflexive partial homotopy (at time $T$, for the family of $\mR$-structures $(\mR^s)_{s\in K}$) if the following properties hold:
        \begin{enumerate}
            \item for all $s\in\sigma\in \mK$ we have $(\psi_s^{\sigma})^{\ast}g'^{,s}_T=g^{\sigma}_{s,0}$;
            \item for all $s\in\tau\subset \sigma$ for $\tau,\sigma\in \mK$ we have $\psi_s^{\sigma}(Z^{\sigma})\subset \psi_s^{\tau}(Z^{\tau})$
            \item for all $s\in\tau\subset \sigma$ for $\tau,\sigma\in \mK$ we have and all $t\in [0,1]$ we have $((\psi_s^{\tau})^{-1}\circ \psi_s^{\sigma})^{\ast}g^{\sigma}_{s,t}=g^{\sigma}_{s,t}$;
            \item for all $s\in\tau\subset \sigma$ for $\tau,\sigma\in \mK$ and for the closure $\mC$ of each component of $Z^{\tau}\setminus ((\psi_s^{\tau})^{-1}\circ \psi_s^{\sigma})(Z^{\sigma})$ one or both of the following statements are true:
            \begin{enumerate}
                \item  $\psi^{\tau}_s(\mC)\subset U^s_{S2}$ and $\psi^{\tau}_s(\mC)$ is a union of spherical fibers, invariant (as a set) under the action of $f^s$; in addition, for every $t\in [0,1]$ the metric $(\psi^{\tau}_s)_{\ast}g^{\tau}_{s,t}$ restricted to $\psi^{\tau}_{s}(\mC)$ is compatible with the restricted spherical structure;
                \item $\partial \mC=\emptyset, \psi^{\tau}_s(\mC)\subset U^{s}_{S3}$ and for every $s'\in\tau$ in a suitable neighborhood of $s$ the metric $g^{\tau}_{s',t}$ restricted to $\mC$ is a multiple of $g^{\tau}_{s',0}$ for all $t\in [0,1]$. 
            \end{enumerate}
                  \item for every $\sigma\in \mK$ and every component $\sigma$ of $\partial Z^{\sigma}$ the image $\psi^{\sigma}_s(\Sigma)$ is a regular fiber of $\mS^s$ for all $s\in\sigma$; moreover, there is an $\e>0$, independent of $s$, such that for all $t\in[0,1]$ the metric $(\psi^{\sigma}_s)_{\ast}g^{\sigma}_{s,t}$ is actually compatible with $\mS^s$ on an $\e$-collar neighborhood of $\psi^{\sigma}_s(\Sigma)$ inside $\psi^{\sigma}_s(Z^{\sigma})$.
        \end{enumerate}
        We say that the partial homotopy in question is \emph{PSC-conformal over $s\in K$} if for any simplex $\sigma\in \mK$ containing $s$ and for any $t\in [0,1]$ the reflexive manifold $(Z^{\sigma},g^{\sigma}_{s,t},f^{\sigma})$ is reflexively PSC-conformal. Lastly, if $Z^{\sigma}=\emptyset$ for all $\sigma\in \mK$, then the partial homotopy is called \emph{trivial}.
                   \end{definition}
                   
                   We note that a (reflexive) partial homotopy allows to derive an actual (reflexive) homotopy whenever  
                   \[
                   \psi^{\sigma}_s \  \text{is surjective for any simplex} \ \ \sigma \hspace{7mm} \ (\ast)
                   \]
                   as precisely explained in Proposition 7.5 in \cite{BamKle19} (whose straightforward argument can be transplanted unmodified to our equivariant setting). As a convenient yet imprecise terminology, we will call \emph{surjective} a partial homotopy for which property $(\ast)$ holds true. Basically, one can just set 
                   \[
                   h^s_t=(\psi^{\sigma}_s)_{\ast} g^{\sigma}_{s,t}
                   \]
                   where $\sigma\in\mK$ is just \emph{any} simplex contaning $s\in K$, thanks to the consistency conditions (2) and (3) above. 
                   
                   Roughly speaking, by considering the singular Ricci flow spacetime generated by given data (cf. statement of Theorem \ref{thm:MainHomotopy}), and in particular by `moving backward in time' so starting after the extinction time $T_{ext}$ and then considering the portion of the flow happening between times $T$ and $T-\Delta T$ we will be able to build partial homotopies at times 
                   \[
                   T, T-\Delta T, \ldots, T-m\Delta T, 0
                   \]
                   that happen on (suitably regular) subdomains of $\mM_T, \mM_{T-\Delta T}, \ldots,\mM_{T-m\Delta T}, \mM_{0}=M$ which depend, in a rather non-trivial fashion, both from $j=0,1,\ldots, m$ and from $s$, but essentially only miss certain high curvature regions, at worst. In particular, the construction can be designed so that the partial homotopy at time $0$ is surjective (in the sense we just defined) and this in fact it allows to obtain a surjective reflexive partial homotopy, which is what we actually need.

        Hence, we have reduced our initial question to the construction of reflexive partial homotopies, which is the object of Section \ref{sec:PartialHom}. Roughly speaking, the domains of the partial homotopies vary, as we transition from time $T-j
        \Delta T$ to time $T-(j+1)
        \Delta T$ through \emph{two} types of moves: either by an equivariant enlargement of a neck-type domain (an operation which we study in Appendix \ref{app:EquivExtHom}) or by equivariant removal of a 3-disk in a region close to a Bryant soliton near its tip (which is the object of Appendix \ref{app:EquivDiskRemov} instead).

        This outline being given, we will now get somewhat more technical about some of the notions introduced above, starting with the rounding process.

        \section{Construction of reflexive round structures}\label{sec:Rstruct}

        
        In order to state the existence result for families of reflexive $\mR$-structures we need to recall here, for the first time, some basic definitions of essential \emph{scales} in the study of Ricci flows:
        \begin{itemize}
            \item (curvature scale) given a manifold $M$ and a Riemannian metric $g$ on $M$ we shall say that $\rho: M\to (0,\infty]$ is a curvature scale function if there exists a constant $C>0$ such that
            \begin{equation}\label{eq:rhoCurvScale}
            C^{-1}\rho^{-2}\leq |\text{Rm}_g|\leq C\rho^{-2}
            \end{equation}
            on all of $(M,g)$; given a Ricci flow spacetime $\mM$ we will similarly say that $\rho:\mM\to (0,\infty]$ is a curvature scale function if \eqref{eq:rhoCurvScale} holds on all of $\mM$ and (consistently with Definition \ref{def:SingularRicciFlow}) $\text{Rm}$ always refers to the metric $g$ on the horizontal sub-bundle, i.e. we work on the time-t-slices. Following Section 3.4 in \cite{BamKle19} we agree, throughout this article, to make a special choice of the curvature scale function, given by
            \begin{equation}\label{eq:rhoCurvScaleSpecification}
            \rho(x)=\min \left\{(R_g(x))^{-1/2}_{+}, \ 10|\text{Rm}_g(x)|^{-1/2}\right\}
            \end{equation}
            where the factor $10$ is conveniently introduced so to ensure that if $x\in \mM$ satisfies the $\e$-canonical neighborhood assumption for $\e$ small enough (that is to say: for $\e<\e_0$ where $\e_0$ is a universal constant), then $\rho(x)=(R_g(x))^{-1/2}$.  
            \item (initial scale) given a manifold $M$ and a Riemannian metric $g$ on $M$ we shall define the initial condition scale $r_{\text{initial}}(M,g)$ as follows
            \begin{equation}\label{eq:InitialScale}
              r_{\text{initial}}(M,g)=\min \left\{\inf_{M}|\text{Rm}_g|^{-1/2}, \inf_{M}|\nabla \text{Rm}_g|^{-1/3}, \ \text{inj rad}(M,g) \right\}  
            \end{equation}
            where $\text{inj rad}(M,g)$ stands for the injectivity radius of $(M,g)$; if $t\geq 0$ and $\mM$ is a Ricci flow spacetime, then we will in particular write $r_{\text{initial}}(\mM_t,g_t)$ in order to refer to the time-t-slice of the spacetime in question. We recall, parenthetically, that this quantity naturally comes up for the following purpose: given any $\e>0$ there exists a smooth function $r_{\text{can},\e}: \mathbb{R}_{+}\times [0,\infty)\to\mathbb{R}_{+}$ such that for any $T\geq 0$ and any singular Ricci flow $\mM$ it holds that $\mM$ satisfies the $\e$-canonical neighbrhood assumption below scale
            \[
            r_{\text{can},\e}(r_{\text{initial}}(\mM_0,g_0), T) \ \text{on the time interval} \ [0,T]
            \]
            (so this can be regarded, a posteriori, as a specification of part (3) in Definition \ref{def:SingularRicciFlow}). With slight notational abuse, it will be convenient to regard $r_{\text{can},\e}:\cup_{s}\mM^s\to\mathbb{R}$ i.e. if $x\in \mM^s_t$ we will from now onwards write $r_{\text{can},\e}(x)$ in order to mean $r_{\text{can},\e}(r_{\text{initial}}(\mM^s_0,g^s_0),t)$.
        \end{itemize}
        
        We can then proceed with the key statement of this section, where (as in Theorem \ref{thm:ShortTime}) $K$ stands for an arbitrary topological space.
        
        \begin{theorem}\label{thm:ExistR-struct}  
        For any $\delta>0$ there is a constant $C=C(\delta)$ and a continuous, non-increasing function $r_{\text{rot},\delta}:\mathbb{R}_{+}\times [0,\infty)\to\mathbb{R}_{+}$ such that the following holds.
        
        Consider a transversely continuous family $(\mM^s)_{s\in K}$ of \emph{reflexive} singular Ricci flows, as per Definition \ref{def:ReflexiveRicciFlow}.
        Then there is a transversely continuous family $(\mR^s)_{s\in K}$ of \emph{reflexive} $\mR-$structures, as per Definition \ref{def:R-structure} (with $\mR^s$ an $\mR$-structure for $\mM^s$, for any $s\in K$), such that for any $s\in K$ the following assertions hold true:
        \begin{enumerate}
            \item (large curvature regions) $\mR^s$ is supported on the set 
            \[
            \left\{x\in \mM^s \ : \ \rho_{g'^{,s}}(x)<r_{\text{rot},\delta}(r_{\text{initial}}(\mM^s_0,g^s_0),\mt(x)) \right\};
            \]
            \item (bounded curvature regions) $g'^{,s}= g^s$ and $\partial'^{,s}_{\mt}=\partial^s_{\mt}$ on the set 
            \[
            \left\{x\in \mM^s \ : \ \rho_{g'^{,s}}(x)>C r_{\text{rot},\delta}(r_{\text{initial}}(\mM^s_0,g^s_0),\mt(x)) \right\}\supset \mM^s_0;
            \]
            \item (closedness) for any $m_1, m_2\in \left\{0,1,\ldots , [\delta^{-1}]\right\}$ we have 
            \[
            |\nabla^m\partial^{m_2}_{\mt}(g'^{,s}-g^s)|\leq\delta \rho^{-m_1-2m_2}, \ |\nabla^m\partial^{m_2}_{\mt}(\partial'^{,s}_{\mt}-\partial^s_{\mt})|\leq\delta \rho^{1-m_1-2m_2};
            \]
            \item (normalisation) if $(\mM^s_0, g^s_0)$ is homothetic to a quotient of the round sphere or the round cylinder, then $g'^{,s}=g^s$ and $\partial'^{,s}_{\mt}=\partial^s_{\mt}$ on all of $\mM^s$.
            \item (scaling property of the rounding radius) for all $a, r_0>0$ and $t\geq 0$
            \[
            r_{\text{rot},\delta}(a r_0, a^2 t)= a r_{\text{rot},\delta}(r_0,t).
            \]
        \end{enumerate}
        \end{theorem}
        
        \begin{remark}\label{rem:WellDefRstruct}
        We explicitly note that the five properties in the statement above are all part of the non-equivariant existence result for families of $\mR$-structures (Theorem 5.12 in \cite{BamKle19}), so all we need to do is to check that, if the input is a family of \emph{reflexive} singular Ricci flows then the whole construction can be suitably adapted so to
        produce, as an output, a family of \emph{reflexive} $\mR$-structures. As indicated above, this amount to taking care of those steps that are not canonical.
        \end{remark}

We introduce a monotone non-decreasing function $\nu\in C^{\infty}(\mathbb{R},\mathbb{R})$ such that 	\[
	\nu(x)=
	\begin{cases}
	0 & \text{if} \ x\leq 1/10, \\
	1 & \text{if} \ x\geq 1- 1/10,
	\end{cases}
	\]
	which will be employed multiple times in step-by-step successive refinements of the construction of $\mR$-structures and partial homotopies.

        \begin{proof}
        Keeping well in mind the content of Remark \ref{rem:WellDefRstruct}, we will now revisit the \emph{six} steps in the construction, indicating the necessary modifications (if any) or simply discussing why such a construction actually determines a reflexive $\mR$-structure in the case in question. In particular, we stress that all constants displayed below are defined exactly as in Section 5 of \cite{BamKle19} (thus \emph{we have decided not to redefine them here}), unless otherwise stated. In particular, all objects defined at the $k$-step in the construction are systematically assigned an index $k$ (as for, e.g., $\eta_k$ or $g'^{,s}_{k}$), and for the sake of clarity we convene to even update the indices of those objects that are left unmodified in a given step of the argument. 
        
        \
        
        \emph{First step: a cutoff function singling out the almost extinct components.} Here one can follow Section 5.7 of \cite{BamKle19} without modifications. Indeed, note that (relying on Theorem \ref{thm:UniquenessEquivariance}) for any $D'>0$ we have that a connected component $\mC\subset \mM^s_t$ is $D'$-small (in the sense that $\text{diam}(\mC)<D'\rho$ and $\rho<r_{\text{can},\e}/10$ at all points of $\mC$)
          if and only if $f^s_t(\mC)$ is so, and then further observe that the ancillary functions $\eta^*_1, \eta^{**}_1$ and thus also $\eta'_1, \eta_1$ (each being constant on any given connected component of $\mM^s_t$) attain the same value on $\mC$ and $f^s_t(\mC)$. 
        
        \
         
        \emph{Second step: modifications in regions that are close to Bryant solitons.} One can construct the spherical structures $\mS^s_2$, the compatible metrics $g'^{,s}_2$ and the local spatial vector fields $Y^s$ as follows. Recalling the fact that the tip of a Bryant soliton can be characterised as the (unique) global maximum point of the scalar curvature, given any $s\in K$ we consider the set $E^s\subset \mM^s$ of points $x'\in \mM^s_{t}$ (as one varies $t\geq 0$) such that all these requirements are met:
        \begin{enumerate}
            \item $\eta_1(x')<1$ (i.e. points not already selected in the almost-extinct components);
            \item $\rho(x')<10^{-1}r_{\text{can},\e}(x')$ (i.e. points of high curvature);
            \item $\nabla R_{g^s_t}(x')=0$ (i.e. points that are critical for the scalar curvature);
            \item $(\mM^s_t, g^s_t, x')$ is $\delta^{\ast}$-close to $(M_{\text{Bry}}, g_{\text{Bry}}, x_{\text{Bry}})$ at some scale;
            \item the diameter of the component of $\mM^s_t$ containing $x'$ is larger than $100 D_0\rho(x')$.
        \end{enumerate}
        
        Claim 5.21 of \cite{BamKle19} then ensures, among other things, that $E^s\subset \mM^s$ is a 1-dimensional submanifold, which is transversely continuous as one varies $s$, and (most importantly) 
        \begin{multline}\label{eq:TipsSeparation}
        \text{the balls} \ \left\{B_{g^s_t}(x', 10 D_0 \rho(x')), \ x'\in E^s\cap \mM^s_{t} \right\} \ \text{are pairwise disjoint}, \\
        \text{and} \ \text{inj}_{(\mM^s_t, g^s_t)}(x')>10 D_0 \rho(x').
        \end{multline}
        That being said, note that we can write $E^s= E^s_{I}\sqcup E^{s}_{II}$ where 
        \[
        E^s_{I}:=E^s \cap \Fix(f^s) 
        \]
        while of course $E^s_{II}$ is the complementary set, consisting of those points of $E^s$ that are not fixed by the map $f^s$ of the reflexive singular Ricci flow we are considering. 
        Since, as already mentioned (part c) of Claim 5.21 in \cite{BamKle19}), the set $E^s$ is transversely continuous as one varies $s$, it follows that the sets $E^s_{I}, E^s_{II}$  will also enjoy that property.
        Furthermore, one can perform a transversely continuous choice of a subset $E^s_{II,M}$ such that
        $E^s_{II,M} \cap f^s(E^s_{II,M})=\emptyset$ and $E^s_{II,M} \cup f^s(E^s_{II,M})=E^s_{II}$.
        
        For any $x'\in E^s_{II,M}\cap \mM^s_t$ we rely on the construction defined in Section 5.8 of \cite{BamKle19}: if $\hat{x}'_{s'}(t')$ is a local description of $E^s$ near $x'$ (as we vary $s'$ near $s$ and $t'$ near $t$) we fix a continuous, parallel in $t$, family of isometries $\chi$ such that $\chi_{(s',t')}: \mathbb{R}^3\to T_{\hat{x}'_{s'}(t')}\mM^{s'}_{t'}$ and employ it to define a family of exponential normal coordinates which allow to transplant the standard action of $O(3)$ on the metric ball $B_{g^s_t}(\hat{x}'_{s'}(t'), 10 D_0 \rho(\hat{x}'_{s'}(t')))$; at a local level, namely on $B_{g^s_t}(\hat{x}'_{s'}(t'), 10 D_0 \rho(\hat{x}'_{s'}(t')))$, the spherical structure $\mS'^{,s}$ is directly determined by the (orbits of the) group action, while the metric $g''$ and the vector field $\partial^{s'}_{\mt}+Y^{s'}$ are defined  (again locally on that same ball) by averaging the corresponding quantities $g$ and $\partial^{s'}_{\mt}$, respectively. Then one can use a cutoff function to interpolate between $g''$ and $g$ (between, say, radii $2D_0\rho(x')$ and $3D_0\rho(x')$) so to obtain a globally defined metric $g'''$. Hence, if $E^s_{II} \ni x''=f^s(x')$ we simply employ the map $f^s$ to replicate the construction around $x''$ so that the requirements in Definition \ref{def:R-structure} will be tautologically true by construction. Here we explicitly note that condition \eqref{eq:TipsSeparation}) ensures, in particular, that if $x'\in E^s_{II}$ then  $B_{g^s_t}(\hat{x}'_{s'}(t'), 10 D_0 \rho(\hat{x}'_{s'}(t')))$ does not intersect $\Fix(f)$.
        
        On the other hand, if instead $x'\in E^s_{I}$ so if we work around a point on the fixed locus $\Sigma^s_t=\Fix(f^s_t)$, for some $t\geq 0$, then we observe that one can choose the family of isometries $\chi$ in a way that (with the notation exactly as in the previous paragraph) $\mathbb{R}^2\times \left\{0\right\} \cong T_{\hat{x}'_{s'}(t')}\Sigma^{s'}_{t'}\subset T_{\hat{x}'_{s'}(t')}\mM^{s'}_{t'}$: if we then follow the very same averaging/interpolation construction we will obtain, on $B_{g^s_t}(\hat{x}'_{s'}(t'), 10 D_0 \rho(\hat{x}'_{s'}(t')))$, locally defined objects with the desired equivariance properties. 
        
        At that stage, one can define the desired metric by setting
        \[
        (g'^{,s}_2)_x:=g^s_x+\nu\left(\frac{r_{\text{can},\e}(x)}{10^2\rho(x)}\right)\cdot \nu(2-2\eta_1(x))\cdot ((g'''^{,s})_x-(g^s)_x),
        \]
        the spherical structure $\mS_2$ as the restriction of $\mS'$ (wherever defined via the local construction) to $\left\{\rho_{g'_2}<r_{\text{can},\e}/200\right\}\cap \left\{\eta_1<1/2\right\}$ and, correspondingly, for any spherical fiber $\mO$, the locally defined vector fields $(Y^{s'}_{\mO})$, which by construction have the property that $\partial^{s'}_{\mt}+Y^{s'}_{\mO}$ shall preserve $\mathcal{S}^{s'}_{2}$. Lastly, it is necessary to modify the cutoff function above by letting $\eta_2(s)=\nu(2\eta_1(x))$.
        
        \
        
        \emph{Third step: modifications in regions that are close to round cylinders.}
         Let us first remark that, after the first two steps in the construction we have (cf. Lemma 5.20 in \cite{BamKle19}, item c)) that for every point $x\in \mM^s_t$ with $\eta_2(x)<1$ at least one of the following assertions hold true:
        \begin{enumerate}
            \item [(c1)] {$\rho(x)>\alpha r_{\text{can},\e}(x)$;}
            \item [(c2)] {$x$ is the center of a $\delta^{\ast}$-neck with respect to $g'^{,s}_2$;}
            \item [(c3)] {there is an open neighborhood of $x$ that admits a two-fold cover in which a lift of $x$ is a center of a $\delta^*$-neck with respect to $g'^{,s}_2$;}
            \item [(c4)] {$x\in \text{domain}(\mS^s_2)$;}
        \end{enumerate}
        where $\delta^{\ast}>0$ is arbitrary and $\alpha\leq \overline{\alpha}(\delta^*)$ (without loss of generality $\alpha\leq 10^{-3})$. The scope of the present step is to construct $\mR$-structures around points satisfying either condition (c2) or (c3); as a result, we will have that the metric $g'^{,s}_3$ on $\mM^s$ is rounded (i.e. rotationally symmetric) everywhere except in regions where the curvature scale is large, or on almost-extinct components.
        For our purposes, it is convenient to start by only dealing with the case (c2). Recall that (by the very definition of $\delta$-neck, cf. Definition 3.25 of \cite{BamKle19}) the set of points $N_C$ satisfying (c2), for fixed $\delta^*$ is open in $\mM^s_t$ and, by a perturbative (implicit-function theorem) argument, if $\delta^*$ is small enough then for any such point $x$ there exists a \emph{unique} constant mean curvature sphere $\Gamma_x$ in metric $g'^{,s}_{2,t}$ passing through $x$; furthermore the diameter of $\Gamma_x$ is bounded from above by $100\rho(x)$ and the metric is close to that of a round sphere.

        Now, if $x\in N_C\cap \Fix(f^s_t)$ then by the uniqueness statement above one has $f^s_t(\Gamma_x)=\Gamma_x$ and, more generally, for any $y\in N_C$ we have $\Gamma_{f^s_t(y)}=f^s_t(\Gamma_y)$; at the level of metrics, since any map $f^s_t$ is an isometry (cf. Definition \ref{def:ReflexiveRicciFlow}) we will have
        \begin{equation}\label{eq:SymLeaf}
        (f^s_t)^*(g'^{,s}_{2,t})_{|\Gamma_{f^s_t(y)}}=((g'^{,s}_{2,t})_{|\Gamma_y}).
        \end{equation}
        
        In order to construct the spherical structure $\mS'^{,s}$ extending $\mS^{s}_2$, and to define the metric $g''^{,s}$ we proceed as in the proof of Lemma 5.24 in \cite{BamKle19}. To ensure the equivariance of the resulting output, we need to note that if $x\in N_C\cap \Fix(f^s_t)$ then (by virtue of lemma 5.13 item (d) of \cite{BamKle19}) 
        \begin{equation}\label{eq:SymCyl1}
        (f^{s}_t)^{\ast}(RD^2 (g'^{,s}_{2,t})_{|\Gamma_x})=RD^2 ((g'^{,s}_{2,t})_{|\Gamma_x})
        \end{equation}
        where $RD^2$ is the rounding operator on two-dimensional spheres, and for any $y\in N_C$
        \begin{equation}\label{eq:SymCyl2}
        (f^s_t)^*(RD^2(g'^{,s}_{2,t})_{|\Gamma_{f^s_t(y)}})=RD^2((g'^{,s}_{2,t})_{|\Gamma_y}).
        \end{equation}
        
        Concerning the definition of the metric along the `transverse' directions, the uniqueness of a vector field $Z_{\Gamma, V}$ of minimal $L^2$-norm, under the given constraints of preserving the spherical structure and having unit integral in the $g'^{,s}_2$-normal direction determined by the unit normal $V$, implies at once that
        \begin{equation}\label{eq:SymCyl3}
         Z_{\Gamma, V}(y)=Z_{f^s_t(\Gamma), (f^s_t)_\ast V}(f^s_t(y)) \ \ \text{for all} \ y\in \Gamma.
        \end{equation}
        (Observe that such claimed uniqueness result descends at once from the strict convexity of the functional.)
        
        It thus follows, as a result of \eqref{eq:SymLeaf}, \eqref{eq:SymCyl2} and \eqref{eq:SymCyl3}, that not only there exists a reflexive spherical structure $\mS'^{,s}$ (cf. Definition \ref{def:S-structure}) on all points satisfying the condition (c2) above but, in addition, one can define the compatible metric $g''^{,s}$ by imposing for $x\in \mM^s_t$ the orthogonal decomposition of $T_x \mM^s_t$ as $T\Gamma_x\oplus_{\perp}\langle Z_{\Gamma,V}\rangle$, where we take $g''^{,s}$ to equal $RD^2((g'^{,s}_{2,t})_{|\Gamma_y})$ in the top $2\times 2$ block, and this metric
        satisfies (on its domain) $(f^s)^*(g''^{,s})=g''^{,s}$.
        
        Let us now consider a point $x\in \mM^s_t$ satisfying condition (c3) instead. By passing to a local two-fold cover we see that there exists a unique constant mean curvature $\mathbb{R}\mathbb{P}^2$, henceforth denoted by $\Gamma_x$, passing through $x$. At this stage, we note that set of points satisfying such a condition (c3) must be disjoint from $\Fix(f^s_t)$. Indeed, suppose on the contrary $x\in \Fix(f^s_t)$ to be such a point: like we noticed above, necessarily $f(\Gamma_x)=\Gamma_x$.  Let us recall that, by explicit classification of totally geodesic surfaces in a round cylinder (cf. Remark 4.13 in \cite{CarLi19}, based on \cite{SouVan12}), at the level of the two-fold cover there are only two possible local pictures.
         In the first case $\Gamma_x$ is a connected component of $\Fix(f^s_t)$, which is impossible since $\Fix(f^s_t)$ is two-sided hence orientable given that $\mM^s_t$ is (cf. Definition \ref{def:ReflexTriple},  and Definition \ref{def:ReflexiveRicciFlow}). On the other hand, it could happen that the lift of $\Gamma_x$ meets the lift of $\Fix(f^s_t)$ orthogonally along a circle, which (once we project back to $\mM^s_t$) is incompatible with the assumption $(\star_{\text{sep}})$, i.e. with the requirement that for $y\in\mM^s_t$ it cannot happen that $y$ and $f(y)$ are in the same connected component of $\mM^s_t\setminus\Fix(f^s_t)$. As a result, that claim being verified, case (c3) can be handled exactly as in the non-equivariant case (as there is `decoupling'), with no modifications whatsoever. 
        
        \
        
        We then proceed with the construction of the vector fields preserving the spherical structure. Building again on the aforementioned classification of totally geodesic surfaces in round cylinders, case (c2) above may correspond to two different geometric pictures. We will here start with the second possible picture: locally near a point $x\in \mM^s_t$ with $x\in N_C$ we have that $\Gamma_x$ intersects $\Fix(f^s_t)$ orthogonally. 
        Informally speaking, one only needs to follow the trajectory $x(t')$ for $t'$ close to $t$. More precisely, if we consider $\Gamma_x$ as above (thus, by definition, a fiber of $\mS'^{,s}$ that is not a fiber of $\mS^s_2$, the spherical structure we had defined in the second step and that has just be extended to one with larger support) for any $y\in \Gamma_x$ and two smooth families of orthonormal vector fields $v_{i,t}$ (for $i=1,2$) along the curve $y(t')$ that remain tangent to the fiber $\Gamma_{y(t')}$ as $t'$ varies (near $t$), and subject to the additional property that $dv_{i,t'}/dt$ is normal to $v_{j,t'}$ ($j\neq i$) along the evolution as well as the `equivariance constraint' that $v_{1,t'}$ is tangent to the intersection circle, there exists a unique family of homothetic embeddings $\beta_{t'}: S^2\to\mM^s_t$ parametrising the path $\Gamma_{y(t')}$, such that:
        \begin{itemize}
            \item $\beta_{t'}(y_0)=y(t')$ for $y_0=(1,0,0)\in S^2$;
            \item $d\beta_{t'}(u_1)=\kappa(t')v_{1,t'}$ and $d\beta_{t'}(u_2)=\kappa(t')v_{2,t'}$ for some positive factor $\kappa(t')$, and fixed tangent vectors $u_1=(0,1,0), u_2=(0,0,1)\in T_{y_0}S^2$;
            \item  is equivariant in the sense that 
        \[
        f^{s'}_{t'}\circ \beta_{t'}(x_0)=\beta_{t'}(f_0(x_0)), \ \text{for all} \ x_0\in S^2,
        \]
        where we recall that $f_0$ denotes the restriction to $S^2$ of the standard reflection with respect to the third coordinate plane in $\R^3$.
        \end{itemize}
         (Of course, it is understood that $\beta_{t'}$ is a \emph{homothetic embedding} when one considers the \emph{rounded} metrics on the target.)
        
        The construction above can now be performed `in families', depending on the parameter $s$ also, and determines charts (in the sense of Section 4 of \cite{BamKle19}) of the type $\chi=(\chi^{s'}_{t'})$ with (for any fixed $s')$ $\chi^{s'}_{t'}: S^2\times (-a^{s'},a^{s'})\to\mM^{s'}_{t'}$ satisfying
        \begin{equation}\label{eq:EquivEmbed}
        f^{s'}_{t'}\circ \chi^{s'}_{t'}(x_0,r_0)=\chi^{s'}_{t'}(f_0(x_0),r_0), \ \text{for all} \ x_0\in S^2, r_0\in (-a^{s'},a^{s'}).
        \end{equation}
           and an associated transversely continuous family of $O(3)$-actions as in the definition of reflexive spherical structure (cf. Definition \ref{def:S-structure}). In particular, thanks to \eqref{eq:EquivEmbed} the action satisfies, in turn, the identity that  $f^{s'}_{t'}(\zeta(\rho, z))=\zeta(\rho, f^{s'}_{t'}(y))$ for all $\rho\in O(3), z\in V$. Now, the vector field $\partial^{s'}_{\mt}+Y^{s'}_{\chi}$ is obtained by averaging $\partial_{\mt}$ along the orbits of the action (exactly as we had done, in the previous step, near the tip of regions close to Bryant solitons), and is therefore equivariant under $(f^{s'}_{t'})_{\ast}$.
        
        In the first possible local picture for case (c2) we have that, given $x\in \mM^{s}_t\cap N_c$, $\Gamma_x$ coincides with a connected component of $\Fix(f^s_t)$. In this case, we recall (straight from condition (2) in Definition \ref{def:ReflexiveRicciFlow}) that, since $\Fix(f^s_t)$ is preserved by the flow of $\partial_{\mt}$, necessarily the previous construction gives $Y^{s}_{\chi}=0$ along $\Gamma_x$. More generally, when dealing with a leaf that is not necessarily the central one, we follow the line of the previous construction, but now imposing the `equivariance constraint' by taking $v_{1,t'}, v_{2,t'}$ subject to the (different) requirement that 
        \[
        (f^{s'}_{t'})_{\ast} v^{s'}_{i,t'}(y(t'))= v^{s'}_{i,t'}((f^{s'}_{t'}(y(t'))).
        \]
        From there, in an open $f^{s'}_{t'}$-invariant neighborhood of the central leaf, we obtain an equivariant chart map $\chi$ satisfying
        \[
        f^{s'}_{t'}\circ \chi^{s'}_{t'}((x_0, r_0))= \chi^{s'}_{t'}(((x_0, -r_0))), \ \text{for all} \ x_0\in S^2, r_0\in (-a^{s'},a^{s'})
        \]
        whence the corresponding equivariance property of the action $\zeta$ and, next, the fact that
        the vector field $Y^{s'}$ will satisfy, by construction, the equation
        \begin{equation}\label{eq:VectorField}
        (f^{s}_t)_{\ast} Y^s_{\chi^s_t}(x)= Y^{s}_{\chi^s_t}(f^{s}_t(x)).    
            \end{equation}
            
        The case of open sets of points that lift to centers of $\delta$-necks, so corresponding to condition (c3), causes no troubles in view of the discussion above, and, more specifically, because of the fact that such regions are away from $\Fix(f^s_t)$: one works at the level of the two-fold cover, constructs the vector field there, and employs the maps $f^s_t$ to `reflect through' $\Fix(f^s_t)$, so to ensure the equivariance of the resulting vector field.   
        
        \
        
        Lastly, like in step two, one needs to perform some final adjustments and interpolations so to define the metric $g'^{,s}_3$ globally on $\mM^s$. Specifically, we shall set $\eta_3(x):=\nu(2\eta_2(x))$, then for $x\in\text{domain}(\mS'^{,s})$
 \[
        (g'^{,s}_3)_x:=(g'^{,s}_2)_x+\nu\left(\frac{\alpha r_{\text{can},\e}(x)}{10^2\rho(x)}\right)\cdot \nu(2-2\eta_2(x))\cdot ((g''^{,s})_x-(g'^{,s}_2)_x),
        \]
        else if $x\notin\text{domain}(\mS'^{,s})$ we just leave the metric unchanged i.e. we set $g'^{,s}_3=g'^{,s}_2$.
        The spherical structure $\mS^s_3$ is declared to be the restriction of $\mS'^{,s}$ (wherever defined via the local construction) to $\left\{\rho_{g'_3}<\alpha r_{\text{can},\e}/400\right\}\cap \left\{\eta_2<1/2\right\}$; correspondingly, for any spherical fiber $\mO$, we have the locally defined vector fields $(Y^{s'}_{\mO})$, which by construction shall preserve $\mathcal{S}^{s'}_{3}$. All of these definitions are given in terms of geometric quantities only, namely in terms of $\eta_1, r_{\text{can},\e}, \rho, \rho_{g'_3}$ so, as such, and in view of Theorem \ref{thm:UniquenessEquivariance}, they determine a reflexive $\mR$-structure, as desired.       
       
       \
       
        \emph{Fourth step: modification of the time vector field.} As a next step, we refine the constructions given in two previous ones so to obtain a transversely continuous family of smooth vector fields $(\partial'^{,s}_{\mt,4})$, satisfying $\partial'^{,s}_{\mt,4}\mt=1$, and modify all related structures (in particular obtaining a transversely continuous family of spherical structures $(\mS^s_4)$) that are preserved by the flow of such vector fields.
        
        As an \emph{Ansatz}, we will take
        \begin{equation}\label{eq:AnsatzVectField}
            \partial'^{,s}_{\mt,4}=\partial^{s}_{\mt}+\eta^* Z^s
        \end{equation}
        where $Z^s$ is a vector field defined on a suitable subset $U^s$ of $\text{domain}(\mS^s_3)$, and $\eta^*:\cup_{s}\mM^s\to [0,1]$ is a cutoff function.
        
        More specifically, (keeping in mind the outcome of the third step) for any $s\in K$ we set
        \begin{equation}\label{eq:Domain}
            U^s:=\left\{\eta_3<1\right\} \cap \left\{\rho_{g'_3}<\alpha r_{\text{can},\epsilon}/2\right\} \cap \mM^s \subset \text{domain}(\mS^s_3)
            \end{equation}
            and we define $Z^s$ as follows: for any spherical fiber $\mO$ in $U^s\cap \mM^s_t$ (for any $t\geq 0$) $Z^s_{|\mO}=Z'_{\mO}$ where $Z'_{\mO}$ is the element of minimal $L^2$-norm in the class of \emph{admissible} vector fields; a vector field $Z'$ along $\mO$ is defined admissible if $\partial^s_{\mt}+Z'$ can be extended to a vector field defined in an open set containing $\mO$ (inside $\mM^s$) and preserving $\mS^s_3$ therein.
        
        Then, arguing exactly as in the non-equivariant case one shows that $Z^s$ is a well-defined, smooth and transversely continuous vector field on $\cup_s U_s$; also, the very same local comparison formula proven there, which reads
        \begin{equation}\label{eq:RepTheory}
            Y^{s'}_{\mO}-Z^{s'}=\frac{1}{|O(3)|}\int_{O(3)}(1+3\text{tr} A)(\zeta^{s'}_A)_{\ast}Y^{s'}_{\mO}dA,
        \end{equation}
        (where $(\zeta^{s'})$ denotes a transversely continuous family of local $O(3)$ actions, that for each $s'$ are compatible with $\mS^{s'}_3$ and isometric with respect to $g'^{,s'}_3$, and the averaged integral to be understood with respect to the natural Haar measure)
        allows one to show that the vector field $Z^s$ satisfies the equation
        \[
        (f^{s}_t)_{\ast} Z^s_t(x)= Z^{s}_t(f^{s}_t(x))
        \]
        since this is already known to be the case for the other two terms in \eqref{eq:RepTheory}, namely those involving $Y^{s'}_{\mO}$, as a consequence of the previous two steps.
        
        That being said, the aforementioned cutoff function $\eta^*$ shall be defined by letting
        \[
        \eta^*(x)=\nu\left(\frac{\alpha r_{\text{can},\e}(x)}{10^2 \rho_{g'_3}(x)}\right)\cdot \nu(2-2\eta_3(x)).
        \] 
        
        Indeed, notice that
        $\text{supp} (\eta^*) \subset \left\{\eta_3<19/20 \right\} \cap \left\{\rho_{g'_3}<\alpha r_{\text{can},\e}/10\right\}$
which in particular implies $\text{supp} (\eta^{\ast}) \subset \cup_s U^s$; in addition, we have that  $\eta^*=1$ identically on $\left\{\eta_3< 11/20\right\} \cap \left\{\rho_{g'_3}<\alpha r_{\text{can},\e}/90\right\}$ so one will define $\mS^s_4$ to be the restriction of $\mS^s_3$ to $U'\cap \mM^s$
where it is set $U'=\left\{\eta_3< 11/20\right\} \cap \left\{\rho_{g'_3}<\alpha r_{\text{can},\e}/10^2\right\}$. Lastly, as in any of the previous steps, we will set $\eta_4(x)=\nu(2\eta_3(x))$.

\

\emph{Fifth step: refinements for the almost-extinct components.} The scope of this step is to suitably extend the spherical structure $\mS_4$ to $\mS_5$ and modify the metric $g'^{,s}_4$ to  $g'^{,s}_5$ so that, in addition to all other properties, one can ensure that  $g'^{,s}_5$ remains compatible with $\mS_5$ for an a priori longer time and, in fact, till the stage where the universal cover of any connected component for which $\eta_4>0$ (a posteriori $\eta_5>0$) is $\delta$-close to the round sphere, modulo rescaling. The precise statement is Lemma 5.29 in \cite{BamKle19}. It is simple to check that the argument presented by Bamler-Kleiner for this step carries over \emph{verbatim} to our setting.
Indeed, those components for which $\eta_4>0$ (a condition which is more restrictive than $\eta_1>0$, so that Lemma 5.15 therein applies) are already known to be compact, to have positive sectional curvature and bounded normalised diameter hence they become extinct in finite time and (for any given $s$) the evolution of each of them occurs in a product domain $U_{\mC}\subset \mM^s$ (where the original singular Ricci flow is nothing but a classical Ricci flow). The construction is then performed one component at a time. In our case there can be two possible situations depending on whether, for any such $\mC\subset \mM^s_t$, one had $\mC\cap \Fix(f^s_t)\neq\emptyset$ or instead  $\mC\cap \Fix(f^s_t)=\emptyset$ (in which case we will have a `twin' connected component $\mC'$). Either way, all operations in the construction (rescaling via $\hat{\rho}$, pulling-back, evolving by volume-normalised Ricci flow) do preserve the pre-existing symmetries of the reflexive Ricci flow $(\mM, \mt, \partial_\mt,g, f)$. The function $\eta_5$, that is constant on each connected component of $\mM^s_t$ (for any fixed $s\in K$ and $t\geq 0$) equals $\eta_4$ where $\eta_4=0$ and is modified where $\eta_4>0$ consistently with the previous construction, is also formally defined exactly as in the non-equivariant case.

\

\emph{Sixth step: trivialisation of the metric on almost-extinct components and conclusion.} The construction is further refined, and then completed, in order to accomodate condition (4) of Definition \ref{def:PrelimR}: one needs to actually round the metric on those connected components where $\eta_5>0$, then modify the time vector field and the spherical structures accordingly, as we are about to describe. 

Like in the previous step, given any parameter $s\in K$ and $t\geq 0$ the subset $\left\{\eta_5>0\right\}\cap \mM^s_t$ consists of compact connected components and one can perform the needed modifications one component at a time. So, since (the universal cover of) any such $(\mathcal{C}, g'^{,s}_t)$ is quantitatively close to a round sphere modulo scaling, one can first apply the rounding operator and set
$(g''^{,s}_{t})_{|{\mC}}=\text{RD}^3((g'^{,s}_{5,t})_{|{\mC}})$.
Thereby one defines a smooth metric on $\left\{\eta_5>0\right\}\cap \mM^s$, that has the desired equivariance properties because of \cite[Lemma 5.13 item (d)]{BamKle19} (this last statement obviously refers to those components of $\mM^s_t$ intersecting $\Fix(f^s_t)$, for else the conclusion is trivial). Then one defines $\partial''^{,s}_{\mt}=\partial'^{,s}_{\mt}+Z_{\mC}$ by selecting $Z_{\mC}$ as the element of minimal $L^2$-norm (with respect to $g''^{,s}_{\mt}$) in the class of those vector fields $Z''$ such that $\partial'^{,s}_{\mt}+Z''$ determines a flow by homotheties. As recalled above, uniqueness of the minimiser (which follows from the convexity of the functional in question, that is actually minimised on an affine finite-dimensional subspace) forces the equivariance of the resulting vector field. Thereby, we have a smooth vector field $Z^s$ on $\left\{\eta_5>0\right\}\cap \mM^s$, again with the desired equivariance properties. In addition, as already explicitly remarked in \cite{BamKle19}, the invariance of both constructions (i.e. that of the metric, and that of the vector field) under isometries implies that $g''^{,s}$ is still compatible with $\mS^s_5$, and that $\partial''^{,s}_{\mt}$ still preserves $\mS^s_5$. Passing from such connected components to globally defined objects is made by standard interpolation; specifically one sets 
for any $x\in \mM^s_t$
\[
(g'^{,s}_t)_{x}:=(g'^{,s}_{5,t})_x+\nu(2\eta_5(x))\cdot ((g''^{,s}_{t})_x - (g'^{,s}_{5,t})_x)
\]
and
\[
(\partial'^{,s}_{\mt})_{x}:=(\partial'^{,s}_{5,\mt})_x+\nu(2\eta_5(x))\cdot ((\partial''^{,s}_{\mt})_x - (\partial'^{,s}_{5,\mt})_x).
\]
At this very stage, the definition of the sets $U_{S3}, U_{S2}$ and of the spherical structure $\mS$ carries through unmodified from \cite{BamKle19}. Indeed all geometric functions that are employed in these definitions ($\eta_5$, $\hat{\rho}, r_{\text{can},\e}$) are constant as one varies $x$ in any given connected component of $\mM^s_t$ (for any fixed $s,t$) and all related arguments only involve those components where 
$\partial'^{,s}_{\mt}$ determines a flow of homotheties with respect to the metric $g'^{,s}$. Thus, these definitions automatically satisfy the additional equivariance constraints, and the conclusion of Theorem \ref{thm:ExistR-struct} follows.
\end{proof}

        \section{Reflexive partial homotopies and backward-in-time induction}\label{sec:PartialHom}
        
        The purpose of this section is the construction of reflexive partial homotopies: given a continuous family of initial data, and considered the transversely continuous family of singular Ricci flows it generates, we shall employ the associated reflexive $\mR$-structures constructed in the previous section to design a backward induction scheme whose ultimate scope is to prove Proposition \ref{prop:SurjPartHom} below for $T=0$. In order to conveniently state such a result we first need to wrap-up some important properties of reflexive $\mR$-structures, that mostly follow from the previous discussion.
        
        \
        
       We are given a finite-dimensional simplicial complex $\mK$ whose dimension shall henceforth be denoted by $n\geq 0$ and for which we fix a geometric realisation $K$. Considered a family of compact reflexive 3-manifolds $(M, g^s, f)_{s\in K}$ we let $(\mM^s)_{s\in K}$ denote, in short, the associated reflexive singular Ricci flows, which we assume to be all extinct for $T\geq T_{ext}$ (which happens, for instance, under the topological assumptions on $M$ given in Theorem \ref{thm:MainHomotopyConfVersion}).
      Furthermore, we let $(\mR^s)_{s\in K}$ denote the reflexive $\mR$-structures 
        \[ 
        \mR^s=(g'^{,s},\partial'^{,s}_{\mt}, U^{s}_{S2}, U^{s}_{S3}, \mS^s)
        \]
        provided by Theorem \ref{thm:ExistR-struct} for a parameter $\delta>0$. We fix, once and for all, a numerical constant $\Lambda>100$ and convene that all scales, in particular $\rho$, refer to the output of the rounding process, so that for instance $\rho=\rho_{g'^{,s}_t}$ unless otherwise stated.
        
        \begin{lemma}\label{lem:PreparSurjPartHom}
          There exist positive thresholds $\underline{C}_{0}(\Lambda)>10$ and $\overline{\delta}(\Lambda)$ such that, for any $C_0\geq \underline{C}_{0}(\Lambda)$ and $0<\delta\leq \overline{\delta}(\Lambda)$, if
          \begin{equation}\label{eq:LargeScales}
              r_{\text{initial}}(M^s, g^s), r_{\text{can},\delta}(r_{\text{initial}}(M^s, g^s), t), r_{\text{rot},\delta}(r_{\text{initial}}(M^s, g^s), t)> C_0 \ \ \ \forall \ s\in K, t\geq 0 \ \text{if} \ \mM^s_t\neq\emptyset 
          \end{equation}
              then the following assertions hold true:
          \begin{enumerate}
              \item For any $r>0$ there exists $\theta=\theta(r)\in (0,r^2]$ such that for any $s\in K$ and $t_1, t_2\geq 0, |t_1-t_2|\leq \theta(r)$:
              \begin{enumerate}
                  \item if $x\in \mM^s_{t_2}$ with $\rho(x)>r/10$ then $x$ survives till time $t_1$ (that is to say: $x(t_1)$ is well-defined) and $|\rho(x)-\rho(x(t_1))|< 10^{-3}\rho(x)$;
                  \item if, in addition to what is stated in the previous item, $t_1\leq t_2$ and $\rho(x(t_1))\leq \rho(x)\leq 10$, then there are $f$-invariant embedded disks $D'\subset D\subset U^s_{S2}\cap \mM^s_{t_2}$ with $x\in D'$ that are the union of spherical fibers for $\mS^s$, with $D'$ containing a singular fiber that is a point $x'$, and satisfying
                  \begin{equation}\label{eq:ScaleReq}
                  \begin{cases}
                  \rho>\frac{9}{10}\rho(x) & \ \text{on} \ D \\
                  \rho> 2\Lambda^3\rho(x) & \ \text{on} \ \partial D \\
                  \rho<2\rho(x) & \ \text{on} \ D' ;
                  \end{cases}
                  \end{equation}
              \end{enumerate}
              \item If $(M^s, g^s)$ has positive scalar curvature, then $g'^{,s}$ has positive scalar curvature for any $t\geq 0$ such that $\mM^s_t\neq\emptyset$;
              \item If $t\geq 0$ and $Y\subset \mM^s_t$ is a compact 3-dimensional submanifold, such that $(Y, g^s_t, f_{|Y})$ is a reflexive 3-manifold, and if $\partial Y$ consists of regular fibers of $\mS^s$ where $\rho\leq 1$, then $(Y, g'^{,s}_t, f_{|Y})$ is reflexively PSC-conformal.
                      \end{enumerate}
        \end{lemma}
        
        \begin{remark}\label{rem:NormalisationByMult}
        We note that, due to the compactness of the parameter space $K$ and the scaling properties of $r_{\text{rot},\delta}$ we may, and shall in fact assume that, possibly by multiplying the initially assigned metrics $(g^s)_{s\in K}$ by a large constant, condition \eqref{eq:LargeScales} is met. Hence, from property (2) in Theorem \ref{thm:ExistR-struct}, we can without loss of generality assume that $\rho>1$ on $\mM^s_0$ for all $s\in K$. This is a simple, convenient normalisation: as described in Section \ref{sec:OutlineParab}, the partial homotopies we design (say at time $T\geq 0$) will not necessarily be supported on all of $\mM^s_T$, but may in fact miss regions of high curvature, where the threshold defining this notion will be from now onwards be unitary.
         Also, it is appropriate to remark that (by property (1) stated in Theorem \ref{thm:ExistR-struct}) that, with such a normalisation, the set $\left\{\rho=\rho_{g'^{,s}}<C_0\right\}$ is included in the support of $\mR^s$ (which by definition is $U^s_{S2}\cup U^s_{S3}$) for all $s\in K$. 
            \end{remark}
            
            \begin{proof}
             First of all, assertions (1)(a) and (2) follow exactly as in the non-equivariant case (see Lemma 8.4 in \cite{BamKle19}), the latter for $\delta>0$ small enough (in relation to item (3) of Theorem \ref{thm:ExistR-struct}). Assertion (3) only needs one to appeal to Lemma \ref{lemma.PSCconform.criter} in lieu of Lemma 6.7 in \cite{BamKle19}, and requires $\delta>0$ small enough and $C_0$ large enough, instead. So, one is then left with checking that the sets $D, D'$ in assertion (1)(b) can be designed to be $f$-invariant (and that the claimed estimates still hold true). To that aim, let us first recall that (as in the non-equivariant case, so utilising Lemma 3.32 in \cite{BamKle19}) given any $\delta'>0$ one can pick $\delta>0$ small enough and $C_0$ large enough so that the pointed manifold $(M, g^{s}_{t_2},x)$ is $\delta'$-close to the pointed standard Bryant soliton $(M_{\text{Bry}}, g_{\text{Bry}},x_{\text{Bry}})$ at scale $\rho(x)$, where $x_{\text{Bry}}$ stands for the tip of the soliton in question (i.e. the origin in $\R^3$). Hence, given any $\Lambda'>0$ for $\delta'$ small enough it follows that the scalar curvature of $g'^{,s}$ attains a unique maximum in the metric ball of radius $\Lambda'\rho(x)$, at a point $x'$. Now, there are two possible cases. If $x'\in \Fix(f^s_{t_2})$ then we consider the length-minimising geodesic connecting $x$ to $x'$ and take $D'$ to be the union of all spherical fibers of $\mS^s$ intersecting such an arc: since $\mS^s$ has, by construction (cf. Definition \ref{def:R-structure}) the desired equivariance property, it follows that $D'$ is $f$-invariant, and one can then also define a larger disc $D$ based on the scale requirements only (for which we recall that there exists a constant $C_{\text{Bry}}>1$ such that 
             \[
             C^{-1}_{\text{Bry}} d^{-1}_{g_{\text{Bry}}}(x_{\text{Bry}},\cdot) \leq R_{g_{\text{Bry}}}\leq C_{\text{Bry}}d^{-1}_{g_{\text{Bry}}}(x_{\text{Bry}},\cdot)
             \] cf. Section 3 and Theorem 1 in \cite{Bry05}).
             If instead $x'\notin \Fix(f^s_{t_2})$ then (again, by possibly taking $\delta$ even smaller and $C_0$ even larger) one can assume, without loss of generality, that $B_{g'^{,s}_{t_2}}(x',\Lambda'\rho(x))\cap \Fix(f^{s}_{t_2})=\emptyset$ (if not, $x'$ would belong to one of the two connected components of $M\setminus  \Fix(f^{s}_{t_2})$ and there would be a second maximum point of the scalar curvature in the same metric ball centered at $x$ and of radius $\Lambda'\rho(x)$, a contradiction). Thus, in this case, running the argument above for given $\Lambda'>13\Lambda^6 C_{\text{Bry}}$ one can pick $D'$ and $D$ away from $\Fix(f^{s}_t)$ (more precisely: in one of the two open connected components of $M\setminus \Fix(f^s_{t_2})$) and satisfying the curvature requirements in the statement. Thereby, the conclusion follows.
            \end{proof}

        We henceforth convene to choose positive constants $C_0$ and $\delta$ so that all conclusions in the statement of Lemma \ref{lem:PreparSurjPartHom} hold true. For any $k\in\left\{0,1,\ldots, n\right\}$ we further set $r_k:=\Lambda^{-4n+4k-4}$. Finally, let $K'$ denote a closed subset of $K$ with the property that, whenever $s\in K'$ the metric $g^s$ has positive scalar curvature.

        \begin{proposition}\label{prop:SurjPartHom}
        In the setting above, for any $T\geq 0$, by possibly passing to a simplicial refinement of $\mK$ (which we shall not rename) there exists a reflexive partial homotopy at time $T$ for the reflexive $\mR$-structures $(\mR^s)_{s\in K}$
        \[
        \left\{(Z^{\sigma}, (g^{\sigma}_{s,t})_{s\in \sigma, t\in [0,1]}, f^{\sigma}, (\psi^{\sigma}_s)_{s\in\sigma} \right\},
        \]
        in the sense of Definition \ref{def:PartialHom}, such that the following assertions hold true for any $k$-simplex $\sigma$ of $\mK$ and all $s\in\sigma$:
        \begin{enumerate}
            \item $\left\{\rho>1 \right\}\cap \mM^s_T\subset \psi^{\sigma}_s(Z^{\sigma})\subset \left\{\rho>r_k\right\}$;
            \item Every connected component of $\psi^{\sigma}_s(Z^{\sigma})$ contains points with $\rho>\Lambda^2 r_k$;
            \item $\rho>\lambda r_k$ on $\psi^{\sigma}_s(\partial Z^{\sigma})$;
            \item $\left\{(Z^{\sigma}, (g^{\sigma}_{s,t})_{s\in \sigma, t\in [0,1]}, f^{\sigma}, (\psi^{\sigma}_s)_{s\in\sigma} \right\}$ is reflexively $PSC$-conformal over every $s\in K'$ (for every simplex containing it).
        \end{enumerate}
        \end{proposition}
        
        Like it has been anticipated above, the proof of this statement follows by backward induction in time, starting past the extinction time $T_{ext}$, where one can just take $Z^{\sigma}=\emptyset$ for any simplex $\sigma$ (for $\sigma\in \mK$) and then showing, in the inductive step, that one can derive from the conclusion at time $T$ the conclusion at time $T-\Delta T$ for any choice of $\Delta T\in (0, \min\left\{T, \theta(r_0)\right\}$. (Note that here $r_0$ is the constant defined above, i.e. $r_0=\Lambda^{-4n-4}$ for $n$ the dimension of the simplicial complex $\mK$). In particular, if we can reach $T=0$ we obtain, from the item (1) above, that 
        \[
        \left\{\rho>1 \right\}\cap \mM^s_0 = \mM^s_0 \subset \psi^{\sigma}_s(Z^{\sigma})
        \]
        that is to say we would have built a \emph{surjective} reflexive partial homotopy (a wording introduced in Section \ref{sec:OutlineParab}), which implies Theorem \ref{thm:MainHomotopyConfVersion} hence Theorem \ref{thm:MainHomotopy} and, in turn, Theorem \ref{thm:ContractDoubling}.
        
        
        The most basic tool to implement the backward induction scheme that has just been described is the following intuitive result describing how $\mR$-structures can be employed to build a partial homotopy at an earlier time given a partial homotopy at a later one.
        
        \begin{proposition}\label{prop:BackwardFlow}
        Given a transversely continous family of reflexive singular Ricci flows $(\mM^s)_{s\in K}$, let $(\mR^s)_{s\in K}$ be an associated, transversely continuous, family of $\mR$-structures. Consider a reflexive partial homotopy $\left\{(Z^{\sigma}, (g^{\sigma}_{s,t})_{s\in \sigma, t\in [0,1]}, f^{\sigma}, (\psi^{\sigma}_s)_{s\in\sigma}) \right\}_{\sigma\in \mK}$ at time $T$ and let $T'\in [0,T]$. Suppose that:
        \begin{enumerate}
            \item for all $s\in\sigma\in \mK$ all points of $\psi^{\sigma}_s(Z^{\sigma})$ survive until time $T'$ with respect to the flow of the vector field $\partial'^{,s}_{\mt}$;
            \item for all $s\in \tau\subsetneq \sigma\in \mK$ and for all $t'\in [T',T]$ the following inclusion holds:
            \begin{equation}\label{eq:InductiveInclusion}
                (\psi^{\tau}(Z^{\tau})-\psi^{\sigma}(Z^{\sigma}))^{\partial'^{,s}_{\mt}}(t')\subset U^s_{S2}\cup U^{s}_{S3},
            \end{equation}
            where this notation refers to the image of the set in question through the flow of the vector field $\partial'^{,s}_{\mt}$, up to time $t'$.
        \end{enumerate}
        
        Then letting
        \begin{equation}\label{def:BackwardMap}
            (\psi^{\sigma}_s)_{t'}(z):=(\psi^{\sigma}_s(z))^{\partial'^{,s}_{\mt}}(t') \ t'\in [T',T] \ \ \ \ \overline{\psi}^{\sigma}_s:=(\psi^{\sigma}_s)_{T'}
        \end{equation}
        and
         \begin{equation}\label{def:BackwardMetric}
            \overline{g}^{\sigma}_{s,t}:=
            \begin{cases}
               (\psi^{\sigma}_s)^*_{T'+2t(T-T')}g'^{,s}_{T'+2t(T-T')}  & \text{if} \ t\in [0,\frac{1}{2}] \\   
                g^{\sigma}_{s,2t-1} & \text{if} \ t\in [\frac{1}{2},1]    
            \end{cases}
        \end{equation}
        
        defines a reflexive partial homotopy $\left\{(Z^{\sigma}, (\overline{g}^{\sigma}_{s,t})_{s\in \sigma, t\in [0,1]}, f^{\sigma}, (\overline{\psi}^{\sigma}_s)_{s\in\sigma})\right\}_{\sigma\in \mK}$ at time $T'$ such that the following property holds:
        if $\left\{(Z^{\sigma}, (g^{\sigma}_{s,t})_{s\in \sigma, t\in [0,1]}, f^{\sigma}, (\psi^{\sigma}_s)_{s\in\sigma})\right\}_{\sigma\in \mK}$
        is reflexively $PSC$-conformal over $s$ 
        and if $((\psi^{\sigma}(Z^{\sigma}))^{\partial'^{,s}_{\mt}}(t'),g'^{,s}_{t'},f^s_{t'})$ is reflexively PSC-conformal for all $t'\in [T',T]$ 
        then it holds that $\left\{(Z^{\sigma}, (\overline{g}^{\sigma}_{s,t})_{s\in \sigma, t\in [0,1]}, f^{\sigma}, (\overline{\psi}^{\sigma}_s)_{s\in\sigma})\right\}_{\sigma\in \mK}$ is also reflexively $PSC$-conformal over $s$. 
        \end{proposition}

        \begin{proof}
         All claims follow, in a rather straightforward fashion, from the properties defining $\mR$-structures, see Definition \ref{def:R-structure}, exactly as in the proof of Proposition 7.6 in \cite{BamKle19}; the equivariance properties of $(g'^{,s},\partial'^{,s}_{\mt}, U^{s}_{S2}, U^{s}_{S3}, \mS^s)$ together with those of $\psi^{\sigma}_s$ and $g^{\sigma}_{s,t}$ imply the resulting equivariance properties of $\overline{\psi}^{\sigma}_s$ and $\overline{g}^{\sigma}_{s,t}$.  
        \end{proof}
        
        In order to fruitfully exploit the result above so to prove Proposition \ref{prop:SurjPartHom} we need to introduce certain operations on partial homotopies. First of all, we note that there is a very natural way of `passing to a simplicial refinement'.
        
        \begin{lemma}\label{lem:SimplicialRefine}
          Let $\mK$ be a simplicial complex and let $\tilde{\mK}$ be a simplicial refinement thereof; let us convene to identify their geometric realisations $K$ and $\tilde{K}$ respectively.  Given a transversely continous family of reflexive singular Ricci flows $(\mM^s)_{s\in K}$, consider a reflexive partial homotopy $\left\{(Z^{\sigma}, (g^{\sigma}_{s,t})_{s\in \sigma, t\in [0,1]}, f^{\sigma}, (\psi^{\sigma}_s)_{s\in\sigma}) \right\}_{\sigma\in \mK}$ at time $T$. Then, for any $\tilde{\sigma}\in\tilde{\mK}$ let $\sigma_{\tilde{\sigma}}$ denote the simplex of $\mK$ having least dimension among those containing $\tilde{\sigma}$. Letting
          \begin{equation}\label{eq:RefinementRecipe}
              \overline{Z}^{\tilde{\sigma}}:=Z^{\sigma_{\tilde{\sigma}}}, \ \ \overline{\psi}^{\tilde{\sigma}}_s:=\psi^{\sigma_{\tilde{\sigma}}}_s, \ \ \overline{f}^{\tilde{\sigma}}=f^{\sigma_{\tilde{\sigma}}}, \ \  \overline{g}^{\tilde{\sigma}}_{s,t}:=g^{\sigma_{\tilde{\sigma}}}_{s,t}
          \end{equation}
          defines a reflexive partial homotopy $\left\{(\overline{Z}^{\tilde{\sigma}}, (\overline{g}^{\tilde{\sigma}}_{s,t})_{s\in \tilde{\sigma}, t\in [0,1]}, \overline{f}^{\tilde{\sigma}}, (\overline{\psi}^{\tilde{\sigma}}_s)_{s\in\tilde{\sigma}} \right\}_{\tilde{\sigma}\in \tilde{\mK}}$
            at time $T$ satisfying the following property: if $\left\{(Z^{\sigma}, (g^{\sigma}_{s,t})_{s\in \sigma, t\in [0,1]},f^{\sigma} ,(\psi^{\sigma}_s)_{s\in\sigma}) \right\}_{\sigma\in \mK}$ is reflexively PSC-conformal over some $s\in K$ then so is $\left\{(\overline{Z}^{\tilde{\sigma}}, (\overline{g}^{\tilde{\sigma}}_{s,t})_{s\in \tilde{\sigma}, t\in [0,1]},\overline{f}^{\tilde{\sigma}},(\overline{\psi}^{\tilde{\sigma}}_s)_{s\in\tilde{\sigma}})\right\}_{\tilde{\sigma}\in \tilde{\mK}}$.
        \end{lemma}
        
           \begin{proof}
         Straightforward from the definition.
        \end{proof}
        
        Two more operations, of fundamental importance but rather technical character, are the object of Appendix \ref{app:EquivExtHom} (devoted to the equivariant extension of a partial homotopy) and of Appendix \ref{app:EquivDiskRemov} (devoted to the equivariant removal of 3-dimensional disks from a partial homotopy). For the purposes of this discussion, the reader may just appeal to a rather heuristic notion of these operations and then refer to such sections for all details.
        

           Now, the crucial lemma needed to realise the backward-in-time induction is the following statement, which collects some preparatory constructions that are employed when concretely implementing the inductive step.

 \begin{lemma}\label{lem:BackwInductPreparation}
 In the setting of Proposition \ref{prop:SurjPartHom}, there exists a simplicial refinement $\tilde{\mK}$ of $\mK$ (whose geometric realisation we identify, i.e. $K=\tilde{K}$) such that, if we adapt the given reflexive partial homotopy 
 
  \[
        \left\{(Z^{\sigma}, (g^{\sigma}_{s,t})_{s\in \sigma, t\in [0,1]}, f^{\sigma}, (\psi^{\sigma}_s)_{s\in\sigma} \right\},
        \]
that is postulated to satisfy conditions (1), (2), (3), (4) of Proposition \ref{prop:BackwardFlow} ,       
    to $\tilde{\mK}$ in the sense of Lemma \ref{lem:SimplicialRefine}, we may assume for any simplex $\sigma\in \tilde{\mK}$ to have the following additional data:
\begin{enumerate}
    \item a compact manifold with boundary $\hat{Z}^{\sigma}$ and an embedding $\iota^{\sigma}:Z^{\sigma}\to \hat{Z}^{\sigma}$,
    \item an involutive diffeomorphism $\hat{f}^{\sigma}: \hat{Z}^{\sigma}\to \hat{Z}^{\sigma}$ satisfying the assumption $(\star_{\text{sep}})$ ;
    \item a continuous family of embeddings $(\hat{\psi}^{\sigma}_s: \hat{Z}^{\sigma}\to \mM^s_T)_{s\in\sigma}$ satisfying $\hat{\psi}^{\sigma}_s \circ \hat{f}^{\sigma}=f^s \circ\hat{\psi}^{\sigma}_s$ for all $s\in\sigma$;
    \item continuous family of embeddings $(\nu^{\sigma}_{s,j}:D^3\to \mM^s_T)_{s\in\sigma, j=1,\ldots, m_{\sigma}}$ for some non-negative integer $m_{\sigma}\in\mathbb{N}$, so that there exists $p_{\sigma}, q_{\sigma} \in\mathbb{N}$ with $2p_{\sigma}+q_{\sigma}=m_{\sigma}$ and:
    \begin{enumerate}
    \item if $j\leq 2p_{\sigma}$ then $\nu^{\sigma}_{s,j}(D^3)\cap \Fix(f^s)=\emptyset$ and $\nu^{\sigma}_{s,j+p_{\sigma}}=f^{s}_T\circ \nu^{\sigma}_{s,j+p_{\sigma}}$ for all $s\in\sigma$; 
    \item if $2p_{\sigma}+1\leq j\leq m_{\sigma}$ then $\nu^{\sigma}_{s,j}(D^3)\cap \Fix(f^s)\neq\emptyset$ and
     $\nu^{\sigma}_{s,j}\circ f_0 = f^s \circ \nu^{\sigma}_{s,j}$;
    \end{enumerate}
\end{enumerate}
and that, for all $\sigma\in\tilde{K}$ and $s\in \sigma$, 
all these conditions are met ($k\geq 0$ denoting the dimension of $\sigma$):
\begin{enumerate}[label=(\roman*)]
    \item $\iota^{\sigma}(Z^{\sigma})$ is $\hat{f}^{\sigma} $-invariant and $\iota^{\sigma}\circ f^{\sigma}=\hat{f}^{\sigma}\circ \iota^{\sigma}$,
    \item $\hat{\psi}^{\sigma}_s(\hat{Z}^{\sigma})$ is $f^s_T$-invariant and $\hat{\psi}^{\sigma}_s\circ \hat{f}^{\sigma}=f^s_{T}\circ \hat{\psi}^{\sigma}_s$, 
    \item $\psi^{\sigma}_s = \hat{\psi}^{\sigma}_s \circ\iota^{\sigma}$,
    \item the closure $Y$ of any connected component of $\hat{Z}^{\sigma}\setminus \iota^{\sigma}(Z^{\sigma})$ satisfies either of the following two properties, uniformly in $s\in\sigma$:
    \begin{enumerate}
        \item $\hat{\psi}^{\sigma}_s(Y)$ is a union of (possibly singular) fibers of $\mS^s$,
        \item $\partial Y=\emptyset$, $\hat{\psi}^{\sigma}_s(Y)\subset U^s_{S3}$ and $(\hat{\psi}^{\sigma}_s)^{\ast}(g'^{,s}_T)$ is an $s$-dependent multiple of the same constant curvature metric, and if $Y\cap \Fix(\hat{f}^{\sigma})\neq\emptyset$ then $(Y, (\hat{\psi}^{\sigma}_s)^{\ast}(g'^{,s}_T), \hat{f}^{\sigma})$ is a reflexive 3-manifold,
    \end{enumerate}
    \item $\hat{\psi}^{\sigma}_s(\hat{Z}^{\sigma})\setminus \psi^{\sigma}_s(Z^{\sigma})\subset\left\{\rho>2r_k\right\}$, $\hat{\psi}^{\sigma}_s(\partial \hat{Z}^{\sigma})\setminus \psi^{\sigma}_s(\partial Z^{\sigma})\subset\left\{\rho>2\Lambda r_k\right\}$, and every component of $\hat{\psi}^{\sigma}_s(\hat{Z}^{\sigma})$ that does not contain a component of $\psi^{\sigma}_s(Z^{\sigma})$ contains a point where $\rho>2\Lambda^2 r_k$;
    \item $\left\{\rho>\Lambda^3 r_k/3 \right\}\cap \mM^s_T\subset \hat{\psi}^{\sigma}_s(\hat{Z}^{\sigma})$, and in every component of $\mM^s_T\setminus \hat{\psi}^{\sigma}_s(\text{Int} (\hat{Z}^{\sigma}))$ that intersects $\hat{\psi}^{\sigma}_s(\hat{Z}^{\sigma})$ there is a point where $\rho<4 r_k$;
    \item the images of the embeddings $(\nu^{\sigma}_{s,j})_{j=1,\ldots, m_{\sigma}}$ are pairwise disjoint, satisfy $\nu_{s,j}(D^3)\subset \psi^{\sigma}_s(\text{Int}(Z^{\sigma}))\cap U^s_{S2}$, and, in addition, for all $j=1,\ldots, m_{\sigma}$ the pull-back through the embedding $\nu_{s,j}$ of $\mS^s$ coincides with the standard spherical structure of $D^3$,
    \item for all $j=1,\ldots, m_{\sigma}$ one has $\nu^{\sigma}_{s,j}(D^3)\subset \left\{\rho<4\Lambda r_k\right\}$,  $\nu^{\sigma}_{s,j}(\partial D^3)\subset \left\{\rho>2\Lambda r_k\right\}$, and
       $\nu^{\sigma}_{s,1}(D^3)\cup\ldots\cup\nu^{\sigma}_{s,m_{\sigma}}(D^3)$ contains all fibers of $\psi^{\sigma}_s(Z^{\sigma})\cap U^s_{S2}$ that are points where $\rho<4r_k$.
\end{enumerate}
\end{lemma}
 
 For the proof of this preparatory result one can follow the procedure described in the non-equivariant case (pages 99-101 in \cite{BamKle19}, which will indeed provide data with the desired equivariance property. We shall however remark what follows. In defining the map $\nu^{\sigma}_{s_0,j}$ one first looks at the set of points in $\psi^{\sigma}_{s_0}(Z^{\sigma})$ that are singular points for the spherical structure $\mS^{s_0}$ and where $\rho\leq 4 r_k$, and defines $Y$ to be the union of all open $D^3$s such that $X\setminus \left\{x\right\}$ is a union of spherical fibers where $\rho\leq 3\Lambda r_k$.

 Now, keeping in mind our preliminary discussion in Section \ref{sec:OutlineParab} (see in particular Lemma \ref{lemma.SingFiberReflexivePointCase}) then one, and only one, of these two situations will occur:
 \begin{enumerate}
     \item \emph{either} $x\in \Fix(f^{s_0})$ and so all spherical fibers of $Y$ will meet $\Fix(f^{s_0})$ orthogonally along a circle,
     \item \emph{or} $x\notin \Fix(f^{s_0})$, in which case $Y$ must be disjoint from $\Fix(f^{s_0})$ (for, if not, one can perform the same construction for $x'=f^s(s)$, thereby obtaining an open $D^3$, denoted by $Y'$, such that necessarily the closure of $Y\cup Y'$ would equal the whole connected component containing $x, x'$, which however contradicts condition (2), which is an \emph{open} condition, in the statement of Proposition \ref{prop:SurjPartHom}).
 \end{enumerate}

 We can finally proceed with the proof of Proposition \ref{prop:SurjPartHom}, which completes this section.
     
     \begin{proof}
     We assume to be given a reflexive partial homotopy   \begin{equation}\label{eq:AssignedPartHom}
        \left\{(Z^{\sigma}, (g^{\sigma}_{s,t})_{s\in \sigma, t\in [0,1]}, f^{\sigma}, (\psi^{\sigma}_s)_{s\in\sigma} \right\},
        \end{equation}
        at time $T$, satisfying the four conditions stated in Proposition \ref{prop:SurjPartHom}. We then consider the simplicial refinement $\tilde{\mK}$ and the additional data constructed in Lemma \ref{lem:BackwInductPreparation}, which we employ as follows. For any $k\geq 0$ we consider the collection of all simplices of $\tilde{\mK}$ having dimension \emph{exactly equal to} $k$ and proceed, in finitely many steps, starting from the given homotopy (above) and invoking, at each step, first Theorem \ref{thm:ExtPartHom} and then Theorem \ref{thm:DiskRemPartHom} where the input data we employ are only those corresponding to the simplices of dimension $k$. Thus, one verifies as in the non-equivariant case (cf. Lemma 8.13 in \cite{BamKle19}) that, after we have accounted for all simplices of $\tilde{\mK}$ we obtain a reflexive partial homotopy 
        \begin{equation}\label{eq:PreparedPartHom}\left\{(\tilde{Z}^{\sigma}, (\tilde{g}^{\sigma}_{s,t})_{s\in \sigma, t\in [0,1]}, \tilde{f}^{\sigma}, (\tilde{\psi}^{\sigma}_s)_{s\in\sigma} \right\}
        \end{equation}
     (still at time $T$, for the very same reflexive $\mR$-structures) such that
        such that all sets $X^{\sigma}_s:=\tilde{\psi}^{\sigma}_s(\tilde{Z}^{\sigma})\setminus (\nu^{\sigma}_{s,1}(B^3)\cup\ldots\cup\nu^{\sigma}_{s,m_{\sigma}}(B^3))$ survive until time $T-\Delta T$, and denoted by $X^{\sigma}_s(T-\Delta T)$ its image under the backward flow, for time $0<\Delta T\leq \min\left\{T,\theta(r_0)\right\}$,  
           the following assertions hold true for any $k$-simplex $\sigma$ of $\tilde{\mK}$ and all $s\in\sigma$:
        \begin{enumerate}
            \item $\left\{\rho>\Lambda^3 r_k \right\}\cap \mM^s_{T-\Delta T}\subset X^{\sigma}_s(T-\Delta T) \subset \left\{\rho>r_k\right\}$;
            \item Every component of $X^{\sigma}_s$ contains a point $x$ such that $\rho(x(T-\Delta T))>\Lambda^2 r_k$;
            \item $\rho>\Lambda r_k$ on $\partial X^{\sigma}_s(T-\Delta T)$;
            \item the partial homotopy \eqref{eq:PreparedPartHom} is reflexively $PSC$-conformal over every $s\in K'$.
        \end{enumerate}
      This is checked exactly as for Lemma 8.9 in \cite{BamKle19}; note that the desired equivariance properties for this (preparatory) partial homotopy follow from the corresponding properties of the pre-assigned homotopy \eqref{eq:AssignedPartHom} as well as of the `input data' $\hat{Z}^{\sigma}, \hat{f}^{\sigma}, \hat{\psi}^{\sigma}_s, \nu^{\sigma}_{s,j}$.
      
      At that stage, we apply Proposition \ref{prop:BackwardFlow} to such a reflexive partial homotopy, obtaining a reflexive partial homotopy at time $T-\Delta T$ for the singular Ricci flow and the reflexive $\mR$-structures in question. As indicated above, we already know that properties (1), (2) and (3) in the statement of Proposition \ref{prop:SurjPartHom} hold true, while (4) is ensured by the very statement of Proposition \ref{prop:BackwardFlow} thanks to item (3) in Lemma \ref{lem:BackwInductPreparation}.
     \end{proof}


        \section{Applications to initial data sets in general relativity}\label{sec:GR}

    In this section we derive some applications of Theorem \ref{thm:MainOpen} to mathematical general relativity. We fix a background manifold $X_\infty$ that is diffeomorphic to $\R^3\setminus \left(\sqcup_{j=1}^{\ell} P_{\gamma_j}\right)$, where $\{P_{\gamma_j}\}_{j=1}^\ell$ is a \emph{separated} collection of handlebodies of genus $\gamma_j\ge 0$, with the \emph{standard embedding} into $\R^3$.
    Although we restrict our attentions in dimension $3$, many auxiliary results in this section hold for general dimensions.
    
    \begin{remark}
        Recall that $\{P_{\gamma_j}\}_{j=1}^\ell$ is separated, if there exist disjoint 3-balls $B_j\subset \R^3$ such that $P_{\gamma_j}\subset B_j$. We also recall that an embedding $P_{\gamma_j}\hookrightarrow \R^3$ is standard, if it is given by the standard boundary connected sum of $j$ unknotted solid tori in $\R^3$. 
    \end{remark}

    We work with the weighted Sobolev spaces $W^{k,p}_\beta$ defined by the norm
    \begin{equation*}
        \|u\|_{k,p,\beta}=\sum_{0\le i\le k}  \|\rho^{i-\beta-\tfrac 3p}\nabla^i u\|_{L^p(X_\infty)}.
    \end{equation*}
    Here $\rho=(1+|x|^2)^{\tfrac12}$. Fix an integer $k\ge 2$, $p> \tfrac 3 k$ and $\beta\in (-1,0)$. In particular, $W^{k,p}_\beta$ embeds into $C^{0,\alpha}_\beta$ for some $\alpha\in (0,1)$, and $W^{k,p}_\beta$ is closed under multiplication. We refer the readers to \cite[Section 2, Lemma 1]{Maxwell2005solutions} and the reference therein for standard embedding and structural theorems on the weighted Sobolev spaces. 
    
    The space of asymptotically flat metrics is here defined as:
    \[\metaf := \left\{\text{metrics } g \text{ on } X_\infty: g_{ij}-\delta_{ij}\in W^{k,p}_\beta\right\}.\]
    We are interested in the following subspaces of metrics:
    \[\metaf_{R=0, H=0}:=\left\{\text{metrics }g \text{ on } X_\infty: g_{ij}-\delta_{ij}\in W^{k,p}_\beta, R_g=0, H_g=0\right\},\]
    \[\metaf_{R\ge 0, H=0}:=\left\{\text{metrics }g \text{ on }X_\infty: g_{ij}-\delta_{ij}\in W^{k,p}_\beta, R_g\ge 0, H_g=0\right\},\]
    and the spaces of metrics $\metaf_{R=0, H\ge 0}, \metaf_{R\ge 0, H\ge 0}$ defined similarly.
    
    \begin{remark}
        For time symmetric initial data in $\metaf_{R\ge 0, H=0}$ and $\metaf_{R\ge 0, H\ge 0}$, a natural extra assumption is $R_g\in L^1$. This is to guarantee that the ADM mass is well-defined. We note here that although Theorem \ref{thm:GR constractible} does not need this assumption, the same conclusion holds for the spaces of metrics with the additional assumption that $R_g\in L^1$.
    \end{remark}
    
    We start with a few additional facts on these weighted Sobolev spaces.
    
    \begin{proposition}[{\cite[Section 3, Proposition 1]{Maxwell2005solutions}}]\label{proposition.weighted.sobolev.fredholm}
        Let $(X_\infty,g)$ be asymptotically flat of class $W^{k,p}_\beta$, $k\ge 2$, $k>3/p$, and $\beta\in (-1,0)$. Then for any $A\in W^{k-2,p}_{\beta-2}(X_\infty), B\in W^{k-1-\tfrac 1p, p}(\partial X)$, the operator
        \begin{equation}\label{equation.operator.gr}
            \left(-\Delta_g + A, (\partial_V+B)|_{\partial X_\infty}\right): W^{k,p}_\beta(X_\infty)\to W^{k-2,p}_{\beta-2}(X_\infty)\times W^{k-1-\tfrac 1p, p}(\partial X_\infty)
        \end{equation}
        is Fredholm of index 0. Moreover, if $A, B \ge 0$ then it is an isomorphism.
    \end{proposition}
    


    We would also need a suitable version of the maximum principle for the operator \eqref{equation.operator.gr}.
    
    \begin{proposition}[{\cite[Section 3, Lemma 2]{Maxwell2005solutions}}]\label{proposition.gr.max.principle}
        Suppose $(X_\infty,g), A, B$ satisfy the conditions in Proposition \ref{proposition.weighted.sobolev.fredholm}, and assume in addition that $A\ge 0$, $B\ge 0$. Then any $u\in W^{k,p}_\beta$ with
        \[
        \begin{cases}
            -\Delta_g u + A u\le 0,\\
            \hspace{4mm} \partial_V u + B u\le 0
        \end{cases}
        \]
        satisfies $u\le 0$.
    \end{proposition}

    In the first theorem of this section, we prove the first part of Theorem \ref{thm:GR constractible}.
    
    \begin{theorem}\label{thm.gr.all.R.H}
        The spaces $\metaf_{R=0,H=0}, \metaf_{R\ge 0, H=0}, \metaf_{R=0, H\ge 0}, \metaf_{R\ge 0, H\ge 0}$ are (not empty and) contractible.    
    \end{theorem}
    The proof is divided into several steps. We first show that $\metaf_{R=0,H=0}$ is nonempty. Denote by $\varphi: X_\infty\to S^3, \infty\mapsto x_{\infty}$ the inverse of the stereographic projection, and $X=\im(\varphi)$ a compact 3-manifold that is diffeomorphic to $S^3\setminus \left(\sqcup_{j=1}^\ell P_{\gamma_j}\right)$. In analogy with Theorems \ref{thm:ReductionToDouble}, \ref{thm:ContractDoubling} and \ref{thm:ANRstructure}, we then prove that these spaces of metrics are weakly contractible, and finally that each of them is an ANR. 
    
    \begin{lemma}\label{lemma.gr.nonempty}
        Suppose $X_\infty = \R^3\setminus \left(\sqcup_{j=1}^{\ell} P_{\gamma_j}\right)$, and $\{P_{\gamma_j}\}_{j=1}^\ell$ are separated genus $\gamma_j$ handlebodies with the standard embeddings into $\R^3$. Then there exists an asymptotically flat metric $g$ on $X_\infty$ with $R_g=0$ and $H_g=0$.
    \end{lemma}
    
    \begin{proof}
        Since $P_{\gamma_j}$ are separated, $X$ is diffeomorphic to $\left(S^3\setminus P_{\gamma_1}\right)\#\cdots\#\left(S^3\setminus P_{\gamma_\ell}\right)$. Since each $P_{\gamma_j}$ is a standard embedding, $S^3\setminus P_{\gamma_j}$ is diffeomorphic to $P_{\gamma_j}$. By \cite[Theorem 1.1]{CarLi19}, $X$ supports a metric $g$ with $R_g>0$ and $H_g=0$.
        
        Let $G$ be the Green's function for the conformal Laplacian with Neumann boundary condition on $(X,g)$, with pole at $x_{\infty}$; it is readily checked that (since $R_g>0$) this must be a positive function. Then the metric $g_{AF}=\phi^*(G^4 g)$ is an asymptotically flat metric of class $W^{k,p}_\beta$ on $X_\infty$, for any $k,p$ and $\beta\in (-1,0)$, with $R_{g_{AF}}=0$ and $H_{g_{AF}}=0$.
    \end{proof}

    \begin{proposition}\label{prop.gr.R=0.H=0}
        The space $\metaf_{R=0,H=0}$ is weakly contractible.
    \end{proposition}
    
    \begin{remark*}
        In particular, $\metaf_{R=0,H=0}$ is path connected. This extends \cite[Theorem 1.3]{HirschLesourd2021moduli} to manifolds with general topology.
    \end{remark*}

    For any nonnegative integer $\ell$, fix a continuous map $g: S^\ell\to \metaf_{R=0,H=0}$. We will prove that this map extends to a continuous map $D^{\ell+1}\to \metaf_{R=0,H=0}$, and thus $\pi_\ell(\metaf_{R=0,H=0})=0$. As in \cite{Mar12,HirschLesourd2021moduli}, we divide the proof into 3 steps.
    
    \begin{proposition}\label{prop.gr.step1}
        $g:S^\ell \to \metaf_{R=0,H=0}$ extends to a homotopy $g: S^\ell\times [0,1]\to \metaf_{R=0,H=0}$ such that each $g(\xi,1)$ is smooth and conformally flat outside a compact set.
    \end{proposition}

  \begin{proof}
        Fix a smooth cutoff function $0\le \eta\le 1$ such that $\eta(t)=1$ for $t\le 1$ and $\eta(t)=0$ for $t\ge 2$. Set $\eta_R(t)=\eta(t/R)$. Take $R_0>0$ such that $B_{R_0}(0)$ enclosed all the compact boundary components of $X_\infty$. For $R>R_0$ and $\xi \in S^\ell$, define $g_R(\xi)=\eta_R g(\xi) + (1-\eta_R) \delta$, and approximate $g_R(\xi)$ by a smooth metric $g_R'(\xi)$ such that $g_R':S^\ell\to \metaf$ is continuous, $\supp (g_R-g_R')\subset B_{4R}(0)$, and $\|g_R-g_R'\|_{k,p,\beta}$ is sufficiently small for all $\xi$. For $\mu \in [0,1]$, define $g_R(\xi, \mu)=(1-\mu)g(\xi)+ \mu g_R'(\xi)$. 
        
        Note that 
        \[(\Delta_{g(\xi)} -\tfrac 18 R_{g(\xi)}, \partial_V +\tfrac14 H_{g(\xi)})=(\Delta_{g_\xi}, \partial_V):W^{k,p}_\beta (X_\infty) \to W^{k-2,p}_{\beta-2}(X_\infty)\times W^{k-1-\tfrac 1p, p}(\partial X_\infty)\]
        is an isomorphism by Proposition \ref{proposition.weighted.sobolev.fredholm}. Hence there exists $\eps(g(\xi))>0$ such that the analogous operator is an isomorphism for all $\tilde g$ with $\|\tilde g  - g(\xi)\|_{k,p,\beta}< \eps (g(\xi))$. Thus, by choosing $R$ sufficiently large and then $g_R'(\xi)$ sufficiently close to $g_R(\xi)$, such that $\|g_R'(\xi)-g(\xi)\|_{k,p,\beta}<\eps(g(\xi))$, we have that $\|g_R(\xi,\mu)-g_R(\xi)\|_{k,p,\beta}<\eps(g(\xi))$. Therefore the operator
        \[\left(\Delta_{g_R(\xi,\mu)}-\tfrac18 R_{g_R(\xi,\mu)}, \partial_V+\tfrac14 H_{g_R(\xi,\mu)}\right): W^{k,p}_\beta(X_\infty)\to W^{k-2,p}_{\beta-2}(X_\infty)\times W^{k-1-\tfrac1p, p}(\partial X_\infty)\]
        is also an isomorphism. Let $v_R(\xi,\mu)\in W^{k,p}_\beta(X_\infty)$ be the solution to
        \[\left(\Delta_{g_R(\xi,\mu)}-\tfrac18 R_{g_R(\xi,\mu)}, \partial_V+\tfrac14 H_{g_R(\xi,\mu)}\right)(v)=\left(\tfrac18 R_{g_R(\xi,\mu)}, -\tfrac14 H_{g_R(\xi,\mu)} \right).\]
        Then $g(\xi,\mu):=(1+v_R(\xi,\mu))^4 g_R(\xi,\mu)$ is the desired homotopy.
      
    \end{proof}
    
    Thus, without loss of generality we may assume $g:S^\ell\to \metaf_{R=0,H=0}$ is a continuous family of smooth metrics that are conformally flat outside a compact set. To proceed, recall that $\varphi: X_\infty \to S^3$, $\infty\mapsto x_{\infty}$ is the inverse of the stereographic projection, and $X=\im(\varphi)$. Moreover, we fix $I(z)=\tfrac{z}{|z|^2}$ the inversion with respect to the unit sphere in $\R^3$, and $\psi: B_1(0)\to S^3$ be such that $\psi^{-1}(\varphi(z))=I(z)$.
    
    \begin{proposition}\label{prop.gr.step2}
        Let $g:S^\ell\to \metaf_{R=0,H=0}$ be a continuous family of smooth metrics that are conformally flat outside a compact set. Then there exists a continuous family of metrics $\overline g: S^\ell \to \met(X)$ of positive Yamabe type and minimal boundary, such that:
        \begin{enumerate}
            \item $\exp_{x_{\infty},\overline{g}(\xi)}^{-1} (\varphi(z))=I(z)$ outside a compact set;
            \item $g(\xi)=\varphi^*(G^4\overline{g}(\xi))$, where $G$ is the Green's function of the conformal Laplacian $L_{\overline{g}(\xi)}$ with Neumann boundary condition and a pole at $x_{\infty}$, i.e. the distributional solution to 
            \[L_{\overline{g}(\xi)}G=-4\pi \delta_{x_{\infty}},\quad \partial_V G=0.\]
        \end{enumerate}
    \end{proposition}

      \begin{proof}
      By compactness, there exists $R>0$ and a continuous family of smooth positive functions $u\in 1+W^{k,p}_{\beta}$ such that $g(\xi)(z)=u(\xi)^4(z)\delta$, if $|z|\ge R$. Let $v(\xi)$ be a smooth positive function on $X_{\infty}$ such that $v(\xi)(z)=|z|u(\xi)(z)$ for $|z|\ge R$, and $v(\xi)=1$ in a neighborhood of $\partial X_\infty$. Define $g(\xi)'=v^{-4}g(\xi)$, and $\overline g(\xi)=\varphi_*(g'(\xi))$. This defines a continuous family of metrics $g: S^\ell\to \met(X)$. For each $\xi\in S^\ell$, we observe:
      \begin{enumerate}[label=(\roman*)]
          \item $H_{g'}=H_g=0$ on $\partial X_\infty$, as $v=1$ near $\partial X_\infty$.
          \item Around $x_{\infty}$, $\overline{g}=\varphi_*(|z|^{-4}\delta)=\psi_*(\delta)$, hence $\psi=\exp_{p,\overline g}$.
          \item We set $G=v\circ\varphi^{-1}$. Since $R_{g}=0$ and $G^4 \overline g$ is isometric (via the mapping $\varphi$) to $v^4 g'=g$, we have $L_{\overline g}(G)=0$. Similarly, since both $\overline g$ and $g$ have minimal boundary, $\partial_V G=0$. Moreover, note that $G>0$ and $\lim_{x\to 0}|x|G(\psi(x))=1$. Therefore $G$ is the Green's function of the conformal Laplacian with the Neumann boundary condition and a pole at $x_{\infty}$.
          \item Since, by construction, $G$ is everywhere positive, the first Neumann eigenvalue of $L_{\overline g (\xi)}$ is positive. In fact, suppose $f>0$ is the first eigenfunction such that $L_{\overline g(\xi)} f = -\lambda_1 f$ and $\partial_V f=0$, then integration by parts gives:
          \[4\pi f(p)= -\int_X\left(L_{\overline g (\xi)}f\right)G = \lambda_1 \int_X fG.\]
          Hence $\lambda_1>0$. This implies that $\overline g(\xi)$ is of positive Yamabe type with minimal boundary.
      \end{enumerate}
    \end{proof}
    
    To prove Proposition \ref{prop.gr.R=0.H=0}, we will also need the following Lemma on the existence of diffeomorphisms on a continuous family of Riemannian metrics, extending the exponential map near a point. We remark that a slightly different result was proved in \cite[Lemma 4.5]{HirschLesourd2021moduli}.

    \begin{lemma}\label{lemma.diffeom.exponential.map}
     Let $K$ be a compact topological space, $\{g(\xi)\}_{\xi\in K}$ be a continuous family of smooth Riemannian metrics on $B_1(0)\subset \R^n$. Then there exist $\varepsilon>0$ and a continuous family of diffeomorphisms $F(\xi): B_1(0)\to B_1(0)$ such that:
     \begin{enumerate}
         \item $F(\xi)=\exp_{0,g(\xi)}$ in a geodesic ball of radius $\eps$.
         \item $F(\xi)$ is the identity map in $B_1(0)\setminus B_{2/3}(0)$.
     \end{enumerate}
     Moreover, if $g(\xi)=\delta$, then we can choose $F(\xi)$ to be the identity map.
    \end{lemma}

        \begin{proof}
      For each $\xi\in K$, let $B^{(\xi)}(r), S^{\xi}(r)$ be the geodesic ball and geodesic sphere of radius $r$ centered at $0$, respectively. For each $g(\xi)$, consider the rescaled Riemannian manifolds $(B_r^{(\xi)}, g_r(\xi)):=r^{-1}(B^{(\xi)}(r), g(\xi)|_{B^{(\xi)}(r)})$. They converge smoothly in the unit ball to a flat Riemannian metric on $\R^n$, whose metric tensor at each point is given by $g(\xi)(0)$. We denote this limit by $(B^{(\xi)}, g_0(\xi))$. Then $B^{(\xi)}$ is an ellipsoid with respect to the standard Euclidean metric $\delta$. To proceed, we observe the following basic fact.
      
      \begin{fact}
          Let $\Omega$ be a open set in $\R^n$, $p$ be a point in $\Omega$ and $g_i$ a sequence of smooth Riemannian metrics on $\Omega$, assumed to converge locally smoothly to a Riemannian metric $g$. Then for $r<\min\{d_{g_i}(p,\partial \Omega)\}$, the exponential maps $\exp_{p,g_i}(r\cdot)\to \exp_{p,g}(r\cdot)$ on the unit tangent bundle at $p$.
      \end{fact}
      
      Using this fact, we see that for $r<\min\{\inj(g_i,p)\}$, the $g_i$-geodesic spheres of radius $r$ centered at $p$ (regular level sets of $d_{g_i}(p,\cdot)$) converge smoothly graphically as $r\to 0$. We also observe that $\exp_{p, g_r(\xi)}\to \exp_{p, g_0(\xi)}$ among all unit tangential vectors. 
      
      For a point $x\in B_1(0)$, $\exp_{0,g_0(\xi)}(x)=\|x\|_{g_0(\xi)}^{-1}x$, and hence $\left\|\exp_{0,g_r(\xi)}(x)-|x|_{g_0(\xi)}^{-1}x\right\|_{\delta}\to 0$ as $r\to 0$, here the distance is measured in $B_1(0)$ with respect to the Euclidean metric $\delta$. Consider the continuous map $F_r: B_1(0)\to B_1(0)$ defined as follows: for $x\in B_r(0)$, $F_r(x)=\exp_{0,g(\xi)}(x)$; for $x\in B_1(0)\setminus B_{1/2}(0)$, $F_r(x)=x$; for $x\in B_{1/2}(0)\setminus B_r(0)$, $F_r$, restricted to each line segment connecting $\tfrac{rx}{\|x\|_{\delta}}$ and $\tfrac{x}{2\|x\|_{\delta}}$, is a linear interpolation that maps to the Euclidean line segment connecting $\exp_{0,g(\xi)}\left(\tfrac{rx}{\|x\|_{\delta}}\right)$ and $\tfrac{x}{2\|x\|_{\delta}}$. As $r\to 0$, the rescaled map $\bar F_r: B_{1/r}(0)\to B_{1/r}(0)$, defined as $\bar F_r (r)=\tfrac{1}{r}  F_r(rx)$, converges to the linear interpolation between the map that takes the Euclidean unit ball to the ellipsoid $B^{(\xi)}$, and the identity map near infinity, and in particular is continuously invertible. Hence there exists $\eps(\xi)>0$ such that $ F_r$ is a continuous invertible map when $r\in (0,\eps(\xi))$. Since $K$ is compact, we can take $\inf_{\xi}\eps(\xi)>\eps>0$, and then the map $ F_\eps(\xi)$ is well-defined, continuous and has continuous inverse. 
      
      By definition $ F_\eps(\xi)$ depends continuously on $\xi$, and it is in fact the identity map when $g(\xi)=\delta$; also $F_\eps(\xi)=\exp_{0,g(\xi)}$ in $B^{(\xi)}(\eps)$. $F_\eps(\xi)$ is piecewise smooth, with two singular interfaces $S^{(\xi)}(\eps)$ and $\partial B_{1/2}(0)$, along which it is only continuous. We then take $ F(\xi)$ to be a smooth diffeomorphism by applying a standard mollification on $ F_\eps(\xi)$, in a small neighborhood of the interfaces whose size is proportional to $\|g_\eps(\xi)-\delta\|_{B_1(0)}$, such that $ F(\xi)= F_\eps(\xi)$ agree in $B_{\eps/2}(0)$ and $B_1(0)\setminus B_{2/3}(0)$.
      
    \end{proof}
    
    We are now ready to prove Proposition \ref{prop.gr.R=0.H=0}.
    
     \begin{proof}[Proof of Proposition \ref{prop.gr.R=0.H=0}]
      By Proposition \ref{prop.gr.step1}, we may assume, without loss of generality, that for each $\xi\in S^\ell$ the metric $g(\xi)$ is smooth and conformally flat outside a compact set. Applying Proposition \ref{prop.gr.step2}, we obtain a continuous family of smooth, Yamabe positive metrics $\overline g: S^\ell\to \met (X)$ with minimal boundary. Note that $g(\xi)=\varphi^*(G^4\overline g(\xi))$, where $G(\xi)$ is the Green's function to the conformal Laplacian on $(X,\overline g)$ with the Neumann boundary condition and a pole at $x_{\infty}$. 
      
      For each $\xi\in S^\ell$, let $u(\xi)$ be the first eigenfunction for the conformal Laplacian on $(X,\overline g(\xi))$ with Neumann boundary condition, normalized so that $\|u(\xi)\|_{(L^2(X,\overline g))}=1$. Then the map $\overline g:S^\ell\to \met(X)$ extends to a homotopy $S^\ell\times[0,1]\to \met(X)$, $\overline g(\xi,\mu)=(1-\mu +\mu u(\xi))^4 \overline g(\xi)$, such that $\overline g(\cdot, 1)\in \met_{R> 0, H=0}$. Note that each $\overline g(\xi,\mu)$ is in the conformal class $[g(\xi)]$, and hence is Yamabe positive with minimal boundary. Since $\met_{R>0 , H=0}$ is contractible (Theorem \ref{thm:MainOpen}), this map further extends to to a map $\overline g: D^{\ell+1}\to \met(X)$ where each $\overline g(\xi)$ is Yamabe positive with minimal boundary. 
      
      Since $D^{\ell+1}$ is compact, by choosing a sufficiently small $\tau>0$, we assume $\psi: (B_\tau(0),0)\to (S^3,x_{\infty})$ is a coordinate neighborhood of $x_\infty$ in each $(X,\overline g(\xi))$, that is contained in the geodesic normal neighborhood of $\overline g(\xi)$ near $x_{\infty}$. Moreover, when $\xi\in S^\ell$, $\psi^*(\overline g(\xi))=\delta$ in $B_\tau (0)$. Now we apply Lemma \ref{lemma.diffeom.exponential.map} to the metrics $\psi^*(\overline g(\xi))$ (in $B_\tau(0)$ instead) as one varies $\xi\in D^{\ell+1}$, and obtain a continuous family of diffeomorphisms $F(\xi): B_\tau(0)\to B_\tau(0)$. We then define $\tilde\psi = \psi\circ F: B_\tau(0)\rightarrow S^3$, and set
      \[\tilde \varphi=\begin{cases} \varphi \quad &\text{ in }B_{1/\tau}(0); \\ \tilde \psi\circ I \quad & \text{ in }\R^3\setminus B_{1/\tau}(0).\end{cases}\]
      Consequently, for $\xi\in D^{\ell+1}$, we further define 
      \[\tilde g(\xi)=\tilde \varphi^*(G(\xi)^4 \overline g(\xi)).\]
      Here, consistently with the notation we have introduced above, we still denote by $G(\xi)$ the Green's function of the conformal Laplacian on $(X,\overline g(\xi))$ with Neumann boundary condition on $(X,\overline g(\xi))$ and pole at $x_{\infty}$, for the extended map $\overline g: D^{\ell+1}\to \met(X)$. For $\xi \in S^\ell$, since $\psi^*(\overline g(\xi))=\delta$ in $B_\tau (0)$, $F(\xi)$ is the identity map, and so (keeping in mind part (2) of Proposition \ref{prop.gr.step2}) $\tilde g(\xi)=g(\xi)$ for such $\xi$. Moreover, since $F(\xi)=\exp_{0,\psi^* (\overline g(\xi))}$ near $0$, we have $\delta=F(\xi)^*(\psi^*(\overline g(\xi))) = (\psi\circ F(\xi))^*(\overline g)$ near $0$. Thus $\tilde \psi=\psi\circ F=\exp_{x_{\infty},\overline g(\xi)}$ in a neighborhood of $x_{\infty}$, which (together with the form of the expansion of the Green's function in normal coordinates, see \cite{Sch84}) ensures that indeed $\tilde{g}(\xi)$ is a smooth asymptotically flat metric in $\metaf_{R=0, H=0}$. 
    \end{proof}

    We proceed to consider asymptotically flat metrics with other boundary conditions.

    \begin{proposition}\label{proposition.gr.other.curvature.conditions}
        Consider $\metaf_\ast\in \{\metaf_{R=0, H\ge 0}, \metaf_{R\ge 0, H=0}, \metaf_{R\ge 0, H\ge 0}\}$, and a continuous map $g: S^\ell\to \metaf_\ast$. Then $g$ extends to a homotopy $g: S^\ell\times [0,1]\to \metaf_\ast$, such that each $g(\xi, 1)\in \metaf_{R=0, H=0}$. Consequently, the spaces $\metaf_{R=0, H\ge 0}, \metaf_{R\ge 0, H=0}, \metaf_{R\ge 0, H\ge 0}$ are weakly contractible.
    \end{proposition}
    
    \begin{proof}
      Fix a space of metrics $\metaf_\ast$ as above. For $\xi\in S^\ell$, since $R_{g(\xi)}\ge 0$, $H_{g(\xi)}\ge 0$, the operator
      \[\left(\Delta_{g(\xi)}-\tfrac18 R_{g(\xi)}, \partial_V  +\tfrac 14 H_{g(\xi)}\right): W^{k,p}_\beta(X_\infty)\to W^{k-2,p}_{\beta-2}(X_\infty)\times W^{k-1-\tfrac 1p, p}(\partial X_\infty)\]
      is an isomorphism. We take the unique solution $u(\xi)$ to
      \[\left(\Delta_{g(\xi)}-\tfrac18 R_{g(\xi)}, \partial_V  +\tfrac 14 H_{g(\xi)}\right)(u)=(\tfrac 18 R_{g(\xi)}, -\tfrac14 H_{g(\xi)}).\]
      By Proposition \ref{proposition.gr.max.principle} and a standard application of the strong maximum principle one obtains that $-1<u<0$. Then the mapping $g: S^\ell\times [0,1]$, defined by $g(\xi, \mu)= (1+\mu u(\xi))^4 g(\xi)$, is the desired homotopy.
    \end{proof}
    
    To finish the proof of Theorem \ref{thm.gr.all.R.H}, we show that the spaces $\metaf_{R=0,H=0}$, $\metaf_{R\ge 0, H=0}$, $\metaf_{R=0, H\ge 0}$, $\metaf_{R\ge 0, H\ge 0}$ are ANRs. The proof here is closely related to that of Theorem \ref{thm:ANRstructure}.
    
    \begin{proposition}
        For $\metaf_\ast\in \{\metaf_{R=0,H=0}, \metaf_{R\ge 0, H=0}, \metaf_{R=0, H\ge 0}, \metaf_{R\ge 0, H\ge 0}\}$, there exists an open subset  $\mU^{AF}\subset \metaf$ containing $\metaf_\ast$ and a retraction $F: \mU^{AF}\to \metaf_\ast$. Consequently, $\metaf_\ast$ is an ANR.
    \end{proposition}
    
    \begin{proof}
        We formulate the proof for $\metaf_\ast=\metaf_{R=0, H=0}$, as the argument for the other three spaces of metrics is very similar. For each $g\in \metaf_{R=0,H=0}$, the operator
        \[\left(\Delta_{g} - \tfrac18 R_{g}, \partial_V+\tfrac14 H_g\right)=(\Delta_g, \partial_V): W^{k,p}_\beta(X_\infty) \to W^{k-2,p}_\beta(X_\infty)\times W^{k-1-\tfrac 1p, p}(\partial X_\infty)\]
        is an isomorphism. Thus, there exists $\eps(g)>0$ such that when $\tilde g\in \metaf$, $\|\tilde g-g\|_{k,p,\beta}<\eps(g)$, the operator
        \[\left(\Delta_{\tilde g} - \tfrac18 R_{\tilde g}, \partial_V+\tfrac14 H_{\tilde g}\right): W^{k,p}_\beta(X_\infty) \to W^{k-2,p}_\beta(X_\infty)\times W^{k-1-\tfrac 1p, p}(\partial X_\infty)\]
        is also an isomorphism. Moreover, observe that $u=0$ is the unique solution to $-\Delta_g u =0$, $\partial_V u =0$. Thus, by shrinking $\eps(g)$ if necessary, the unique solution $u(\tilde g)$ to
        \begin{equation}\label{eq:MasterANRaf}\left(\Delta_{\tilde g} - \tfrac18 R_{\tilde g}, \partial_V+\tfrac14 H_{\tilde g}\right)u(\tilde g)=\left(\tfrac 18 R_{\tilde g}, -\tfrac14 H_{\tilde g}\right)
        \end{equation}
        satisfies $|u(\tilde g)|<\tfrac12$.
        
        We define an open neighborhood $U(g)$ as:
        \[U(g)=\left\{\tilde g\in \metaf : \|\tilde g-g\|_{k,p,\beta}<\eps(g)\right\},\]
        and then set
        \[\mU^{AF} = \bigcup_{g\in \metaf_\ast} U(g).\]
        We then define the map $F: \mU^{AF}\to \metaf_\ast$ as
        \[F(\tilde g)=\left(1+u(\tilde g)\right)^4 \tilde g.\]
        We verify that $F: \mU^{AF}\to \metaf$ has the properties we want. First, by definition $\mU^{AF}$ is open, $\metaf_\ast\subset \mU^{AF}$, and $F$ is a continuous map between the spaces. Also, when $\tilde g\in \metaf_\ast$, $R_{\tilde g}=0$ and $H_{\tilde g}=0$, and thus $u(\tilde g)=0$, hence $F(\tilde g)=\tilde g$. Finally, $R_{F(\tilde g)}=0$ and $H_{F(\tilde g)}=0$ for all $\tilde g\in \mU^{AF}$.
        
        When dealing with the other three spaces it suffices to replace in the right-hand side of \eqref{eq:MasterANRaf} the pair $\left(\tfrac 18 R_{\tilde g}, -\tfrac14 H_{\tilde g}\right)$ by means of 
        \[
        \left(\tfrac 18 (R_{\tilde g})_{-}, -\tfrac14 H_{\tilde g}\right), \ \left(\tfrac 18 R_{\tilde g}, -\tfrac14 (H_{\tilde g})_{-}\right), \ \left(\tfrac 18 (R_{\tilde g})_{-}, -\tfrac14 (H_{\tilde g})_{-}\right)
        \]
        respectively. The structure of the argument is unaffected, we omit the details.
    \end{proof}

    Lastly, we get to studying the space of asymptotically flat initial data sets on $X_\infty$, namely, the class of all pairs $(g,h)$ such that $g\in \metaf$ and $h$ is a symmetric $(0,2)$-tensor with $h_{ij}\in W^{k-1,p}_{\beta-1}$. The vacuum constraint equations on $X_\infty$ read as follows:
    \begin{equation*}
        \begin{cases}
            R_g + (\tr_g h)^2  - |h|_g^2 = 0,\\
        \div_g\left( h - (\tr_g h)g\right)=0.
        \end{cases}
    \end{equation*}

    On $\partial X_\infty$, the so-called convergence for the outgoing family of null geodesics is given by
    \[\theta_+=\tr_g h - h(V,V) - H_g.\]
    We pose the condition that $\partial X_\infty$ is marginally outer trapped, i.e. that $\theta_+=0$. 
    
    We shall be concerned with \textit{maximal} initial data sets, which by definition corresponds to requiring $\tr_g h =0$. We further pose the additional (rather standard) condition that $h(V,V)\le 0$, and note that, under this condition, $H_g\ge 0$. To sum up, we define the space of maximal solutions to the vacuum constraint equation with marginally outer trapped boundary as
    \begin{multline*}
        \metaf_{T}=\left\{(g,h): g\in \metaf, h \text{ is a symmetric }(0,2) \text{ tensor} \ \text{satisfying} \ h\in W^{k-1,p}_{\beta-1},\right. \\ 
        \left. \vphantom{W^{k,p}_\beta} R_g = |h|_g^2 , \ \tr_g h=0, \ \div_g h = 0, \text{ and }H_g= -h(V,V)\ge 0\right\}.
    \end{multline*}
    
    In order to prove that $\metaf_T$ is contractible, we first show that all of its homotopy groups actually vanish.
    
    \begin{proposition}\label{proposition.gr.metaf.weakly.contractible}
        For each $\ell\ge 0$, $\pi_\ell(\metaf_T)=0$.
    \end{proposition}

 \begin{proof}
        We exploit a nice trick by Marques (see \cite[Section 9]{Mar12}) and consider the following space of asymptotically flat data:
        \begin{multline*}
       \tilde \met^{AF}_T=\left\{(g,h): g\in \metaf, h \text{ is a symmetric }(0,2) \text{ tensor} \ \text{satisfying} \ h\in W^{k-1,p}_{\beta-1},\right. \\ 
        \left. \vphantom{W^{k,p}_\beta} R_g \geq |h|_g^2 , \ \tr_g h=0, \ \div_g h = 0, \text{ and }H_g \geq -h(V,V)\ge 0\right\}.
    \end{multline*}
        Observe that $\tilde \met^{AF}_T$ is weakly contractible. Indeed, any continuous map $(g,h): S^\ell\to \tilde\met^{AF}_T$ extends to a homotopy $(g,h): S^\ell \times [0,1]\to \tilde \met^{AF}_T$, defined by
        \[g(\xi,\mu)=g(\xi),\quad h(\xi,\mu)= (1-\mu)h(\xi).\]
        Note that this map is well defined, and $(g,h)(\cdot, 1)$ is a continuous map $S^\ell \to \metaf_{R\ge 0, H\ge 0}$, which is weakly contractible by Proposition \ref{proposition.gr.other.curvature.conditions}. 
        
        We will now construct a retraction $F: \tilde \met^{AF}_T\to \metaf_T$, that is, a continuous map $F$ such that $F=id$ on $\metaf_T$. Once we have this, take $i:\metaf_T\hookrightarrow \tilde \met^{AF}_T$ the inclusion map and consider the composition $F\circ i: \metaf_T\xhookrightarrow{i}\tilde \met^{AF}_T \xrightarrow{F} \metaf_T$. Since $F\circ i=id$, for each $\ell\ge 0$, $F_*: \pi_\ell(\tilde \met^{AF}_T)\to \pi_\ell (\metaf_T)$ is surjective. But $\pi_\ell(\tilde \met^{AF}_T)=0$, so $\pi_\ell(\metaf_T)=0$.
        
        For $(g,h)\in \tilde \met^{AF}_T$, we will construct $u\in 1+ W^{k,p}_\beta(X_\infty)$, $u>0$, such that $(\hat g,\hat h)=(u^4 g, u^{-2}h)\in \metaf_T$ and we will then define $F(g,h)=(\hat g,\hat h)$. Note that for any $u>0$, $\tr_{\hat g}\hat h=0$ and $\div_{\hat g}\hat h=0$. Under the conformal change, $R_{\hat g}= u^{-5}(-8\Delta_g u + R_g u)$, $H_{\hat g}=u^{-3}(4\partial_V u+ H_g u)$. Moreover, $|\hat h|_{\hat g}^2 = u^{-12} |h|_g^2$. If $V$ is the $g$-unit outward-pointing normal vector, then $u^{-2} V$ is the $\hat g$-unit outward-pointing normal vector, and $\hat h(u^{-2}V, u^{-2}V)=u^{-6} h(V,V)$. Therefore, $(\hat g,\hat h)\in \metaf_T$ if and only if $u$ solves the Lichnerowicz equation with oblique boundary condition
        \begin{equation}\label{eq.gr.lichnerowicz}
            \begin{cases}
                -\Delta_g u + \frac18 R_g u - \frac18 |h|_g^2 u^{-7}=0 \quad \text{ in }X_\infty,\\
                \partial_V u + \frac14 H_g u +\frac14 h(V,V) u^{-3}=0 \quad \text { on }\partial X_\infty.
            \end{cases}
        \end{equation}
        We solve the equation for $\phi=u-1\in W^{k,p}_\beta(X_\infty)$ using the method of sub- and supersolutions in the form given in \cite[Section 3, Proposition 2]{Maxwell2005solutions}. Since $R_g\ge |h|_g^2$ and $H_g\ge -h(V,V)$, $\phi_+=0$ is a supersolution. For a subsolution, we apply Proposition \ref{proposition.weighted.sobolev.fredholm} and take the unique solution $\phi_{-}\in W^{k,p}_\beta$ of
        \[\left( \Delta_g - \tfrac18 R_g, \partial_V +\tfrac14 H_g\right) \phi_{-} = \left(\tfrac18 R_g, -\tfrac14 H_g  \right).\]
        By the maximum principle, Proposition \ref{proposition.gr.max.principle} (and, again, the strong maximum principle), applied to both $\phi_{-}$ and $1+\phi_{-}$, $-1< \phi_{-}\le 0$. We have:
        \begin{align*}
            &-\Delta_g \phi_{-} + \frac18 R_g (1+\phi_{-}) - \frac18 |h|_g^2 (1+\phi_{-})^{-7}=-\frac18 |h|_g^2 (1+\phi_{-})^{-7}\le 0 \quad \text{ in }X_\infty,\\
                &\partial_V \phi_{-} + \frac14 H_g (1+\phi_{-}) +\frac14 h(V,V) (1+\phi_{-})^{-3}=\frac14 h(V,V)(1+\phi_{-})^{-3}\le 0 \quad \text { on }\partial X_\infty.
        \end{align*}
        Thus there exists a solution $u\in 1+W^{k,p}_\beta(X_\infty)$, $0< u\le 1$. Moreover, this solution is unique, for if $u_1,u_2\in 1+W^{k,p}_\beta(X_\infty)$ are two such solutions to \eqref{eq.gr.lichnerowicz}, then $u_1-u_2$ solves an equation in the form
        \[\left( - \Delta_g + \Lambda_1, \partial_V +\Lambda_2 \right)(u_1-u_2)=0,\]
        for  $\Lambda_1, \Lambda_2\ge 0$. Therefore $u_1=u_2$ by Proposition \ref{proposition.weighted.sobolev.fredholm}. Moreover, if $(g,h)\in \metaf_T$, then $u=1$ and hence $F(g,h)=(g,h)$.
           \end{proof}

    We next prove that, under the stronger condition that $k\ge 2$ and $k>\frac {3}{p}+1$, the space $\metaf_T$ is an ANR.
    
    \begin{theorem}\label{thm.gr.maximal.solution.is.anr}
        Let $(X_\infty,g)$ be an asymptotically flat manifold of class $W^{k,p}_\beta$, $k\ge 2$, $k>\frac 3p +1$, and $\beta\in (-1,0)$. Then the space of maximal solutions to the vacuum constraint equations with marginally outer trapped boundary $\metaf_T$ is an ANR.
    \end{theorem}
    
    The proof of Theorem \ref{thm.gr.maximal.solution.is.anr} is divided into several steps. We first introduce some convenient notations and collect some fundamental facts. We denote by $W^{k,p}_\beta(\sym^2 X_\infty)$, $W^{k,p}_{\beta}(\Gamma X_\infty)$ and $W^{k,p}\left(\Lambda^1 (\partial X_\infty)\right)$ the space of symmetric $(0,2)$-tensors, vector fields on $X_\infty$ and $1$-forms on $\partial X_\infty$, respectively, of $W^{k,p}_\beta$ and $W^{k,p}$ regularity.
    
    Let $K$ denote the conformal Killing operator, defined as $K Y = \tfrac 12 \mL_Y g - \tfrac 1n (\div_g Y)g$, here $\mL$ is the Lie derivative, and $Y\in W^{2,p}_\beta(\Gamma X_\infty)$. One easily checks that for any $Y$, $\tr_g KY=0$, and $K: W^{k,p}_\beta(\Gamma X_\infty)\to W^{k-1,p}_{\beta-1}(\sym^2 X_\infty)$ is a bounded operator. It is well-known that the formal $L^2$ adjoint of $K$ is $\div_g$ (composed with the musical isomorphism $^{\#}$). We denote by $\Delta_K = \div_g K = K^*K$ the vector Laplacian, and consider the associated Neumann boundary operator $B:W^{k,p}_{\beta}(X_\infty)(\Gamma X_\infty)\to W^{k-1-\tfrac 1p,p}(\Lambda^1(\partial X_\infty))$, defined as $BY= KY(V,\cdot)$. A fundamental result in this direction was proved in \cite[Section 5, Theorem 3]{Maxwell2005solutions}.
    
    \begin{theorem}\label{thm.gr.vector.laplacian.isomorphism}
        Suppose $(X_\infty,g)$ is asymptotically flat of class $W^{k,p}_\beta$, $k\ge 2$, $k>3/p + 1$, $\beta\in (-1,0)$. Then the operator
        \[(\Delta_K, B): W^{k,p}_\beta(\Gamma X_\infty)\to W^{k-2,p}_{\beta-2}(\Gamma X_\infty)\times W^{k-1-\tfrac 1p,p}\left(\Lambda^1(\partial X_\infty)\right)\]
        is an isomorphism.
    \end{theorem}
    
    \begin{remark}
        In general, if we just assume $k>n/p$, then $(\Delta_K, B)$ is a Fredholm operator of index $0$, see \cite[Section 5, Proposition 6]{Maxwell2005solutions}. The stronger assumption $k>n/p+1$ is used to show that $\ker K=0$, i.e. there is no non-trivial conformal Killing vector field on $(X_\infty, g)$ vanishing at infinity, see \cite[Section 6]{Maxwell2005solutions}.
    \end{remark}
    
    \begin{proposition}\label{proposition.gr.traceless.transverse.anr}
        For $k\ge 2, k>3/p+1, \beta\in (-1,0)$, the space of asymptotically flat metrics $g$ and $g$-transverse, traceless symmetric $(0,2)$-tensors
        \[\metaf_{TT} := \left\{(g,h): g\in \metaf, h\in W^{k-1,p}_{\beta-1}(\sym^2 X_\infty), \tr_g h=0, \div_g h=0\right\}\]
        is an ANR.
    \end{proposition}

 \begin{proof}
        The proof here is inspired by the classical $L^2$ decomposition by York \cite{York1974decomposition}. Denote by
        \[\metaf_0 : = \left\{(g,h): g\in \metaf, h\in W^{k-1,p}_{\beta-1}(\sym^2 X_\infty)\right\}\]
        the space of all metrics and symmetric $(0,2)$-tensors. Then $\metaf_0$ is the product (direct sum) of two Banach spaces, and hence an ANR. We will now construct a retraction $F:\metaf_0\to \metaf_{TT}$, i.e. a continuous $F$ such that $F=id$ on $\metaf_{TT}$. 
        
        For each $(g,h)\in \metaf_0$, based on Theorem \ref{thm.gr.vector.laplacian.isomorphism} let $Y\in W^{k,p}_\beta(\Gamma X_\infty)$ be the unique solution to the elliptic problem
        \[\begin{cases}\Delta_K Y = K^*(h-\tfrac  1n (\tr_g h)g),\\ BY = 0,\end{cases}\]
        and define 
        \[F(g,h)=\left(g, h- \tfrac 1n (\tr_g h)g - KY\right).\]
        We verify that $F$ has our desired properties. First, if $(g,h)\in \metaf_{TT}$, then $\tr_g h=0$ and $K^*h =(\div_g h)^{\#}=0$. Thus $Y=0$ is the only solution to $(\Delta_K , B)Y=(0,0)$. Therefore $F=id$ on $\metaf_{TT}$. 
        Then, take a pair $(g,h)$: we observe that $\tr_g (h-\tfrac 1n (\tr_g h)g-KY)=0$, and $(\div_g (h-\tfrac1n (\tr_g h)g-KY))^{\#}=K^*(h-\tfrac 1n (\tr_g h )g)-\Delta_K Y=0$, thus $F(g,h)\in \metaf_{TT}$ which completes the proof. 
    \end{proof}
    
    \begin{remark}
        The argument above gives us, as a byproduct, the following York decomposition in the asymptotically flat setting. For any traceless symmetric $(0,2)$ tensor $h_0\in W^{k-1,p}_{\beta-1}$, there is a unique decomposition
        \[h_0 = h + KY,\]
        where $h$ is a traceless, transverse symmetric $(0,2)$ tensor in $W^{k-1,p}_{\beta-1}$, $Y$ is a $W^{k,p}_\beta$ vector field with the Neumann boundary condition $BY=0$, and moreover $h\perp KY$ in $L^2(X_\infty)$.
    \end{remark}
    We are now ready to prove Theorem \ref{thm.gr.maximal.solution.is.anr}
    \begin{proof}[Proof of Theorem \ref{thm.gr.maximal.solution.is.anr}]
        We will construct an open neighborhood $\metaf_1$ of $\metaf_T$ inside $\metaf_{TT}$, and a retraction $F: \metaf_1\to \metaf_T$. Then it follows from Proposition \ref{prop:Criteria} that $\metaf_T$ is an ANR, as desired.
        
        Take an element $(g,h)\in \metaf_T$. For $u\in 1+W^{k,p}_\beta(X_\infty)$, consider the functional 
        \[G:\metaf_{TT}\times \left(1+W^{k,p}_\beta(X_\infty)\right)\to W^{k-2,p}_{\beta-2}(X_\infty)\times W^{k-1-\tfrac 1p, p}(\partial X_\infty),\]
        defined by
        \[G(\hat g,\hat h,u)=\left(-\Delta_{\hat g} u +\frac 18 R_{\hat g} u -\frac 18 |\hat h|_{\hat g}^2 u^{-7}, \partial_V u+\frac 14 H_{\hat g} u +\frac 14 \hat h(V_{\hat g},V_{\hat g})u^{-3}\right),\]
        where it is understood (as above) that $V_{\hat g}$ denotes the outward-pointing unit normal in metric $\hat{g}$.
        The Fr\'echet derivative of $G$ with respect to $u$, evaluated at $(g,h,1)$,
        \[\frac{\partial G}{\partial u}\bigg|_{\hat g=g, \hat h=h, u=1}: W^{k,p}_{\beta}(X_\infty)\to W^{k-2,p}_{\beta-2}(X_\infty)\times W^{k-1-\tfrac 1p, p}(\partial X_\infty),\]
        is given by
        \[v\mapsto \left(-\Delta_g v + \tfrac 18 R_g v + \tfrac 78 |h|_g^2 v, \partial_V v + \tfrac 14 H_g v - \tfrac 34 h(V,V) v\right).\]
        Using the condition that $R_g= |h|_g^2$ and $H_g= -h(V,V)$, we then have
        \[\frac{\partial G}{\partial u}\bigg|_{(g,h,1)}(v) = \left(-\Delta_g v + R_g v, \partial_V v + H_g v\right). \]
        Since $R_g\ge 0, H_g\ge 0$, the operator $(-\Delta_g + R_g, \partial_V + H_g)$ is an isomorphism by Proposition \ref{proposition.weighted.sobolev.fredholm}. 
        
        By the implicit function theorem, there exists $\eps(g,h)>0$ and a continuous map $u:(\hat g,\hat h)\mapsto u(\hat g,\hat h)$, defined on $(\hat g,\hat h)$ satisfying $(\hat g,\hat h)\in \metaf_{TT}$, $\|g-\hat g\|_{k,p,\beta}< \eps$, $\|h-\hat h\|_{k-1,p,\beta-1}<\eps$, such that $u(g,h)=1$, $u(\hat g,\hat h)>0$, and $u=u(\hat g,\hat h)$ is the only solution to $G(\hat g, \hat h, u)=G(g,h,1)=0$.

        Hence, define
        \[\metaf_1 := \bigcup_{(g,h)\in \metaf_T} \{(\hat g,\hat h)\in \metaf_{TT}: \|g-\hat g\|_{k,p,\beta}< \eps(g,h), \|h-\hat h\|_{k-1,p,\beta-1}<\eps(g,h)\},\]
        and the map $F: \metaf_1\to \metaf_T$ as
        \[F(\hat g,\hat h)=\left(u^4(\hat g,\hat h) \hat g, u^{-2}(\hat g,\hat h)\hat h\right).\]
        Now, by construction $F=id$ on $\metaf_T$. Also, since $\tr_{\hat g}\hat h=0$ and $\div_{\hat g}\hat h=0$, and $u(\hat g,\hat h)$ is the unique solution to \eqref{eq.gr.lichnerowicz} (that is to say: to the Lichnerowitz equation with oblique boundary conditions), it follows that $F$ is indeed a well-defined continuous map to $\metaf_T$ and so we conclude that $\metaf_T$ is an ANR as claimed.
        \end{proof}

\appendix

\section{Normal injectivity radius}\label{sec:NormalInjRad}

Here we prove the following ancillary result, that was employed in Section \ref{sec:ANR} when showing that for any compact manifold with boundary $X$ the (sub-)spaces of Riemannian metrics $\met_{R>0, H\geq 0}$ and $\met_{R>0, H=0}$ are absolute neighborhood retracts.

\begin{proposition}\label{proposition.continuity.normal.inj}
    The map $g\to i(\partial X,g)$ is continuous with respect to the $C^2$ topology on the space of Riemannian metrics.
\end{proposition}

We recall that $i(\partial X,g)$ is the notation we employ to denote the normal injectivity radius of $\partial X$ as a submanifold of $X$, with respect to the background Riemannian metric $g$.
Our proof is a simple extension of an elegant argument by Sakai \cite{Sakai1983continuity} on the continuity of the injectivity radius functions on closed Riemannian manifolds. In fact, that proof directly carries over to our setting, once we have a nice characterization of $i(\partial X, g)$ in terms of basic geometric quantities.

Suppose $p\in \partial X$, $q\in X$ and let $\gamma$ be a unit speed normal geodesic connecting $p$ and $q$, such that $\dot\gamma(0)$ is normal to $\partial X$. Recall that, by definition, $J$ is a normal Jacobi field along $\gamma$ if $J$ satisfies the Jacobi field equation $\nabla_{\dot \gamma}\nabla_{\dot \gamma} J + \text{Rm}_g (\dot\gamma, J)\dot\gamma=0$, and $\langle J,\dot\gamma\rangle=0$. We have the following lemma.

\begin{lemma}\label{lemma.charactering.normal.inj}
    Suppose $(X,g)$ is a compact Riemannian manifold with boundary. Then its normal injectivity radius $i(\partial X, g)$ equals the smallest of the following two quantities:
    \begin{enumerate}
        \item The minimum length $r$ of a unit speed geodesic $\gamma:[0,r]$ with $\gamma(0)\in \partial X$, $\dot\gamma(0)=-V$, such that there exists a non-trivial normal Jacobi field $J$ along $\gamma$ such that 
        \begin{equation}\label{eq.jacobi}
            \nabla_{\dot\gamma(0)} J(0) = -S(J(0)), \quad J(r)=0.
        \end{equation}
        Here $S$ is the shape operator on $\partial X$ with respect to $V$.
        \item half the length of a geodesic connecting two points on $\partial X$, that it is orthogonal to $\partial X$ at both such endpoints.
    \end{enumerate}
\end{lemma}

\begin{proof}
    The proof here is relatively close, in its structure, to the classical argument by Klingenberg \cite{Klingenberg1959contribution}. 
    For a point $p\in \partial X$, there exists a maximal distance $r(p)>0$ such that the distance from $\exp_p(-rV(p))$ to $\partial X$ is equal to $r$, for $r\in [0,r(p))$. It is elementary to check that $r(p)$ is a continuous function on $\partial X$. The normal injectivity radius is given by $\min\{r(p): p\in \partial X\}$. Suppose $p\in \partial X$ is such that $r(p)=i(\partial X,g)$, $q=\exp_p(-r(p)V(p))$, and $\gamma$ is a minimizing geodesic connecting $p$ to $q$. For a sequence of positive numbers $\eps_i\to 0$, let $\sigma_j$ be a unit speed minimizing geodesic from $\partial X$ to $\gamma(r(p)+\eps_j)$. Then a subsequence of $\sigma_j$, which we still denote by $\sigma_j$, converges to a geodesic $\sigma$ from $\partial X$ to $q$. We separate two cases.

    Case 1: $\gamma=\sigma$. Consider the map $F: \partial X\to X$ defined by $F(p')=\exp_{p'}(-r(p)V(p'))$. We claim that $dF$ has a non-trivial kernel at $p$. Otherwise, since $\tfrac{d}{dr}\exp_{p}(-r V(p))$ is the parallel translate of $-V(p)$ along $\gamma$ (and hence nonzero), we have that $d\tilde F$ is nonsingular at $(p,r(p))$, where $\tilde F$ is the map defined $\tilde F(p',r)=\exp_{p'}(-rV(p'))$ (here we used the fact that $d\tilde F$ orthogonally decomposes in $T_p(\partial X)$ and $\R$). Thus, there exists an open neighborhood $U$ of $(p,r(p))\in \partial X\times \R$, such that $F|_U$ is a diffeomorphism. Consider the minimizing geodesic $\sigma_i$, and denote its two endpoints by $p_i\in \partial X$ and $q_i=\gamma(r(p)+\eps_i)$. Note that, denoted by $|\sigma_i|$ the length of the curve $\sigma_i$ with respect to the metric $g$, we have  $|\sigma_i| = d_g(p_i,q_i)\ge d_g(p,q)-d_g(p, p_i)-d_g(q,q_i)$, and trivially $|\sigma_i|\le r(p)+\eps_i$. Since $p_i\to p$ and $\eps_i\to 0$, it follows that $(p_i, |\sigma_i|) \in  U$ for sufficiently large $i$. This is a contradiction, as $\tilde F(p_i, |\sigma_i|)=\tilde F(p, r(p)+\eps_i)$. 
    
    We now consider the map $f(s,t)=\exp_{\alpha(s)}(-tV(\alpha(s)))$, where $\alpha(s)$ is a unit speed curve on $\partial X$ with $\alpha(0)=p$, and such that $\dot\alpha(0)\in \ker dF$. Then $\frac{\partial f}{\partial s}(0,t)$ is a Jacobi field $J(t)$, and $J(r)=0$; also note that $\dot\alpha(0)=J(0)$. We compute 
    \[\frac{\partial^2 f}{\partial t\partial s}(0,0)=\frac{\partial^2 f}{\partial s\partial t}(0,0)=-\nabla_{\left(\dot\alpha(0)\right)}{V(0)}.\]
    This shows $\nabla_{\dot\gamma(0)}J(0)=-S(J(0))$. Therefore \eqref{eq.jacobi} holds.

    Case 2: $\gamma\ne \sigma$. We claim that the velocity vectors $\dot\gamma$ and $\dot\sigma$, at the point $q$, are lying on the same line and point to opposite directions, and hence the concatenated path $(\sigma^{-1})*\gamma$ is a geodesic orthogonal to $\partial X$ at both endpoints. Otherwise, let $\beta\in (0,\pi)$ be the angle formed by such two vectors, and let $v\in T_q X$ be the unit tangential vector lying on the plane spanned by them, and forming an angle $\tfrac{\beta}{2}$ with each of them. By the first variation of length, moving in the direction of $\tfrac{\beta}{2}$ would decrease the distance to $\partial X$, thereby contradicting the fact that $p$ has been selected as a minimizer of the function $r:\partial X\to\R$. 
\end{proof}

In particular, by standard ODE comparison theory, one derives from the previous lemma that if the sectional curvatures of $X$ and the length of the second fundamental form of $\partial X$ are bounded by $\Lambda$, then there is a lower bound $C(\Lambda)$ for the length of geodesic in case (1) of Lemma \ref{lemma.charactering.normal.inj}.

\begin{proof}[Proof of Proposition \ref{proposition.continuity.normal.inj}]
    Fix a compact smooth manifold with boundary $X$. Let $\{g_k\}_{k=1}^\infty$ be a sequence of Riemannian metrics and $g_k\to g$ in the $C^2$ topology. Let $r_k=i(\partial X, g_k)$. 
    
    We first prove $\limsup r_k\le i(\partial X,g)$. Denote by $r=\limsup r_k$ and suppose, without loss of generality, that $r_k\to r$. It suffices to prove that for any $p\in \partial X$, $d_g(p, \exp_p(-r V(p)))=r$. Let $q=\exp_p(-rV(p))$. Clearly $d_g(p,q)\le r$, so we only need to prove that $d_g(p,q)\ge r$. For each $k$, let $V_k(p)$ be the outward $g_k$-unit normal vector at $p$, and let $q_k=\exp_{p,g_k}(-r_k V_k(p))$. As $k\to \infty$, $g_k\to g, r_k\to r, V_k\to V$. Thus $q_k\to q$. By definition, $d_{g_k}(p,q_k)=r_k$, and hence $d_{g_k}(p,q)\ge d_{g_k}(p,q_k)-d_{g_k}(q,q_k)=r_k - d_{g_k}(q,q_k)$. Taking limits as $k\to \infty$, we have $d_g(p,q)\ge \limsup r_k=r$, as desired.
    
    We then prove $\liminf r_k \ge i(\partial X,g)$. Let $r=\liminf r_k$. We observe first that $r>0$. Indeed, since $g_k\to g$ in $C^2$, they satisfy a uniform upper bound on the sectional curvature and the length of the second fundamental form of the boundary. Therefore any $g_k$-normal geodesic with a Jacobi field satisfying \eqref{eq.jacobi} has a uniformly positive lower bound on its length. Also, a sequence of $g_k$-normal geodesic that is $g_k$-normal to $\partial X$ at both endpoints subsequentially converges to a $g$-normal geodesic, posing a uniform lower bound on their length. Therefore $r>0$. 
    
To proceed, we distinguish two cases depending on whether \emph{for infinitely many values of $k$} the normal injectivity radius $i(\partial X, g_k)$ is realised by geodesic segments with a Jacobi field, or instead by a geodesic that is normal at both its endpoints (as specified by Lemma \ref{lemma.charactering.normal.inj}).
In the first case, take points $p_k\in \partial X$, $q_k\in X$ such that $d_{g_k}(p_k,q_k)=r_k$; by possibly extracting a subsequence, we assume that $p_k\to p$, $q_k\to q$, and a minimizing normal geodesic $\gamma_k$ from $p_k$ to $q_k$ converges to $\gamma$. 
     For infinitely many $k$, there exists a normal Jacobi field $J_k$ on $\gamma_k$ satisfying \eqref{eq.jacobi}, thus by scaling if necessary (note that the Jacobi equation and \eqref{eq.jacobi} are scaling invariant), we assume $J_k(p_k)$ is a unit vector. Then by possibly extracting another subsequence, the vectors $J_k(0)$ (and hence $\nabla_{\dot\gamma_k(0)}J_k(0)$, by \eqref{eq.jacobi}) converge as $k\to \infty$. Using the $C^2$ convergence of metrics, we conclude that $J_k$ converges to a (non-trivial) vector field $J$ on $\gamma([0,r])$ satisfying the Jacobi equation such that \eqref{eq.jacobi} holds. This implies $i(\partial X, g)\leq r$.
    
    On the other hand, if for infinitely many $k$, there exists a $g_k$ geodesic $\gamma_k$, $g_k$-normal to $\partial X$ at its endpoints, then by possibly taking another subsequence, $\gamma_k$ converges to a $g$-geodesic $\gamma$ of length $2 \lim r_k=2r$, $g$-normal to $\partial X$. Thus, once again, $i(\partial X,g)\leq r$. This finishes the proof.
\end{proof}

\section{The notion of PSC-conformal metrics for reflexive manifolds}\label{sec:PSCconf}

  In this appendix we discuss some properties of reflexive manifolds that are PSC-conformal (in the sense of Definition \ref{def:ReflexivePSCconf}). The properties of general PSC-conformal manifolds (i.e. with no equivariance conditions) are discussed in \cite[Section 6.2]{BamKle19}. We present the necessary extensions of their arguments in order to adapt the theory to our reflexive setting. We begin with the remark that condition (1) in Definition \ref{def:ReflexivePSCconf} is equivalent to $8\Delta_g w - R_g w <0$ on $M$, (2) $\Leftrightarrow$ $w^4|_{\partial M}$ is equal to the constant sectional curvature of the induced metric on $\partial M$, and (3) $\Leftrightarrow$ $A_g = -V(w^4) g$, where $A_g$ is the second fundamental form of $\partial M$ with respect to the outward unit normal.

  \begin{lemma}\label{lemma.PSCconform.extension}
    Let $(M,g,f)$ be a compact reflexive 3-manifold with boundary, $Z\subset M$ an $f$-invariant, compact three-dimensional submanifold and $\mS$ be a reflexive spherical structure on $M$. Suppose that:
    \begin{enumerate}
        \item $g$ is compatible with $\mS$.
        \item $(Z,g,f)$ is reflexively PSC-conformal.
        \item $\overline{M\setminus Z}$ is a union of (possibly singular) fibers of $\mS$.
    \end{enumerate}
    Then $(M,g,f)$ is also reflexively PSC-conformal.
  \end{lemma}
  
  \begin{proof}
   We proceed as in \cite[Lemma 6.6]{BamKle19} with a few notable modifications. We first discuss the case where $M\setminus Z$ contains singular fibers of $\mS$ or boundary components of $M$. Let $\{\Sigma_1,\cdots,\Sigma_m\}$ be the set of all singular fibers or boundary components of $\partial M$ that are contained in $M\setminus Z$ (note, in particular, that some of these fibers may consist of a single point). By Lemma \ref{lemma.SingFiberReflexive}, singular fibers that are diffeomorphic to $\mathbb{R}\mathbb{P}^2$ appear in pairs, related by means of the map $f$. Choose pairwise disjoint open neighborhoods $V_j$ of $\Sigma_j$, such that for each $j$ the closure $\overline{V}_j$ is the union of spherical fibers, and $\overline{V}_1\sqcup\cdots \sqcup \overline{V}_m$ is $f$-invariant. If a certain $V_j$ satisfies $V_j=f(V_j)$, then $\overline{V}_j$ is diffeomorphic to $D^3$ or $S^2\times [-1,1]$, and thus is conformal, via an $f$-invariant factor, to a positive scalar curvature metric on $D^3$ which is cylindrical near the boundary, or the product metric on $S^2\times [-1,1]$. 
   
   Thus, by possibly taking $Z'=Z\cup_{j=1}^m V_j$, it suffices to consider the case where $M\setminus Z$ contains only regular spherical fibers and is disjoint from $\partial M$. Then each connected component $V$ of $\overline{M\setminus Z}$ is diffeomorphic to $S^2\times [-1,1]$. Since it follows from the definition of reflexive spherical structure, by virtue of assumption (3) above, that $M\setminus Z$ is $f$-invariant, either $\overline{V}$ is $f$-invariant, or $\overline{V}$ and $f(\overline{V})$ are disjoint and contained in $\overline{M\setminus Z}$. In the second case, one may simply take the same conformal factor $w^*$ on $V$ as in \cite[Lemma 6.6]{BamKle19} and $w^*\circ f$ on $f(V)$. In the first case, note that by taking $\eps>0$ sufficiently small, $V'$ is contained in the domain of $\mS$, and thus we can take an $f$-invariant parametrization $\phi:S^2\times (-L-\eps, L+\eps) \to V'$ of an open neighborhood $V'$ of $\overline{V}$, such that $\overline{V}=\phi(S^2\times [-L,L])$ and the metric on $V'$, when pulled-back through $\phi$, can be written as $h^4(g_{S^2}+dr^2)$ for some $f$-invariant function $h\in C^\infty(V')$. Denote by $p\in S^2, r\in (-L-\eps, L+\eps)$ the local coordinates. Since $g$ is assumed to be compatible with $\mS$, $h$ only depends on the variable $r$ on $(-L-\eps,L+\eps)$. 
   
   By Lemma \ref{lem:ReflSstructClassification}, $f$ acts by reflection across either $S^1\times (-L-\eps, L+\eps)$ for some equator $S^1\subset S^2$, or $S^2\times \{0\}$. In the latter case, the conformal factor $h$ is an even function in $r$. Let $w$ be the (pull-back of the) $f$-equivariant conformal factor on $Z$ such that the conditions in Definition \ref{def:ReflexivePSCconf} hold. In local coordinates $(p,r)$, we have $(wh)(\pm L)=1$ and $\partial_r(wh)|_{r=\pm L}=0$. Consider the function $\tilde w$ defined by
   \[\tilde w(p,r)= \begin{cases} w(p,r) \quad & r\in (-L-\eps, -L]\cup [L, L+\eps),\\ \frac{1}{h(r)} \quad & r\in [-L,L].\end{cases}\]
   Then $\tilde w$ is an $f$-equivariant $C^{1,1}$ function, $\tilde w$ is smooth from both sides to $\phi(S^2\times \{-L\})$ and $\phi(S^2\times \{L\})$, and satisfies $8\Delta_{\phi^{
   \ast}g} \tilde w - R_{\phi^{\ast}g}\tilde w<0$ whenever it is smooth. We can thus take a smoothing $w^*$ of $\tilde w$, defined on $\phi(S^2\times (-L-\eps, L+\eps))$ that exactly coincides with $w$ near $\pm (L+\eps)$, such that $8 \Delta_{\phi^{
   \ast}g} w^* - R_{\phi^{
   \ast}g}w^*<0$ holds on $Z$. 
  \end{proof}

  We next establish a Lemma that extends \cite[Lemma 6.7]{BamKle19} in the reflexive setting. 
  
  \begin{lemma}\label{lemma.PSCconform.criter}
    There are constants $\eps,c\in (0,1)$ such that the following holds. Suppose $(M,g,f)$ is a reflexive 3-manifold, and let $\rho$ be the curvature scale defined by \eqref{eq:rhoCurvScaleSpecification} (cf. \eqref{eq:rhoCurvScale}). Assume, for some $r>0$, the existence of a reflexive spherical structure $\mS$ such that $\{\rho<r\}\subset \text{domain} (\mS)$, $g$ is compatible with it, and, in addition, one is given $Z\subset M$ a compact 3-dimensional $f$-invariant submanifold such that:
    \begin{enumerate}
        \item $\partial Z$ is a union of regular fibers.
        \item $(Z,g)$ has positive scalar curvature.
        \item Every point of $Z\cap\{\rho<r\}$ satisfies the $\eps$-canonical neighborhood property. 
        \item $\rho\le cr$ on $\partial Z$.
    \end{enumerate}
    Then $Z$ is reflexively PSC-conformal.
  \end{lemma}
  
  \begin{proof}
   We closely follow the proof of \cite[Lemma 6.7]{BamKle19}. Without loss of generality assume $Z$ is connected, $\partial Z\ne \emptyset$ and $Z$ is not the union of spherical fibers (for else one would simply appeal to Lemma \ref{lemma.PSCconform.extension}). Fix $\lambda\in (\sqrt c/2, \sqrt c)$ such that $\rho\ne \lambda r$ on any singular spherical fiber. Define $Z_0=\{\rho\ge \lambda r\}\cap Z$, and take $Z_1$ to be the union of $Z_0$ and all components of $Z\setminus Z_0$ disjoint from $\partial Z$. By Lemma \ref{lemma.PSCconform.extension}, it suffices to prove that $(Z_1,g)$ is reflexively PSC-conformal. Denote by $\Sigma$ a connected component of $\partial Z_1$. Then $\Sigma$ is diffeomorphic to $S^2$, and $\rho=\lambda r<r$ on $\Sigma$. 
   
   Through the limit argument presented in \cite[Claim 6.8]{BamKle19}, by taking $\eps$ and $c$ sufficiently small, every point on $\Sigma$ has a canonical neighborhood that is $\eps$ close to the round cylinder in $C^{[1/\eps]}$ topology after rescaling by $\rho^{-1}$. In other words, for any point $x\in \Sigma$, $\rho^{-1}(\Sigma-x)$ has an open neighborhood $U$ that is $\eps$ close to $S^2\times (-1/\eps,1/\eps)$ in the $C^{[1/\eps]}$ topology. We work with this rescaled manifold. As in the proof of Lemma \ref{lemma.PSCconform.extension}, we fix an $f$-invariant parametrization $\phi: S^2\times (-1/\eps,1/\eps)\to U$ such that $\Sigma=\phi(S^2\times \{0\})$ (which determines local coordinates $p\in S^2, r\in (-1/\eps, 1/\eps)$) and, since $g$ is compatible with $\mS$, $g=h(r)^4(g_{S^2}+dr^2)$ for $h\in C^\infty(-1/\eps, 1/\eps)$, $\|h-1\|_{C^{[1/\eps]}(-1/\eps, 1/\eps)}<\eps$. By a standard interpolation argument, there exists a function $u\in C^\infty(-1/\eps,1/\eps)$ such that:
   \begin{itemize}
       \item $(uh)(0)=1$, $(uh)'(0)=0$.
       \item $u=1$ in $(-1/\eps, -10]\cup [10,1/\eps)$.
       \item $\|u-1\|_{C^2(-1/\eps,1/\eps)}<2\eps$.
   \end{itemize}
   Moreover, $u$ can also be constructed as $f$-invariant, thanks to the classification of reflexive necks. We therefore conclude that the manifold $(Z_1,u^4 g)$ has positive scalar curvature, $\Sigma$ is totally geodesic, and $\Sigma$ is the unit round sphere with the induced metric.
  \end{proof}

  We proceed to observe that the construction in \cite[Lemma 6.9]{BamKle19} works in the reflexive setting.
  
  \begin{lemma}\label{lemma.PSCconform.modify}
    Let $(M,g,f)$ be a reflexively PSC-conformal 3-manifold and let $p$ be a point in the interior of $M$ (i.e. $p\in M\setminus\partial M$). Then there exist a constant $a$ and a positive function $w\in C^\infty(M)$ satisfying:
    \begin{enumerate}
        \item $w\circ f=w$ on $M$,
        \item $8\Delta_g w-R_g w<0$ on $M$,
        \item $2w^4=R_g$ on $\partial M$,
        \item $A_g=-V(w^4)g$ on $\partial M$,
    \end{enumerate}
 such that $w=w(p)-a\cdot d^2(p,\cdot)$ in an open neighborhood of $p$.
  \end{lemma}
  
  In other words, we can modify a conformal factor (subject to the requirements of Definition \ref{def:ReflexivePSCconf}) so to make it `in normal form' near an interior point.
  
  \begin{proof}
   Let $w$ be the reflexive conformal factor as given in Definition \ref{def:ReflexivePSCconf}. We distinguish two cases. If $p\ne \Fix(f)$, let $q=f(p)$. We work with disjoint $f$-related coordinate neighborhoods $U_p$ of $p$ and $U_q$ of $q$. The equivariance assumption implies $df(\nabla w(p))=\nabla w(q)$, so either $\nabla w(p)=\nabla w(q)=0$ or, if not, we claim one can design a preliminary modification of the conformal factor so that this additional condition is also accommodated. 
   Indeed, note that 
  the geodesics $\gamma_p$ and $\gamma_q$, emanating from $p$ and $q$ with initial velocity $\nabla w(p)/|\nabla w(p)|$ and $\nabla w(q)/|\nabla w(q)|$, respectively, are also $f$-related in the sense that (by standard ODE uniqueness arguments) $\gamma_q=f\circ \gamma_p$. As a result, appealing to the exact construction as in \cite[Claim 6.11]{BamKle19}, we obtain a new $f$-equivariant conformal factor $\bar w$ such that $\nabla \bar w(p)=\nabla \bar w(q)=0$. Then the same local interpolation argument, applied near $p$ and $q$, further deforms $\bar w$ in a $f$-equivariant manner to $\tilde w$ so that $\tilde w= \tilde w(p)-a\cdot d^2(p,\cdot)$ in $U_p$ and $\tilde w= \tilde w(q)-a \cdot d^2(q,\cdot)$ in $U_q$.
   
   If $p\in \Fix(f)$, then $\nabla w(p)\in T_p (\Fix(f))$ as $w$ is $f$-equivariant ($w\circ f=w$). Thus the same argument as in \cite{BamKle19} \emph{automatically} produces an $f$-equivariant conformal factor satisfying the conclusion of Lemma \ref{lemma.PSCconform.modify}.
    \end{proof}
  
  Finally, by applying Lemma \ref{lemma.PSCconform.extension}, Lemma \ref{lemma.PSCconform.modify} and the argument in \cite[Lemma 6.12]{BamKle19} verbatim in an equivariant fashion, we get the following.
  
  \begin{lemma}\label{lemma.PSCconform.excision}
    Let $M$ be a compact 3-manifold with boundary, and let $(g^s)_{s\in K}$ be a continuous family of Riemannian metrics on $M$, parametrised by a compact topological space $K$. Assume further that $(M,g^s,f)$ is reflexively PSC-conformal for all $s\in K$ for some smooth map $f:M\to M$.
    
    Suppose we have:
    \begin{itemize}
        \item \emph{either} a continuous family of $f$-equivariant embeddings $\left(\mu_s: B^3(1)\to M\right)_{s\in K}$ such that for every $s\in K$, $\mu_s^* g^s$ is compatible with the standard spherical structure on $B^3(1)$;
        \item \emph{or} a continuous family of $f$-equivariant embeddings $\left(\mu_s: B^3(1)\sqcup B^3(1)\to M\right)_{s\in K}$ such that the images of the two copies of $B^3(1)$ are disjoint, and on each of them $\mu_s^* g^s$ is compatible with the standard spherical structure on $B^3(1)$.
    \end{itemize}
    Then there is a constant $r_0\in (0,1)$ such that for all $0<r\le r_0$, $s\in K$, the Riemannian manifold $\left(M\setminus \mu_s(B^3(r)), g^s,f\right)$ (or $\left(M\setminus \mu_s(B^3(r)\sqcup B^3(r)), g^s,f\right)$, respectively) is reflexively PSC-conformal.
  \end{lemma}
  
  In order to avoid ambiguities, let us specify what we mean when we write that the given embeddings are $f$-equivariant: in the first case, we require that $f\circ \mu_s = \mu_s\circ f_0$; in the second case, we require instead that $f\circ \mu_s=\mu_s\circ\sigma$ where $\sigma: B^3(1)\sqcup B^3(1)\to  B^3(1)\sqcup B^3(1)$ is the involution that maps each point in either copy of $B^3(1)$ to the corresponding point in the other copy.
  
  This last result will be crucially employed in Appendix \ref{app:EquivDiskRemov}, when implementing the equivariant removal of (three-dimensional) disks from a partial homotopy, which in turn is an essentially ingredient for the backward-in-time induction scheme that is employed in the proof of the main theorem (more specifically, for Proposition \ref{prop:SurjPartHom}).

\section{The equivariant extension of a partial homotopy}\label{app:EquivExtHom}

The scope of the present appendix is to provide a suitable, equivariant extension result for partial homotopies. In order to get there, we first need to give two ancillary results concerning the equivariant extension of families of metrics parametrised, say, over a cylinder $D^k\times I$ where the metrics in question are assigned over a `parabolic' portion of the boundary, namely for $(D^k\times \left\{0\right\})\cup (\partial D^k\times I)$. The first such ancillary statement reads as follows:

\begin{lemma}\label{lem:ExtensionFamilyMet}
Let $Z$ be a compact 3-manifold with boundary, and let $f\in C^{\infty}(Z,Z)$ denote an
    involutive diffeomorphism satisfying the assumption $(\star_{\text{sep}})$. Let then $Y$ denote a connected, compact 3-dimensional submanifold with boundary of $Z$, which we assume to $f$-invariant. Suppose that there exists reflexive spherical structures $(\mS^s)_{s\in D^k}$, such that for all $s\in D^k$ we have that $\mS^s$ is defined on an open subset $U^s$ containing $Y$, and $\partial Y$ consists of regular fibers.
    Consider three continuous family of metrics 
    \[
    (g^1_s)_{s\in D^k} \ \text{on} \ Z, \ \ \  (g^2_{s,t})_{s\in D^k, t\in I} \ \text{on} \ \overline{Z\setminus Y}, \ \ \ (g^3_{s,t})_{s\in \partial D^k, t\in I} \ \text{on} \ Z
    \]
    such that the following assertions hold true:
    \begin{enumerate}
        \item $(Z, g^1_s, f), (\overline{Z\setminus Y}, g^2_{s,t}, f), (Z, g^3_{s,t}, f)$ are reflexive 3-manifolds for all values of $s\in D^k$ (or $s\in \partial D^k$) and for any $t\in I$ in the second and third case;
        \item $g^1_s, g^2_{s,t}, g^{3}_{s,t}$ are compatible with $\mS^s$, on the appropriate domains of definition, for all values of $s\in D^k$ (or $s\in \partial D^k$) and for any $t\in I$ in the second and third case;
        \item any pair of mutual consistency condition is satisfied, namely
        \[
        g^1_s=g^2_{s,0} \ \text{on} \ \overline{Z\setminus Y} \ \text{for all} \ s\in D^k, 
        \]
        \[
        g^1_s=g^3_{s,0} \ \text{on} \ Z \ \text{for all} \ s\in \partial D^k,
        \]
        \[
        g^2_{s,t}=g^3_{s,t} \ \text{on} \ \overline{Z\setminus Y} \ \text{for all} \ s\in \partial D^k, t\in I.
        \]
    \end{enumerate}
    Then there exists a continuous family of metrics $(h_{s,t})_{s\in D^k, t\in I}$ on $Z$, such that the following assertions hold true:
    \begin{enumerate}[label=(\roman*)]
        \item $(Z, h_{s,t}, f)$ is a reflexive 3-manifold for all values of $s\in D^k, t\in I$; 
        \item $h_{s,t}$ is compatible with $\mS^s$ for all values of $s\in D^k$, and for any $t\in I$;
        \item $h_{s,t}$ extends the assigned families of metrics in the sense that
         \[
        h_{s,0}=g^1_s \ \text{on} \ Z \ \text{for all} \ s\in D^k, 
        \]
        \[
        h_{s,t}=g^2_{s,t} \ \text{on} \ \overline{Z\setminus Y} \ \text{for all} \ s\in D^k, t\in I,
        \]
        \[
        h_{s,t}=g^3_{s,t} \ \text{on} \ Z \ \text{for all} \ s\in \partial D^k, t\in I.
        \]
    \end{enumerate}
    \end{lemma}

\begin{proof}
First of all, note that since $Y$ is assumed to be connected (and $\Fix(f)$ is postulated to be separating, as specified by $(\star_{\text{sep}})$) it must be $\Fix(f)\cap Y\neq\emptyset$.
That being said, thanks to the preliminary reductions presented in the proof of Proposition 6.15 in \cite{BamKle19}, and keeping in mind the classification result given in Lemma \ref{lem:ReflSstructClassification}, one only needs to consider the following two cases:

\emph{Case 1:}
$Z=S^2\times [-2,2], Y=S^2\times [-1,1] $ and $f(x,r)=(x,-r)$, \emph{or}

\emph{Case 2:} $Z=S^2\times [-2,2], Y=S^2\times [-1,1] $ and $f(x,r)=(f_0(x),r)$.

 Now, based on the normal form of compatible metrics (as per Lemma \ref{lem:NormalForm}), the whole construction reduces to a suitable extension procedure for functions, possibly satisfying additional symmetry conditions. We will now proceed with the discussion of these aspects, following the two cases above.

In the first case (corresponding to item (1) in Lemma \ref{lem:NormalForm}) the problem reduces to proving that, given $\phi_{s,t}\in C^{\infty}([-2,-1]\cup [1,2])$, where $s\in D^k$ and $t\in I$, satisfying $\phi_{s,0}=0$ on $[-2,2]$ for all $s\in D^k$ and being even (respectively: odd) for all $s\in D^k, t\in I$ then there exists $\psi_{s,t}\in C^{\infty}([-2,2])$, where $s\in D^k$ and $t\in I$, satisfying $\psi_{s,0}=\phi_{s,0}=0$ on $[-2,2]$ for all $s\in D^k$ and being even (respectively: odd) for all $s\in D^k, t\in I$. To that aim, we use Seeley's smooth extension theorem and a cutoff function to first extend each $\phi_{s,t}$ to a smooth function on $[-2,0]$ supported in $[-2,-1/2]$ and then perform an even (respectively: odd) reflection. This ensures that (employing the notation of Lemma \ref{lem:NormalForm}) the smooth extended functions $a,b$ are even, while $c_1, c_2, c_3$ are odd.

In the second case (corresponding to item (2) in Lemma \ref{lem:NormalForm}) we just need to ensure that the coefficients $c_3$ are extended to be zero for all values of the parameters $(s,t)\in D^k\times I$ and otherwise follow, for $a,b,c_1, c_2$, the non-equivariant extension procedure as no additional parity requirements are present. 
\end{proof}

We shall now proceed with a second extension result for families of metrics, which captures a different type of extension problem where one needs to add novel connected components.

\begin{lemma}\label{lem:ExtToNewConnComp}
Let $M=S^3$, and let $f\in C^{\infty}(M,M)$ denote an
    involutive diffeomorphism satisfying the assumption $(\star_{\text{sep}})$.
   Let $g_0$ denote a metric of constant sectional curvature equal to 1, and let $\Delta^n$ denote the standard $n$-dimensional simplex for some non-negative integer $n$. Let us further assume we are given the following additional data:
    \begin{itemize}
       \item a continuous positive function $\lambda:\Delta^n\to \R$ and a continuous family of metrics $(k_{s,t})_{s\in\partial\Delta^n,t\in I}$ on $M$ such that $k_{s,0}=\lambda^2(s)g_0$ for every $s\in\partial\Delta^n$, and such that $(M,k_{s,t},f)$ is a reflexive 3-manifold for every $s\in\partial\Delta^n$ and $t\in I$,
        \item an open, non-empty subset $A\subset\Delta^n$, a closed subset $E\subset\partial\Delta^n$ contained in $A$,
        \item a transversely continuous family of reflexive spherical structures $(\mS^s)_{s\in A}$ on $M$,
         \end{itemize}
    and that they satisfy the following assumptions:
    \begin{enumerate}
        \item for all $s\in A$ the metric $g_0$ is compatible with $\mS^s$,
        \item for all $s\in E$ and $t\in I$ the metric $k_{s,t}$ is compatible with $\mS^s$,
        \item for all $s\in \partial \Delta^n\setminus E$ and $t\in I$ the metric $k_{s,t}$ is an $(s,t)$-dependent multiple of $g_0$.
    \end{enumerate}
Then there is a continuous family of metrics $(h_{s,t})_{s\in\Delta^n,t\in I}$ on $M$  such that $(M,h_{s,t},f)$ is a reflexive 3-manifold for every $s\in\Delta^n$ and $t\in I$, and a closed subset $E'$ of $A\subset \Delta^n$ such that:
\begin{enumerate}[label=(\roman*)]
    \item $h_{s,0}=\lambda^2(s)g_0$ for all $s\in \Delta^n$ and $h_{s,t}=k_{s,t}$ for all $s\in\partial\Delta^n$ and $t\in I$,
    \item  for all $s\in A$ and $t\in I$ the metric $h_{s,t}$ is compatible with $\mS^s$, and for all $s\in\Delta^n\setminus E'$ and $t\in I$ the metric $h_{s,t}$ is an $(s,t)$-dependent multiple of $g_0$.
\end{enumerate}
\end{lemma}

Before we outline the proof, let us add a couple of comments. First,
the reason why we can \emph{a priori} reduce to $M=S^3$ (so neglecting the other spherical space forms) has to do with the more restrictive form of Lemma \ref{lem:ReflSstructClassification} compared to the non-equivariant case. Second, keeping in mind the content of Lemma \ref{lemma.SingFiberReflexivePointCase}, for any given $s\in \Delta^n$ only two possible pictures can occur: either $\Fix(f)$ is the central leaf of the spherical structure $\mS^s$ or, instead, $\Fix(f)$ passes through the two singular fibers of $\mS^s$ (and, for any metric as in the statement above, meets each regular fiber orthogonally along a circle).

\begin{proof}
Following the argument in the first paragraph of the proof of Proposition 6.18 in \cite{BamKle19}, based on suitable simplicial refinements, one only needs to handle the cases when $E=\emptyset$ or $A=\Delta^n$.

If $E=\emptyset$ we take $E'=\emptyset$ and note that, by assumption, for all $(s,t)\in \partial\Delta^n\times I$ the metric $k_{s,t}$ is a multiple of $g_0$, say $k_{s,t}=\mu^2(s,t)g_0$ for some continuous function $\mu$ hence we can just take $h_{s,t}:\Delta^n\times I\to\mathbb{R}$ to be defined through a \emph{positive} continuous extension to $\Delta^n\times I$ of the (mutually compatible) conformal factors $\lambda$ and $\mu$. Let us still denote by $\mu\in C^0(\Delta^n\times I)$ such an extension. Now, it follows from our hypotheses that on the one hand $(M,g_0,f)$ is a reflexive 3-manifold and so $(M,\mu^2(s,t)g_0, f)$ will also be, while on the other hand $g_0$ is compatible with $\mS^s$ for any $s\in A$ and thus the same conclusion will equally hold for $\mu^2(s,t)g_0$ for any choice of $s$ and $t$.

If instead $A=\Delta^n$ we take $E'=\Delta^n$ and design an extension such that $h_{s,t}$ is compatible with $\mS^s$ for any $s\in \Delta^n$ and $t\in I$. This task is similar in spirit to what we did above in the proof of Lemma \ref{lem:ExtensionFamilyMet}, albeit strictly simpler since we only have data assigned on the `parabolic boundary' but no need to respect a family of metrics assigned on a subset of the ambient manifold. In short, since the pairs $(D^n\times I, (D^n\times\left\{0\right\})\cup (\partial D^n\times I))$ and $(D^n\times I, D^n\times\left\{0\right\})$ are patently homeomorphic we can reduce our initial problem to that of extending a given family of metrics $(k_{s,0})_{s\in\Delta^n}$, for which we postulate that $k_{s,0}$ is compatible with $\mS^s$ for any $s\in \Delta^n$: to do so, we simply let $h_{s,t}=k_{s,0}$ and all conclusions follow.
\end{proof}

We can now finally proceed with the aforementioned extension theorem for partial homotopies.
 
\begin{theorem}\label{thm:ExtPartHom}
Consider a reflexive partial homotopy 
\begin{equation}\label{eq:InitialPartialHom}
        \left\{(Z^{\sigma}, (g^{\sigma}_{s,t})_{s\in\sigma, t\in [0,1]}, f^{\sigma}, (\psi^{\sigma}_s)_{s\in\sigma})\right\}_{\sigma\in \mK}
          \end{equation}    
at time $T\geq 0$ for a family of reflexive singular Ricci flows $(\mM^s)_{s\in K}$. Fix some simplex $\sigma\in\mK$ and assume we are assigned the following objects:
\begin{itemize}
    \item a compact 3-manifold with boundary $\hat{Z}^{\sigma}$,
    \item an involutive diffeomorphism $\hat{f}^{\sigma}$ of $\hat{Z}^{\sigma}$ satisfying $(\star_{\text{sep}})$;
    \item an embedding $\iota^{\sigma}:Z^{\sigma}\to \hat{Z}^{\sigma}$,
    \item a continuous family of embeddings $(\hat{\psi}^{\sigma}_s: \hat{Z}^{\sigma}\to\mM^s_T)_{s\in\sigma}$
\end{itemize}
such that all these conditions are met
\begin{enumerate}
    \item $\iota^{\sigma}(Z^{\sigma})$ is $\hat{f}^{\sigma} $-invariant and $\iota^{\sigma}\circ f^{\sigma}=\hat{f}^{\sigma}\circ \iota^{\sigma}$,
    \item $\hat{\psi}^{\sigma}_s(\hat{Z}^{\sigma})$ is $f^s_T$-invariant and $\hat{\psi}^{\sigma}_s\circ \hat{f}^{\sigma}=f^s_{T}\circ \hat{\psi}^{\sigma}_s$, 
    \item $\psi^{\sigma}_s = \hat{\psi}^{\sigma}_s \circ\iota^{\sigma}$,
    \item for any simplex $\tau\subset \partial\sigma$ and $s\in\tau$ we have $\hat{\psi}^{\sigma}_s(\hat{Z}^{\sigma})\subset \psi^{\tau}_s(Z^{\tau})$,
    \item the closure $Y$ of any connected component of $\hat{Z}^{\sigma}\setminus \iota^{\sigma}(Z^{\sigma})$ satisfies either of the following two properties, uniformly in $s\in\sigma$:
    \begin{enumerate}
        \item $\hat{\psi}^{\sigma}_s(Y)$ is a union of (possibly singular) fibers of $\mS^s$,
        \item $\partial Y=\emptyset$, $\hat{\psi}^{\sigma}_s(Y)\subset U^s_{S3}$ and $(\hat{\psi}^{\sigma}_s)^{\ast}(g'^{,s}_T)$ is an $s$-dependent multiple of the same constant curvature metric, and if $Y\cap \Fix(\hat{f}^{\sigma})\neq\emptyset$ then $(Y, (\hat{\psi}^{\sigma}_s)^{\ast}(g'^{,s}_T), \hat{f}^{\sigma})$ is a reflexive 3-manifold.
    \end{enumerate}
\end{enumerate}
Then there exists a continuous family of metrics $(\hat{g}^{\sigma}_{s\in\sigma, t\in [0,1]})$ such that 
\begin{equation}\label{eq:FinalPartialHom}
        \left\{(Z^{\sigma'}, (g^{\sigma'}_{s,t})_{s\in\sigma', t\in [0,1]}, f^{\sigma'}, (\psi^{\sigma'}_s)_{s\in\sigma'}\right\}_{\sigma'\in \mK,\sigma'\neq\sigma} \cup \left\{(\hat{Z}^{\sigma}, (\hat{g}^{\sigma}_{s,t})_{s\in\sigma, t\in [0,1]}, \hat{f}^{\sigma}, (\hat{\psi}^{\sigma}_s)_{s\in\sigma}\right\}
          \end{equation} 
is a reflexive partial homotopy at time $T$ for $(\mM^s)_{s\in K}$, and if \eqref{eq:InitialPartialHom} is PSC-conformal over some $s\in K$ then so will be \eqref{eq:FinalPartialHom}.
\end{theorem}

To avoid confusion, let us recall here that, in the previous statement, $K$ denotes a geometric realisation of the simplicial complex $\mK$.

\begin{proof}
Let us consider one of the finitely many connected components, say $Y$, of $\hat{Z}^{\sigma}\setminus \iota^{\sigma}(Z^{\sigma})$. Then, there are two different cases to be considered. In the first case $Y$ is disjoint from $\Fix(\hat{f}^{\sigma})$, so one can find a `twin' connected component $Y'$ of $\hat{Z}^{\sigma}\setminus \iota^{\sigma}(Z^{\sigma})$ such that $\hat{f}^{\sigma}_{|Y}: Y\to Y'$ is a diffeomorphism of compact 3-manifolds with boundary. Thus, one can appeal to the non-equivariant construction, as per Proposition 7.9 in \cite{BamKle19}, to define the continuous family of metrics $(\hat{g}^{\sigma}_{s\in\sigma, t\in [0,1]})$ on $Y$, and then define such a family on $Y'$ by push-forward through $\hat{f}^{\sigma}_{|Y}$.
In the second case $Y$ is \emph{not} disjoint from $\Fix(\hat{f}^{\sigma})$, instead. Without loss of generality we can further assume $\hat{Z}^{\sigma}\setminus \iota^{\sigma}(Z^{\sigma})$ to only consist of such a connected component. To handle this case,  we first introduce the continuous family of metrics $(k_{s,t})_{s\in\partial\sigma, t\in I}$ defined on $\hat{Z}^{\sigma}$ by letting (for $s\in\tau$ where $\tau$ is \emph{any} simplex included in $\partial\sigma$)
\[
k_{s,t}=((\psi^{\tau}_s)^{-1}\circ \hat{\psi}^{\sigma}_s)^{\ast}g^{\tau}_{s,t},
\]
which is easily checked to be a well-posed definition by virtue of the axioms defining partial homotopies. Hence, the initial task one is left with is to define a reflexive metric deformation $(\hat{g}^{\sigma}_{s\in\sigma, t\in [0,1]})$ on $Y$ satisfying the requirements that
\[
\hat{g}^{\sigma}_{s,0}=(\hat{\psi}^{\sigma}_s)^{\ast}(g'^{,s}_T)  \ \text{on} \ \hat{Z}^{\sigma}, \ \text{for all} \ s\in\sigma,
\]
\[
\hat{g}^{\sigma}_{s,t}=k_{s,t} \ \text{on} \ \hat{Z}^{\sigma}, \ \text{for all} \ s\in\partial\sigma, t\in I,
\]
\[
\hat{g}^{\sigma}_{s,t}=g^{\sigma}_{s,t} \ \text{on} \ Z^{\sigma}, \ \text{for all} \ s\in\sigma, t\in I.
\]
At this stage, if condition (5)(a) holds uniformly for all $s\in\sigma$ then we perform the construcion employing Lemma \ref{lem:ExtensionFamilyMet}, while if condition (5)(b) holds uniformly for all $s\in\sigma$ then we appeal to Lemma \ref{lem:ExtToNewConnComp} instead. The fact that such an extension, besides determining a continuous family of metrics satisfying the three equations above, is actually a reflexive partial homotopy can be verified as in pages 83-84 of \cite{BamKle19}. The only point we need to stress is that, in checking that:
\begin{itemize}
    \item $(\hat{Z}^{\sigma}, (\hat{g}^{\sigma}_{s,t})_{s\in\sigma, t\in I})$ is a reflexive metric deformation, and that 
    \item if \eqref{eq:InitialPartialHom} is PSC-conformal over some $s\in K$ then so will be \eqref{eq:FinalPartialHom},
\end{itemize}
one will have to employ Lemma \ref{lemma.PSCconform.extension} in lieu of Lemma 6.6 of \cite{BamKle19}. 
\end{proof}

\section{The equivariant removal of 3-disks of a partial homotopy}\label{app:EquivDiskRemov}

This last appendix is devoted to properly designing (the equivariant counterpart of) another operation on partial homotopies, i.e. the removal of disks. From a geometric standpoint, the disks in question are centered at the tips of regions that are close to Bryant solitons and are endowed with a spherical structure (whose only singular fiber is the tip itself). Keeping in mind the content of Lemma \ref{lemma.SingFiberReflexivePointCase},  comparing our task to the non-equivariant treatment we first need the following statement.

\begin{lemma}\label{lem:UseConformalMap}
Let $K$ be a compact topological space, and let $K'$ be a closed subset thereof. Assume we are given a continuous family of metrics $(k_s)_{s\in K}$ on $B^3$ such that $(B^3,k_s, f_0)$ is a reflexive 3-manifold for any $s\in K$ and, in addition, satisfying the property that there exists a continuous family of smooth, positive functions $(w_s)_{s\in K'}$ on $B^3$ so that $w_s\circ f_0=w_s$ and $w^4_s k_s$ has positive scalar curvature. Given any $\overline{r}_1 \in (0,1)$ there exists a continuous family of metrics $(h_{s,t})_{s\in K, t\in I}$ and $r_1\in (0,\overline{r}_1)$ so that the following assertions hold true:
  \begin{enumerate}
      \item $(B^3,h_{s,t}, f_0)$ is a reflexive 3-manifold for all $s\in K$ and $t\in I$, 
      \item $h_{s,0}=k_s$ for all $s\in K$,
      \item $h_{s,1}$ is compatible with the standard spherical structure on $D^3(r_1)$ for all $s\in K$,
      \item $h_{s,1}=k_s$ on $B^3\setminus B^3(\overline{r}_1)$,
      \item if $k_s$ is compatible with the standard spherical structure of $D^3(r)$ for some $r\in [\overline{r}_1,1)$ then so is $h_{s,t}$, for all $t\in I$,
      \item if $s\in K'$ then there exists a continuous family of smooth, positive functions $(v_{s,t})_{s\in K', t\in I}$ on $B^3$ so that $v_{s,t}\circ f_0=v_{s,t}$, $v_{s,t}=w_s$ on $B^3\setminus B^3(\overline{r}_1)$ and $v^4_{s,t}h_{s,t}$ has positive scalar curvature for all $s\in K', t\in I$.
  \end{enumerate}
 \end{lemma}
 

\begin{proof}
One can follow the same construction as in Section 6.6 of \cite{BamKle19}, based on the use of the \emph{conformal} exponential map. If $g$ is a Riemannian metric on $B^3$ such that $(B^3, g, f_0)$ is a reflexive 3-manifold, then the conformal factor $\phi$ constructed by integrating a second-order ODE along curves as explained at page 69 therein will patently satisfy the symmetry requirement that $\phi\circ f_0=\phi$ and thus at each step of the construction the extra discrete symmetry imposed by $f_0$ will also be respected. The only modifications that are needed to respect the equivariance requirement, specifically for the construction justifying Claim 6.29, are replacing the set $S_3$ consisting of $3\times 3$ real symmetric matrices by the subset $\tilde{S}_3$ of those having a block form

\[
\begin{pmatrix}\!
\begin{array}{c|c}
A_1 & \begin{matrix} 0 \\ 0 \end{matrix} \\\hline
\begin{matrix} 0 & 0 \end{matrix} & a_2
\end{array}\!
\end{pmatrix}
\]
for some $A_1\in S_2$ (where $S_2$ is analogously defined as the set of $2\times 2$ real symmetric matrices) and  $a_2\in\mathbb{R}$,
and the open cone $S^{+}_3$ of the positive definite matrices in $S_3$ by the corresponding open cone $\tilde{S}^{+}_3:=\tilde{S}_3\cap S^{+}_3$. Other than that, the argument goes through unaffected.
\end{proof}

We can then proceed with the proof of the main result of this appendix:

\begin{theorem}\label{thm:DiskRemPartHom}
Consider a reflexive partial homotopy 
\begin{equation}\label{eq:InputPartHomDiskRem}
        \left\{(Z^{\sigma}, (g^{\sigma}_{s,t})_{s\in\sigma, t\in [0,1]}, f^{\sigma}, (\psi^{\sigma}_s)_{s\in\sigma})\right\}_{\sigma\in \mK}
          \end{equation}    
at time $T\geq 0$ for a family of reflexive singular Ricci flows $(\mM^s)_{s\in K}$. Let then $K'$ be a closed subset of $K$.
Let further $\sigma\in\mathcal{K}$ be a fixed simplex and let 
\[
\left\{(\nu_{s,j}: D^3\to \mM^s_T)_{s\in\sigma}\right\}_{1\leq j\leq m_{\sigma}}
\]
be $m_{\sigma}\in\mathbb{N}_{\ast}=\left\{1,2,\ldots\right\}$ continuous families of embeddings such that the following statements hold true:
\begin{enumerate}
    \item  there exists $p_{\sigma}, q_{\sigma} \in\mathbb{N}$ with $2p_{\sigma}+q_{\sigma}=m_{\sigma}$ and:
    \begin{enumerate}
    \item if $j\leq 2p_{\sigma}$ then $\nu^{\sigma}_{s,j}(D^3)\cap \Fix(f^s)=\emptyset$ and $\nu^{\sigma}_{s,j+p_{\sigma}}=f^{s}_T\circ \nu^{\sigma}_{s,j+p_{\sigma}}$ for all $s\in\sigma$; 
    \item if $2p_{\sigma}+1\leq j\leq m_{\sigma}$ then $\nu^{\sigma}_{s,j}(D^3)\cap \Fix(f^s)\neq\emptyset$ and
     $\nu^{\sigma}_{s,j}\circ f_0 = f^s \circ \nu^{\sigma}_{s,j}$;
    \end{enumerate}
    \item for all $s\in\sigma$ and $j=1,\ldots, m_{\sigma}$ the pull-back through the embedding $\nu_{s,j}$ of $\mS^s$ coincides with the standard spherical structure of $D^3$,
    \item for all $s\in\sigma$ the images of these embeddings are pairwise disjoint, satisfy $\nu_{s,j}(D^3)\subset \psi^{\sigma}_s(\text{Int}(Z^{\sigma}))\cap U^s_{S2}$, and in addition for all $1\leq j\leq m_{\sigma}$ $\nu_{s,j}(D^3)\cap\psi^{\tau}_s(Z^{\tau})=\emptyset$ whenever $s\in\sigma\subsetneq \tau \in\mK$,
    \item if $\sigma$ is a maximal simplex of $\mathcal{K}$ then for any $\tau\in\mK, s\in\tau\cap\sigma$ and $j=1,\ldots, m_{\sigma}$ the image $\nu_{s,j}(D^3)$ does not contain an entire connected component of $\psi^{\tau}_s(Z^{\tau})$,
    \item the partial homotopy \eqref{eq:InputPartHomDiskRem} is reflexively PSC-conformal over every $s\in K'$ and for all $s\in\sigma \cap K'$ the Riemannian manifold
    \[
    (\psi^{\sigma}_s(Z^{\sigma})\setminus\cup_{j=1}^{m_{\sigma}}(\nu_{s,j}(B^3)),g'^{,s}_T)
    \]
    is reflexively PSC-conformal.
\end{enumerate}
Then there exists a reflexive partial homotopy
\begin{equation}\label{eq:OutputPartHomDiskRem}
        \left\{(\tilde{Z}^{\sigma}, (\tilde{g}^{\sigma}_{s,t})_{s\in\sigma, t\in [0,1]}, \tilde{f}^{\sigma}, (\tilde{\psi}^{\sigma}_s)_{s\in\sigma})\right\}_{\sigma\in \mK}
          \end{equation}    
at time $T\geq 0$ for $(\mM^s)_{s\in K}$ such that the following statements hold:
\begin{enumerate}[label=(\roman*)]
    \item if $\tau\in\mK, \tau\neq\sigma$, one has $\tilde{\psi}^{\tau}_s(\tilde{Z}^{\tau})=\psi^{\tau}_s(Z^{\tau})$ for all $s\in \tau$,
    \item $\tilde{\psi}^{\sigma}_s(\tilde{Z}^{\sigma})=\psi^{\sigma}_s(Z^{\sigma})\setminus\cup_{j=1}^{m_{\sigma}}\nu_{s,j}(B^3)$ for all $s\in\sigma$,
    \item the partial homotopy \eqref{eq:OutputPartHomDiskRem} is reflexively PSC-conformal over every $s\in K'$.
\end{enumerate}
\end{theorem}

\begin{remark}
Actually, we remark that the assumption that \emph{for all $s\in\sigma$ the images of these embeddings are pairwise disjoint} (also keeping in mind that \emph{for all $s\in\sigma$ and $j=1,\ldots, m_{\sigma}$ the pull-back through the embedding $\nu_{s,j}$ of $\mS^s$ coincides with the standard spherical structure of $D^3$}) implies that if $\nu_{s,i}(\left\{0\right\})\notin \Fix(f^s_T)$ (geometrically: the tip of the Bryant soliton does not belong to the fixed locus of the involution in question) then $\nu_{s,i}(D^3)\cap\Fix(f^s_T)=\emptyset$.

If instead $\nu_{s,i}(\left\{0\right\})\in \Fix(f^s_T)$ then the local picture is uniquely determined by Lemma \ref{lemma.SingFiberReflexivePointCase}. We wish to stress that such a characterisation does not allow a `mixed scenario' to occur: a standard connectedness argument shows that it cannot happen that for some values of $s\in\sigma$ one has $\nu_{s,i}(D^3)\cap\Fix(f^s_T)=\emptyset$ while for others $\nu_{s,i}(D^3)\cap\Fix(f^s_T)\neq\emptyset$ instead.
\end{remark}

\begin{proof}
First of all, we note how the proof of Proposition 7.21 in \cite{BamKle19} starts by reducing the problem to the case when $m_{\sigma}=1$ i.e. when there is only one disk to be removed. Due to our additional equivariance constraints, we need some care and have to perform certain preliminary adjustments, as we are about to describe. 

Based on the previous remark, cf. assumption (1) above, we will rather distinguish between case (1)(a) and case (1)(b).

In case (1)(a), we can reduce the discussion to the situation when \emph{a pair} of disks have to be removed. So, we will have two families of embeddings 
\[
(\nu_{s,1}: D^3\to \mM^s_T)_{s\in\sigma}, \ \ (\nu_{s,2}: D^3\to \mM^s_T)_{s\in\sigma}
\]
satisfying the equation $\nu_{s,2}=f^{s}_T\circ \nu_{s,1}$ for all $s\in\sigma$. We first focus our attention on one of these two families of embeddings (say: the first) and then construct:
\begin{enumerate}
\item an open neighborhood $U$ of (the support of) $\sigma$ in $K$, 
    \item a continuous family of \emph{extended} embeddings $(\overline{\nu}_{s,1}: D^3(2)\to \mM^s_T)_{s\in U}$ satisfying $(\overline{\nu}_{s,1})_{|D^3(1)}=\nu_{s,1}$ for all $s\in\sigma$ and whose images are still disjoint from $\Fix(f^s_T)$,
    \item an embedding $\mu_1: D^3(2)\to Z^{\sigma}$ satisfying $\overline{\nu}_s=\psi^{\sigma}_s\circ \mu_1$ for all $s\in\sigma$,
\end{enumerate}
such that all additional requirements listed at page 86 of \cite{BamKle19} are satisfied. Once this is accomplished, we simply set 
\[
\overline{\nu}_{s,2}:=f^{s}_T\circ \overline{\nu}_{s,1}, \ \text{for all} \ s\in U, \ \text{and} \
\mu_2 := f^{\sigma}\circ \mu_1
\]
and note that all such requirements will also by satisfied by (the same set) $U, (\overline{\nu}_{s,2}: D^3(2)\to \mM^s_T)_{s\in U}, \mu_2$.
At this stage, the problem is, so to say, \emph{decoupled}: since the whole construction by Bamler-Kleiner happens locally near the image of $\overline{\nu}$ (thus away from the symmetry locus) we follow the proof of Proposition 7.21 in \cite{BamKle19}, by taking care of  $\nu_{s,1}$ and then repeating the whole procedure for  $\nu_{s,2}$ using $(\overline{\nu}_{s,2})_{s\in U}$ in lieu of $(\overline{\nu}_{s,1})_{s\in U}$, $\mu_2$ in lieu of $\mu_1$ and using the involutions $f^{s}_T$ and $f^{\sigma}$ to replicate each operation on the other side of the symmetry locus.

Hence, our task is instead to handle the case when $m_{\sigma}=1$ and we have a single familiy of embeddings, henceforth denoted $(\nu_s: D^3\to \mM^s_T)_{s\in\sigma}$ whose image intersects the the symmetry locus, i.e.
\[
\nu_s(D^3)\cap \Fix(f^s_T)=\emptyset, \ \text{for all} \ s\in \sigma.
\]
Here, we will have to \emph{adapt} the construction by Bamler-Kleiner to ensure equivariance at each step. Note that the reduction to the case of a single disk to be removed relies on the use of Lemma \ref{lemma.PSCconform.extension} instead of Lemma 6.6 of \cite{BamKle19}.

We first design the continuous family of \emph{extended} embeddings $(\overline{\nu}_{s}: D^3(2)\to \mM^s_T)_{s\in U}$ satisfying $(\overline{\nu}_{s})_{|D^3(1)}=\nu_{s}$ as well as $\overline{\nu}_{s}=f^{s}_T\circ \overline{\nu}_{s}$ for all $s\in U$, where $U$ is a suitably small open neighborhood containing the support of $\sigma$. This is accomplished using the exponential map based at $p:=\nu_s(0)$: we can choose frames (hence coordinates) on $T_p \mM^s_T$, continuously depending on $s$, so that $f_0(\text{exp}^{-1}_p(\nu_s(x))=\nu_s(f_0(x))$. Hence, we can use these coordinates, and the exponential map, to define $\overline{\nu}_s$, which can be done accomodating properties (1),(2),(3),(4) at page 86 of \cite{BamKle19}. For the sake of notational simplicity we will from now onwards simply write $\nu_s$ in lieu of $\overline{\nu}_s$, for any $s\in U$.

Second, we design a map $\mu: D^3(2)\to Z^{\sigma}$ such that $f^{\sigma}\circ \mu = \mu\circ f_0$ and $\nu_s=\psi^{\sigma}_s\circ \mu$. This follows from the fact that, due to our equivariance assumptions, one can find a continuous family of diffeomorphisms $(\chi_s: Z^{\sigma}\to Z^{\sigma})_{s\in\sigma}$ satisfying $\chi_s \circ f^{\sigma}= f^{\sigma}\circ \chi_s$, that equal the identity near the boundary $\partial Z^{\sigma}$ and such that $\chi^{-1}_s\circ (\psi^{\sigma}_s)^{-1}\circ\nu_s$ is constant in $s$ (at any point).

Said that, exactly as in \cite{BamKle19}, one can reduce the problem to the case when:
\begin{itemize}
\item[] if $\tau_1, \tau_2 \in\mathcal{K}, \tau_1\subset \tau_2$ and $s\in \tau_1\cap U$ and $\mC$ is the closure if a component of $Z^{\tau_1}\setminus ((\psi^{\tau_1}_s)^{-1}\circ \psi^{\tau_2}_s)(Z^{\tau_2})$ with $\psi^{\tau_1}_s(\mC)\cap\nu_s(D^3(2))\neq\emptyset$, then Case (4)(a) of Definition \ref{def:PartialHom} applies,
\end{itemize}
we distinguish two cases, depending on whether the simplex $\sigma$ is maximal or not.

If $\sigma$ is not maximal, then one can simply set
\begin{equation}\label{eq:DiskRemDomMap}
    \tilde{Z}^{\sigma}:=Z^{\sigma}\setminus \mu (\text{Int}(D^3(1))), \ \ \tilde{f}^{\sigma}:=(f^{\sigma})_{|\tilde{Z}^{\sigma}} , \ \ \tilde{\psi}^{\sigma}_s:=(\psi^{\sigma}_s)_{|\tilde{Z}^{\sigma}},
\end{equation}
\begin{equation}\label{eq:DiskRemMetr}
    \tilde{g}^{\sigma}_{s,t}:=(g^{\sigma}_{s,t})_{|\tilde{Z}^{\sigma}}
\end{equation}
and otherwise, for all simplices $\tau\neq\sigma$
\begin{equation}\label{eq:DiskRemOthers}
    (\tilde{Z}^{\tau}, (\tilde{g}^{\tau}_{s,t})_{s\in\tau, t\in I} ,\tilde{f}_{\tau},(\tilde{\psi}^{\tau}_s)_{s\in\tau})=(Z^{\tau}, (g^{\tau}_{s,t})_{s\in\tau, t\in I} ,f_{\tau},(\psi^{\tau}_s)_{s\in\tau}).
\end{equation}    
The fact that such a definition determines a partial homotopy with the desired properties is checked exactly as at page 87-88 of \cite{BamKle19}, with the only \emph{caveat} of employing Lemma \ref{lemma.PSCconform.extension} in lieu of Lemma 6.6 therein.

In the maximal case, we will keep the definitions of $\tilde{Z}^{\tau}, \tilde{f}^{\tau}, \tilde{\psi}^{\tau}_s$ as above (for all $\tau\in\mK$, including the case $\tau=\sigma$), but need to re-design a continuous family of metrics $(\tilde{g}^{\tau}_{s\in\tau, t\in I})$ on $\tilde{Z}^{\tau}$ that defines (together with such other data) a partial homotopy at time $T$ for the given singular Ricci flow and associated $\mR$-structures, under the additional requirement that $(\tilde{Z}^{\tau}, (\tilde{g}^{\tau}_{s,t})_{s\in\tau, t\in I} ,\tilde{f}_{\tau},(\tilde{\psi}^{\tau}_s)_{s\in\tau})$ is reflexively PSC-conformal over every $s\in K'$. This does not only affect the data associated to the simplex $\sigma$, but also the nearby ones.

In following the argument of \cite{BamKle19} we note that Claim 7.24 remains true as it is stated, while Claim 7.25 should be slightly modified by additionally requiring the extended metrics $(h_{s,t})_{s\in U, t\in I}$ to be equivariant in the sense that $(f_{0})^{\ast}h_{s,t}=h_{s,t}$ for all $s\in U$ and $t\in I$. Its proof proceeds again by induction on the dimension of the simplices of $K$, and in the inductive step we appeal to Lemma \ref{lem:ExtensionFamilyMet} in lieu of Proposition 6.15 (in fact, when the intersection of $B^3(2)$ with $\nu_s(\psi^{\tau_j}_s(Z^{\tau_j}))$ is an annulus of width less than $1/100$ we can take any smooth extension compatible with the standard spherical structure and the equivariance constraint, which are both taken into account in the normal form given in Case (2) of Lemma \ref{lem:NormalForm}). 
For Claim 7.26 we again need to require the modified (rounded) metrics $(h'_{s,t})_{s\in U, t\in I}$ to be equivariant in the sense that $(f_{0})^{\ast}h'_{s,t}=h'_{s,t}$ for all $s\in U$ and $t\in I$ and, item (g), $(\psi^{\sigma}_s(Z^{\sigma}), k^{\sigma}_{s,t})$ to be reflexively PSC-conformal; in the proof one follows the argument of \cite{BamKle19} verbatim, with the fundamental change of appealing to Lemma \ref{lem:UseConformalMap} in lieu of Proposition 6.26 therein. We then construct the family of `inflationary' diffeomorphisms $\Phi: D^3(1.99)\times (0,1]\to D^3(1.99)$ subject to the additional requirement that $\Phi_u(x)=\varphi(x,u)x$ for some smooth function $\varphi$ satisfying $\varphi(x,u)=\varphi(f_0(x),u)$ for all $x\in D^3(1.99), u\in (0,1]$ and set
$h''_{s,t}=\Phi^{\ast}_{1-(1-r_2)\delta_1(s)\delta_2(t)}h'_{s,t}$ where the cutoff functions $\delta_1, \delta_2$ and the constant $r_2$ are chosen exactly as in \cite{BamKle19}. At the very end of the proof of Claim 7.29 we appeal to Lemma \ref{lemma.PSCconform.excision} in lieu of Lemma 6.12 therein.
Lastly, we declare
\[
\tilde{Z}^{\tau}:=
\begin{cases}
Z^{\tau}\setminus \mu (\text{Int}(D^3(1))) & \text{if} \ \tau=\sigma \\
Z^{\tau} & \text{otherwise},
\end{cases}
\ \ \tilde{f}^{\tau}=(f^{\tau})_{|\tilde{Z}^{\sigma}} , \ \ \tilde{\psi}^{\tau}_s:=(\psi^{\tau}_s)_{|\tilde{Z}^{\tau}},
\]
and we complete the definition of the partial homotopy by letting (for any $\tau\in\mK, s\in\tau, t\in I$)
\[
\tilde{g}^{\tau}_{s,t}:=\begin{cases}
g^{\tau}_{s,t} & \text{on} \ Z^{\tau}\setminus(\psi^{\sigma}_s)^{-1}(\nu_s(D^3(1.99))) \hspace{2mm} \text{if} \ s\in U \\
g^{\tau}_{s,t} & \text{on} \ Z^{\tau}  \hspace{44mm} \text{if} \ s\notin U \\
(\psi^{\sigma}_s)^{\ast}(\nu_s)_{\ast}h''_{s,t} & \text{on} \ (\psi^{\sigma}_s)^{-1}(\nu_s(D^3(1.99))) \hspace{11mm} \text{if} \ s\in U.
\end{cases}
\]
Thereby, the tuple $\left\{\tilde{Z}^{\tau}, (\tilde{g}^{\tau}_{s,t})_{s\in\tau, t\in I} ,\tilde{f}_{\tau},(\tilde{\psi}^{\tau}_s)_{s\in\tau})\right\}_{\tau\in\mK}$ defines a reflexive a partial homotopy with the desired properties. The fact that the equivariance constraints are satisfied follows from the constructions, and all other properties are checked exactly as in the non-equivariant case. 
\end{proof}

\bibliography{biblio}

\end{document}